\documentclass[12pt,a4paper,oneside]{article}

\usepackage[T1]{fontenc}
\usepackage[utf8]{inputenc}
\usepackage{lmodern}

\usepackage{amsmath,amssymb,theorem}
\newif\ifprint\printtrue

\usepackage{mathtools}

\ifprint
\usepackage[left=2.9cm,right=2.9cm,top=2cm,bottom=2cm,includefoot]{geometry}
\else
\usepackage[left=5.5cm,right=0.3cm,top=2cm,bottom=2cm,includefoot]{geometry}
\fi

\usepackage[yyyymmdd,hhmmss]{datetime}

\usepackage{enumerate}
\usepackage[nottoc]{tocbibind}

\usepackage{hyperref}
\usepackage{nameref,zref-xr}
\zxrsetup{toltxlabel}

\usepackage{cite}

\usepackage{epsfig}
\usepackage{epic}
\usepackage{eepic}
\usepackage{graphics}   
\usepackage{graphicx}
\usepackage{color}
\usepackage{longtable}
\usepackage{subfigure}

\usepackage{caption}
\newtheorem{theorem}{Theorem}
\newtheorem{corollary}[theorem]{Corollary}

\newtheorem{lemma}[theorem]{Lemma}
{\theorembodyfont{\rmfamily}\newtheorem{example}[theorem]{Example}}
\newtheorem{proposition}[theorem]{Proposition}
{\theorembodyfont{\rmfamily}}

\newenvironment{proof}{\noindent\textit{Proof.\ }}{\hspace*{\fill}$\Box$\medskip}

\renewcommand{\geq}{\geqslant}
\renewcommand{\leq}{\leqslant}
\newcommand{\tsum}{\textstyle\sum\limits}
\newcommand{\tprod}{\textstyle\prod\limits}
\newcommand{\tint}{\textstyle\int\limits}

\DeclareFontFamily{U}{mathx}{\hyphenchar\font45}
\DeclareFontShape{U}{mathx}{m}{n}{
      <5> <6> <7> <8> <9> <10>
      <10.95> <12> <14.4> <17.28> <20.74> <24.88>
      mathx10
      }{}
\DeclareSymbolFont{mathx}{U}{mathx}{m}{n}
\DeclareMathSymbol{\bigtimes}{1}{mathx}{"91}

\newcommand{\vect}[1]{\underline{\boldsymbol{#1}}}

\usepackage{soul}

\makeatletter
\newcommand{\raisemath}[1]{\mathpalette{\raisem@th{#1}}}
\newcommand{\raisem@th}[3]{\raisebox{#1}{$#2#3$}}
\makeatother

\newcommand{\dchi}{{\displaystyle\chi}}
\newcommand{\vectgamma}{\text{\setul{1.3pt}{.4pt}\ul{$\boldsymbol{\gamma}$}}}
\newcommand{\indexvectgamma}{\text{\setul{0.9pt}{.4pt}\ul{$\boldsymbol{\gamma}$}}}
\newcommand{\vecteta}{\text{\setul{1.3pt}{.4pt}\ul{$\boldsymbol{\eta}$}}}
\newcommand{\indexvecteta}{\text{\setul{0.9pt}{.4pt}\ul{$\boldsymbol{\eta}$}}}

\DeclareMathOperator{\pop}{p}
\newcommand{\p}{\!\rule{0.4pt}{0pt}\pop}
\DeclareMathOperator{\lcm}{lcm}

\DeclareMathOperator{\vspan}{span}
\DeclareMathOperator{\tmodop}{mod}

\newcommand{\tmod}{\,\tmodop}

\numberwithin{equation}{section}
\numberwithin{theorem}{section}
\numberwithin{figure}{section}

\DeclareMathOperator{\LCop}{LC}
\newcommand{\LC}{\vect{\LCop}}

\DeclareMathOperator{\Iop}{I}
\newcommand{\I}{\vect{\Iop}}

\DeclareMathOperator{\Jop}{J}
\newcommand{\J}{\vect{\Jop}}

\newcommand{\vectrho}{\text{\setul{1.5pt}{.4pt}\ul{$\boldsymbol{\rho}$}}}

\newcommand{\vectj}{{\underline{\boldsymbol{j}}}}

\newcommand{\oversim}[1]{\mathrlap{\stackrel{\boldsymbol{\sim}}{\smash{#1}\rule{0pt}{1.25ex}}}#1}
\newcommand{\vectGammacircvect}[1]{%
  \overset{\raise-4.5pt\hbox{\scriptsize\ensuremath{\circ}}}%
          {\vect{\Gamma}}^{\text{\raisebox{-3.5pt}{$(\vect{#1})$}}}}
\newcommand{\vectGammacircvectast}[1]{%
  \overset{\raise-4.5pt\hbox{\scriptsize\ensuremath{\circ}}}%
          {\vect{\Gamma}}^{\text{\raisebox{-3.5pt}{$(\vect{#1}),\ast$}}}}

\renewcommand{\theequation}{\thesection.\arabic{equation}}

\allowdisplaybreaks

\usepackage{fancyhdr} \usepackage[yyyymmdd,hhmmss]{datetime}

\ifprint\else
\fi

\definecolor{mygray}{gray}{0.6}

\begin{document}

\setlength\abovedisplayskip      { 7pt }
\setlength\abovedisplayshortskip { 3pt }
\setlength\belowdisplayskip      { 7pt }
\setlength\belowdisplayshortskip { 3pt }

\title{A unifying theory for multivariate polynomial interpolation on general Lissajous-Chebyshev nodes}

\renewcommand{\thefootnote}{}

\author{\begin{tabular}{ll} Peter Dencker$^{\mathrm{1)}}$\setcounter{footnote}{-1}\footnote{$^{\mathrm{1)}}$Institut f\"ur Mathematik, 
Universit\"at zu L\"ubeck.}\quad & Wolfgang Erb$^{\mathrm{2)}}$\setcounter{footnote}{-1}\footnote{$^{\mathrm{2)}}$Department of Mathematics, University of Hawai`i at M\=anoa. }\\ 
\texttt{\normalsize dencker@math.uni-luebeck.de\quad} &  \texttt{\normalsize erb@math.hawaii.edu}
\end{tabular}}

\date{\ \\[2mm] \today}

\maketitle

\vspace*{-30em}
\begin{equation}\label{1505271126}\tag{$\uparrow$}\end{equation}

\vspace*{23em}

\thispagestyle{empty}

\enlargethispage{2em}

\setcounter{tocdepth}{2}
\tableofcontents

\thispagestyle{empty}

\newpage

\begin{abstract}
The goal of this article is to provide a general multivariate framework that synthesizes well-known non-tensorial 
polnomial interpolation schemes on the Padua points, the Morrow-Patterson-Xu points and 
the Lissajous node points into a single unified theory. The interpolation nodes of these schemes are special cases of the general 
Lissajous-Chebyshev points studied in this article. We will characterize these
Lissajous-Chebyshev points in terms of Lissajous curves and Chebyshev varieties and derive a general discrete orthogonality structure related to these points. 
This discrete orthogonality is used as the key for the proof of the uniqueness of the polynomial interpolation and the derivation of a quadrature rule on these node sets. 
Finally, we give an efficient scheme to compute the polynomial interpolants. 
\end{abstract}

\section{Introduction}

Univariate Chebyshev polynomials as well as its zeros and extremal points are among the most studied objects in interpolation and approximation theory. The reason
for this popularity lies in the large number of useful theoretical and computational properties of these polynomials. Chebyshev polynomials can easily be formulated and computed, 
they satisfy a simple orthogonality and three-term recurrence relation and they are intimately linked to trigonometric polynomials by a change of variables. Polynomial interpolation 
at the Chebyshev nodes or the Chebyshev-Gau{\ss}-Lobatto nodes is also easy to implement and can be carried out very efficiently by a discrete cosine transform.
Further, the absolute numerical condition for polynomial interpolation at these points is only growing logarithmically in the number of points.    

Many of these useful properties of the Chebyshev polynomials on the interval $[-1,1]$ can be carried over directly to the $\mathsf{d}$-dimensional hypercube $[-1,1]^{\mathsf{d}}$ by using
tensor products of univariate Chebyshev polynomials and tensor products of Chebyshev or Chebyshev-Gau{\ss}-Lobatto points as interpolation nodes. This is the standard approach
in most multivariate spectral methods and implemented, for instance, in the Chebfun 2 package \cite{TownsendTrefethen2013}.

\medskip

Non-tensorial constructions of multivariate Chebyshev node points are rare in the literature. In general, leaving the tensor-product setting causes a lot 
of difficulties in theoretical as well as in computational aspects. From the theoretical point of view it gets harder to find a suitable polynomial space in order 
that interpolation at the given node set is unique. Further, a lot of the nice computational properties of the tensor product case as for instance
the relation to the discrete cosine transform must be equilized in a more sophisticated way. An example of a non-tensorial construction on triangles with the derivation of an efficient 
computational scheme for the interpolating polynomials can be found in \cite{RylandMuntheKaas2010}. 

In the bivariate setting, there are two particular non-tensorial constructions in the literature in which the computation of the interpolating polynomial
is almost as easy as in the tensor-product case. The node points of the first construction are referred to as Morrow-Patterson-Xu points and were introduced in 
\cite{MorrowPatterson1978} for multivariate quadrature and in \cite{Xu1996} for bivariate Lagrange interpolation. The second set of nodes is referred to as Padua points and is
studied thouroghly in a series of papers \cite{BosDeMarchiVianelloXu2006,BosDeMarchiVianelloXu2007,CaliariDeMarchiVianello2005,CaliariDeMarchiVianello2008}.
In terms of numerical stability and convergence 
both interpolation schemes have similar properties as the interpolation scheme on a single tensor product set. 

A particular property of the Padua points is that they can be characterized as the boundary and the intersection points of a degenerate Lissajous curve. Based on 
this observation, the theory of the Padua points was recently extended in \cite{Erb2015,ErbKaethnerAhlborgBuzug2015,ErbKaethnerDenckerAhlborg2015} to bivariate and in 
\cite{DenckerErb2015a} to $\mathsf{d}$-variate Lissajous curves and its respective node sets. The interpolation polynomials of these $\mathsf{d}$-variate non-tensorial constructions
showed similar convergence properties as in the $\mathsf{d}$-variate tensor setting. In particular, it is shown in \cite{DenckerErbKolomoitsevLomako2016} that the Lebesgue constant, the 
absolute condition number for interpolation on the points in \cite{DenckerErb2015a} has asymptotically the same growth as the tensor product case. 

\medskip

The goal of this work is to develop a unified theory for multivariate polynomial interpolation theory that incorporates the non-tensorial schemes of the Padua points, 
the Morrow-Patterson-Xu points and the multivariate Lissajous nodes considered in \cite{DenckerErb2015a}. Further, this theory will provide exact
polynomial interpolation schemes for a series of point sets that are up to the moment only studied in 
terms of multivariate quadrature \cite{BojanovPetrova1997,Noskov91}, Chebyshev lattices \cite{CoolsPoppe2011,PoppeCools2013} or hyperinterpolation \cite{DeMarchiVianelloXu2007}. 

\medskip

The Lissajous-Chebyshev point sets $\LC^{(\vect{m})}_{\vect{\kappa}}$, the underlying sets of interpolation nodes for multivariate polynomial interpolation considered
in this work, will be introduced in Section~\ref{17008191031}. The particular name Lissajous-Chebyshev is motivated 
by the fact that the points $\LC^{(\vect{m})}_{\vect{\kappa}}$ have intimate 
relations with multivariate Chebyshev variaties and Lissajous curves. These relations will be worked out in Section~\ref{1708201701}.

\medskip

To prepare an appropriate definition of multivariate polynomial spaces for the interpolation on the node sets $\LC^{(\vect{m})}_{\vect{\kappa}}$, 
we will introduce and study  corresponding  spectral index sets in Section~\ref{17008191035}.
In Section~\ref{17008191036}, we will provide a discrete orthogonality structure related to the interpolation points which is the key element for the proof of the interpolation results
stated in Theorem \ref{1509082002} and Theorem \ref{201512131945} and the
quadrature formula on the node sets $\LC^{(\vect{m})}_{\vect{\kappa}}$ given in Theorem \ref{1509082004}.

\medskip

The main results of this work are stated as Theorem \ref{1509082002} and Theorem \ref{201512131945} in Section~\ref{1609011843}. In these two theorems we show that in 
the polynomial spaces based on the spectral index sets introduced in Section~\ref{17008191035} there exists a unique polynomial interpolant for given data values on the 
node points $\LC^{(\vect{m})}_{\vect{\kappa}}$. As a consequence of 
these results, we obtain in Section~\ref{17008191038} a simple numerical scheme for the computation of the interpolating polynomial. Finally, in Section~\ref{1609031100} we provide a series
of examples of particular Lissajous-Chebyshev nodes that link the results obtained in this paper to already well-known point sets like the Padua or the Morrow-Patterson-Xu points.

\section{General notation}
For a finite set $\mathsf{X}$, the number of elements of $\mathsf{X}$ is denoted by $\#\mathsf{X}$. For $\mathsf{x},\mathsf{y}\in\mathsf{X}$, 
we use the Kronecker delta symbol: $\delta_{\mathsf{x},\mathsf{y}}=1$ if $\mathsf{y}=\mathsf{x}$ and $\delta_{\mathsf{x},\mathsf{y}}=0$ otherwise.
Further, to each $\mathsf{x}\in\mathsf{X}$ we associate a Dirac delta function $\delta_{\mathsf{x}}$ on $\mathsf{X}$ 
defined by $\delta_{\mathsf{x}}(\mathsf{y})=\delta_{\mathsf{x},\mathsf{y}}$.
We denote by $\mathcal{L}(\mathsf{X})$ the vector space of all complex-valued functions defined on $\mathsf{X}$.
The delta functions  $\delta_{\mathsf{x}}$, $\mathsf{x}\in\mathsf{X}$, form a basis of the vector space $\mathcal{L}(\mathsf{X})$.

If $\nu$ is a measure defined on the power set $\mathcal{P}(\mathsf{X})$ of $\mathsf{X}$ and $\nu(\{\mathsf{x}\})>0$ for all $\mathsf{x}\in\mathsf{X}$, an inner product $\langle \,\cdot,\cdot\,\rangle_{\nu}$ for 
the vector space $\mathcal{L}(\mathsf{X})$ is defined by
\[\langle h_1,h_2\rangle_{\nu}=\tint h_1\overline{h}_2\,\mathrm{d}\nu.\]
 The norm corresponding to  $\langle \,\cdot,\cdot\,\rangle_{\nu}$  is denoted by  $\|\cdot\|_{\nu}$.

\medskip

For $\mathsf{d}\in\mathbb{N}$, the elements of 
$\mathbb{R}^{\mathsf{d}}$
are written as $\vect{x}=(x_1,\ldots,x_{\mathsf{d}})$. 
We use the abbreviations $\vect{0}$ and $\vect{1}$ for the $\mathsf{d}$-tuples for which all components are $0$ or $1$, respectively.

The  least common multiple in $\mathbb{N}$ of the elements of a finite set $M\subseteq \mathbb{N}$ is denoted by $\lcm M$, the greatest common divisor of the elements of a set $M\subseteq \mathbb{Z}$, $M\supsetneqq \{0\}$, is denoted by $\gcd M$. 
For $\vect{k}\in\mathbb{N}^{\mathsf{d}}$, we set 
\begin{equation}\label{1608311012} 
\p[\vect{k}]=\tprod_{\mathsf{j}=1}^{\mathsf{d}} k_{\mathsf{j}},\qquad \lcm[\vect{k}]=\lcm\{k_1,\ldots,k_{\mathsf{d}}\},
\end{equation}
and denote 

\vspace*{-2.9em}

\begin{equation}\label{1608311013} 
t^{(\vect{k})}_{l}=\dfrac{l\pi}{\lcm[\vect{k}]},\qquad l\in\mathbb{N}_0.
\end{equation}
Further, for   $k\in\mathbb{N}$, $i\in\mathbb{N}_0$, as well as  $\vect{k}\in\mathbb{N}^{\mathsf{d}}$, $\vect{i}\in\mathbb{N}_0^{\mathsf{d}}$,
we use the notation
\begin{equation}\label{201511121326}
\vect{z}^{(\vect{k})}_{\vect{i}}=\left(z^{(k_1)}_{i_1},\ldots,z^{(k_{\mathsf{d}})}_{i_{\mathsf{d}}}\right),\quad  z^{(k)}_{i}=\cos\left(i\pi/k\right), 
\end{equation}
for the so-called \textit{Chebyshev-Gauß-Lobatto} points.

The following well-known \textit{Chinese remainder theorem} will be used repeatedly in this manuscript. 
Let $\vect{k}\in \mathbb{N}^{\mathsf{d}}$ and $\vect{a}\in\mathbb{Z}^{\mathsf{d}}$. If
\begin{equation}\label{1511161925}
\forall\,\mathsf{i},\mathsf{j}\in\{1,\ldots,\mathsf{d}\}:\quad a_{\mathsf{j}}\equiv a_{\mathsf{i}}\mod \gcd\{k_{\mathsf{i}},k_{\mathsf{j}}\},
\end{equation}
then there exists an  unique $l\in\{\,0,\ldots,\lcm[\vect{k}]-1\,\}$ solving the following system of simultaneous congruences
\begin{equation}\label{1511161926}
\forall\,\mathsf{i}\in\{1,\ldots,\mathsf{d}\}:\quad l=a_{\mathsf{i}} \mod k_{\mathsf{i}}. 
\end{equation}
Note that \eqref{1511161925} is also a necessary condition for the existence of $l\in\mathbb{Z}$ 
satisfying  \eqref{1511161926}. 


\section{The node sets}
\label{17008191031}

 \begin{figure}[htb]
	\centering
	\subfigure[\hspace*{1em} The index set $\I^{(10,5)}_{(0,0)}$]{\includegraphics[scale=0.85]{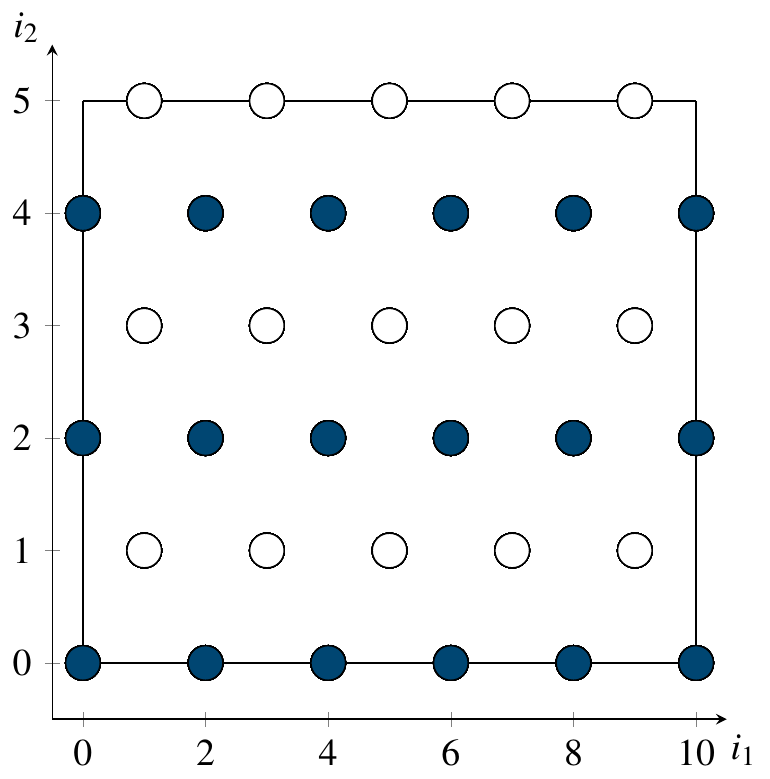}}
	\hfill	
	\subfigure[\hspace*{1em} $\LC^{(10,5)}_{(0,0)}$ and $\mathcal{C}^{(10,5)}_{(0,0)} = \displaystyle \!\! \bigcup_{\rho \in \{0,2,4\}} \!\! \vect{\ell}^{(10,5)}_{(0,\rho)}(\mathbb{R})$
	]{\includegraphics[scale=0.8]{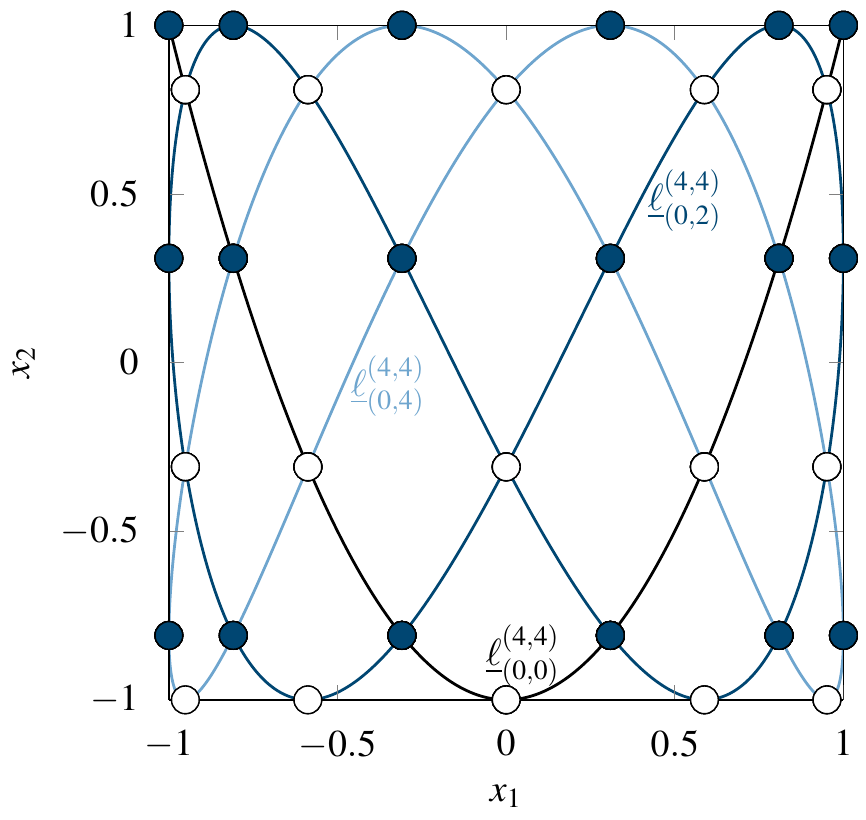}} 
  	\caption{Illustration of the index set $\I^{(10,5)}_{(0,0)}$ (left) and the Lissajous-Chebyshev node points $\LC^{(10,5)}_{(0,0)}$ (right). 
  	The subsets $\I^{(10,5)}_{(0,0),0}$, $\LC^{(10,5)}_{(0,0),0}$ and $\I^{(10,5)}_{(0,0),1}$, $\LC^{(10,5)}_{(0,0),1}$ are colored in blue and 
  	white, respectively. The generating Chebyshev variety $\mathcal{C}^{(10,5)}_{(0,0)}$ (right) can be written as the union of 
  	a degenerate Lissajous curve $\vect{\ell}^{(10,5)}_{(0,0)}(\mathbb{R})$ (black) with two non-degenerate Lissajous curves $\vect{\ell}^{(10,5)}_{(0,2)}(\mathbb{R})$
  	and $\vect{\ell}^{(10,5)}_{(0,4)}(\mathbb{R})$ (dark and light blue, respectively), see also Example \ref{ex:1} and Example \ref{1710311355}. 
  	} 
	\label{fig:doublem-1}
\end{figure}

In order to introduce the node sets for the multivariate polynomial interpolation under consideration, 
we define appropriate index sets $\I^{(\vect{m})}_{\vect{\kappa},0}$, $\I^{(\vect{m})}_{\vect{\kappa},1}$
and $\I^{(\vect{m})}_{\vect{\kappa}}$ that enable us to parametrize the node points. 
For $\vect{m}\in\mathbb{N}^{\mathsf{d}}$ and $\vect{\kappa}\in\mathbb{Z}^{\mathsf{d}}$, we define
\begin{equation}\label{1509221632}
\begin{split}
\I^{(\vect{m})}_{\vect{\kappa}}&=\I^{(\vect{m})}_{\vect{\kappa},0}\cup \I^{(\vect{m})}_{\vect{\kappa},1},\qquad  \text{with the sets $\I^{(\vect{m})}_{\vect{\kappa},\mathfrak{r}}$, $\mathfrak{r}\in\{0,1\}$, given by}\\[0.2em]
\I^{(\vect{m})}_{\vect{\kappa},\mathfrak{r}}&=\left\{\,\vect{i}\in\mathbb{N}_0^{\mathsf{d}}\,\left|\,\forall\,\mathsf{i}:\ 0\leq i_{\mathsf{i}}\leq m_{\mathsf{i}}\ 
\text{and}\ \ i_{\mathsf{i}}\equiv \kappa_{\mathsf{i}}-\mathfrak{r} \tmod 2\right.\,\right\}.
\end{split}
\end{equation}
Further, for $\mathsf{M}\subseteq\{1,\ldots,\mathsf{d}\}$, we introduce the subsets
\begin{equation}\label{1609080206}
\I^{(\vect{m})}_{\vect{\kappa},\mathsf{M}}=\I^{(\vect{m})}_{\vect{\kappa},\mathsf{M},0}\cup \I^{(\vect{m})}_{\vect{\kappa},\mathsf{M},1},\qquad 
\I^{(\vect{m})}_{\vect{\kappa},\mathsf{M},\mathfrak{r}}=
\left\{\,\left.\vect{i}\in \I^{(\vect{m})}_{\vect{\kappa},\mathfrak{r}}\,\right|\,0<i_{\mathsf{i}}<m_{\mathsf{i}}\Leftrightarrow\mathsf{i}\in \mathsf{M}\,\right\}.
\end{equation}
A typical example of these sets for dimension $\mathsf{d} =2$ is illustrated in Figure~\ref{fig:doublem-1}~(a). 
From the cross product structure of the sets $\I^{(\vect{m})}_{\vect{\kappa},0}$ and $\I^{(\vect{m})}_{\vect{\kappa},1}$, we can easily derive 
\begin{equation}\label{1609080404}
\# \I^{(\vect{m})}_{\vect{\kappa},\mathfrak{r}} = \tprod_{\substack{\mathsf{i}\in \{1,\ldots,\mathsf{d}\}\\ m_{\mathsf{i}}\equiv 0\tmod 2\\\kappa_{\mathsf{i}}\equiv \mathfrak{r}\tmod 2}}\text{\raisebox{-0.5em}{$\dfrac{m_{\mathsf{i}}+2}{2}$}}\quad \times\ \tprod_{\substack{\mathsf{i}\in \{1,\ldots,\mathsf{d}\}\\ m_{\mathsf{i}}\equiv 0\tmod 2\\\kappa_{\mathsf{i}}\not\equiv \mathfrak{r}\tmod 2}}\text{\raisebox{-0.5em}{$\dfrac{m_{\mathsf{i}}}2$}}\quad \times\tprod_{\substack{\mathsf{i}\in \{1,\ldots,\mathsf{d}\}\\ m_{\mathsf{i}}\equiv 1\tmod 2}}\text{\raisebox{-0.5em}{$\dfrac{m_{\mathsf{i}}+1}{2}$}}.
\end{equation}
It is easy to check that the value \eqref{1609080404} of $\I^{(\vect{m})}_{\vect{\kappa},\mathfrak{r}}$ can also be written as
\begin{equation}\label{1509221434}
\#\I^{(\vect{m})}_{\vect{\kappa},\mathfrak{r}}=\dfrac1{2^{\mathsf{d}}}\tprod_{\substack{\mathfrak{s}\in \{0,1\}\\\mathfrak{t}\in \{0,1\}}}\tprod_{\substack{\mathsf{i}\in \{1,\ldots,\mathsf{d}\}\\ m_{\mathsf{i}}\equiv \mathfrak{s}\tmod 2\\\kappa_{\mathsf{i}}-\mathfrak{r}\equiv \mathfrak{t}\tmod 2}}(m_{\mathsf{i}}+2-\mathfrak{s}-2\mathfrak{t}+2\mathfrak{s}\mathfrak{t}).
\end{equation}
Note that in  \eqref{1609080404}, \eqref{1509221434} the product over the empty set is as usual considered to be~$1$.

Since $\I^{(\vect{m})}_{\vect{\kappa},0}$ and  $\I^{(\vect{m})}_{\vect{\kappa},1}$ are disjoint sets, we can sum the values in \eqref{1509221434} to obtain 
\begin{equation} \label{1509221435} \#\I^{(\vect{m})}_{\vect{\kappa}}=\#\I^{(\vect{m})}_{\vect{\kappa},0}+\#\I^{(\vect{m})}_{\vect{\kappa},1}.\end{equation}
In the same way we get for the subsets defined in \eqref{1609080206} with $\mathsf{M}\subseteq\{1,\ldots,\mathsf{d}\}$ that
\newlength{\arraycolsepsave}
\setlength{\arraycolsepsave}{\arraycolsep}
\setlength{\arraycolsep}{1pt}
\begin{equation}\label{1509221441}
\#\I^{(\vect{m})}_{\vect{\kappa},\mathsf{M},\mathfrak{r}}=\dfrac1{2^{\# \mathsf{M}}}\tprod_{\substack{\mathfrak{s}\in \{0,1\}\\\mathfrak{t}\in \{0,1\}}}\tprod_{\substack{\mathsf{i}\in \{1,\ldots,\mathsf{d}\}\\ m_{\mathsf{i}}\equiv \mathfrak{s}\tmod 2\\\kappa_{\mathsf{i}}-\mathfrak{r}\equiv \mathfrak{t}\tmod 2}}\text{\raisebox{-0.3em}{$\displaystyle\left\{ \begin{array}{cccl} 
m_{\mathsf{i}}&-&2+\mathfrak{s}+2\mathfrak{t}-2\mathfrak{s}\mathfrak{t} \; &\quad \text{if}\ \mathsf{i}\in \mathsf{M},\\
&&2-\mathfrak{s}-2\mathfrak{t}+2\mathfrak{s}\mathfrak{t} \; &  \quad \text{if}\ \mathsf{i}\not\in \mathsf{M}.
\end{array}\right.$}}
\end{equation}
\setlength{\arraycolsep}{\arraycolsepsave}
Also here, we have 
\begin{equation} \label{1708240203}\#\I^{(\vect{m})}_{\vect{\kappa},\mathsf{M}}=\#\I^{(\vect{m})}_{\vect{\kappa},\mathsf{M},0}+\#\I^{(\vect{m})}_{\vect{\kappa},\mathsf{M},1}.
\end{equation}

\medskip

Now, using the index sets $\I^{(\vect{m})}_{\vect{\kappa},0}$, $\I^{(\vect{m})}_{\vect{\kappa},1}$ and $\I^{(\vect{m})}_{\vect{\kappa}}$, we introduce the node sets for the 
multivariate polynomial interpolation problem considered in this work as
\begin{equation} \label{eq:LC-pointsdefinition} \LC^{(\vect{m})}_{\vect{\kappa}} = \LC^{(\vect{m})}_{\vect{\kappa},0}\cup \LC^{(\vect{m})}_{\vect{\kappa},1},\quad \text{where} \quad
\LC^{(\vect{m})}_{\vect{\kappa},\mathfrak{r}}=\left\{\, \vect{z}^{(\vect{m})}_{\vect{i}}\,\left|\,\vect{i}\in \I^{(\vect{m})}_{\vect{\kappa},\mathfrak{r}} \right.\right\}.
\end{equation}
Here, the bijective mapping $\vect{i}\mapsto \vect{z}^{(\vect{m})}_{\vect{i}}$ from the index sets $\I^{(\vect{m})}_{\vect{\kappa},\mathfrak{r}}$ 
onto the node sets $\LC^{(\vect{m})}_{\vect{\kappa},\mathfrak{r}}$ is carried out with help of the Chebyshev-Gauß-Lobatto points $\vect{z}^{(\vect{m})}_{\vect{i}}$ 
given in~\eqref{201511121326}. Since the sets $\LC^{(\vect{m})}_{\vect{\kappa},0}$ and $\LC^{(\vect{m})}_{\vect{\kappa},1}$ are disjoint, this mapping is in particular
a bijection from  $\I^{(\vect{m})}_{\vect{\kappa}}$ onto $\LC^{(\vect{m})}_{\vect{\kappa}}$. 
Due to the intimate connection of the sets $\LC^{(\vect{m})}_{\vect{\kappa}}$ in~\eqref{eq:LC-pointsdefinition} to multivariate Lissajous curves and Chebyshev varieties, 
we call the elements of $\LC^{(\vect{m})}_{\vect{\kappa}}$ \textit{Lissajous-Chebyshev} points. 
A bivariate example is given in Figure~\ref{fig:doublem-1}~(b).

\medskip

According to the subsets $\I^{(\vect{m})}_{\vect{\kappa},\mathsf{M},\mathfrak{r}} \subseteq \I^{(\vect{m})}_{\vect{\kappa},\mathfrak{r}}$, $\mathsf{M}\subseteq\{1,\ldots,\mathsf{d}\}$, 
we also define
\[\LC^{(\vect{m})}_{\vect{\kappa},\mathsf{M}} =\LC^{(\vect{m})}_{\vect{\kappa},\mathsf{M},0}\cup\LC^{(\vect{m})}_{\vect{\kappa},\mathsf{M},1},\quad \LC^{(\vect{m})}_{\vect{\kappa},\mathsf{M},\mathfrak{r}}=\left\{\, \vect{z}^{(\vect{m})}_{\vect{\kappa},\vect{i}}\,\left|\,\vect{i}\in \I^{(\vect{m})}_{\vect{\kappa},\mathsf{M},\mathfrak{r}} \right.\right\}.\]
Then, we obviously have
\[\LC^{(\vect{m})}_{\vect{\kappa},\mathsf{M}}=\LC^{(\vect{m})}_{\vect{\kappa}}\cap \vect{F}^{\mathsf{d}}_{\mathsf{M}},\qquad 
\LC^{(\vect{m})}_{\vect{\kappa},\mathsf{M},\mathfrak{r}}=\LC^{(\vect{m})}_{\vect{\kappa},\mathfrak{r}}\cap \vect{F}^{\mathsf{d}}_{\mathsf{M}},\]
where
\[
\vect{F}^{\mathsf{d}}_{\mathsf{M}}=\left\{\,\left.\vect{x}\in [-1,1]^{\mathsf{d}}\,\right|\, x_{\mathsf{i}}\in (-1,1) \Leftrightarrow\mathsf{i}\in \mathsf{M} \,\right\}
\]
denotes an assembly of $\# \mathsf{M}$-faces of the hypercube $[-1,1]^{\mathsf{d}}$ which is related to the definition of the subsets $\I^{(\vect{m})}_{\vect{\kappa},\mathsf{M},\mathfrak{r}}$ given in
\eqref{1609080206}. In the same way as above, we get 
\[\#\LC^{(\vect{m})}_{\vect{\kappa}}=\#\I^{(\vect{m})}_{\vect{\kappa}},\qquad \#\LC^{(\vect{m})}_{\vect{\kappa},\mathfrak{r}}=\#\I^{(\vect{m})}_{\vect{\kappa},\mathfrak{r}},\] and 
\[\#\LC^{(\vect{m})}_{\vect{\kappa},\mathsf{M}}=\#\I^{(\vect{m})}_{\vect{\kappa},\mathsf{M}},\qquad \#\LC^{(\vect{m})}_{\vect{\kappa},\mathsf{M},\mathfrak{r}}=\#\I^{(\vect{m})}_{\vect{\kappa},\mathsf{M},\mathfrak{r}},\]
with the explicit values \eqref{1509221434}, \eqref{1509221435} and \eqref{1509221441}, \eqref{1708240203} for the cardinalities. 

\medskip

In this article, we make use of a component-wise multiplicative decomposition of the parameter vector $\vect{m}$. For this decomposition we use the notation \eqref{1608311012}.

\begin{proposition} \label{1509061252} Let $\vect{m}\in\mathbb{N}^{\mathsf{d}}$. There exist (not necessarily  uniquely determined) integer vectors  $\vect{m}^{\sharp},\vect{m}^{\flat}\in\mathbb{N}^{\mathsf{d}}$ such that the following properties are satisfied:
\stepcounter{equation}
\begin{align}
& \label{1509091200A}\text{For all $\mathsf{i}\in\{1,\ldots,\mathsf{d}\}${\rm :} $m_{\mathsf{i}}=m^{\flat}_{\mathsf{i}}m^{\sharp}_{\mathsf{i}}.$}\tag{\theequation a}\\
&\label{1509091200B}\text{For all $\mathsf{i}\in\{1,\ldots,\mathsf{d}\}${\rm :} $m^{\flat}_{\mathsf{i}}$ and $m^{\sharp}_{\mathsf{i}}$ are relatively prime.}\tag{\theequation b}\\
&\label{1509091200C}\text{The numbers  $m^{\sharp}_1,\ldots,m^{\sharp}_{\mathsf{d}}$ are pairwise relatively prime.}\tag{\theequation c}\\
&\label{1509091200D}\text{We have $\lcm[\vect{m}]=\p[\vect{m}^{\sharp}]$.}\tag{\theequation d}
\end{align}
\end{proposition}

We remark that these conditions are not independent of each other. Indeed, under the assumption \eqref{1509091200A} and \eqref{1509091200C}, we  have \eqref{1509091200D} if and only if we have \eqref{1509091200B}.

\medskip

\begin{proof} We denote   the exponent of a prime $p$ in  the prime factorization of $m_{\mathsf{i}}$ and  $\lcm[\vect{m}]$ by 
$e_{\mathsf{i}}(p)$ and $e(p)$, respectively. Let $p_1,\ldots,p_n$ be pairwise different and containing all 
primes that divide $\lcm[\vect{m}]$. We consider a recursive construction. Let $\vect{m}^{\sharp}_0=\vect{m}^{\flat}_0=\vect{1}$ and suppose 
we have already chosen $\vect{m}^{\sharp}_{s-1},\vect{m}^{\flat}_{s-1}$ based on powers of the primes $p_1, \ldots, p_{s-1}$. There is a  (not necessarily  uniquely determined) $\mathsf{j}$ 
such that $e_{\mathsf{j}}(p_{s})=e(p_{s})$. We set
$m^{\sharp}_{s,\mathsf{j}}=p_{s}^{e_{\mathsf{j}}(p_{s})}m^{\sharp}_{s-1,\mathsf{j}}$ and 
$m^{\sharp}_{s,\mathsf{i}}=m^{\sharp}_{s-1,\mathsf{i}}$ for $\mathsf{i}\neq \mathsf{j}$ as well as 
$m^{\flat}_{s,\mathsf{j}}=m^{\flat}_{s-1,\mathsf{j}}$ and $m^{\flat}_{s,\mathsf{i}}=p_{s}^{e_{\mathsf{i}}(p_{s})}m^{\flat}_{s-1,\mathsf{i}}$ 
for $\mathsf{i}\neq \mathsf{j}$. This yields $\vect{m}^{\sharp}_{s},\vect{m}^{\flat}_{s}$ from $\vect{m}^{\sharp}_{s-1},\vect{m}^{\flat}_{s-1}$. 
Then, after $n$ steps, $\vect{m}^{\sharp}=\vect{m}^{\sharp}_n$ and $\vect{m}^{\flat}=\vect{m}^{\flat}_n$ have the asserted properties.
\end{proof}

\medskip

For $\vect{m}^{\sharp},\vect{m}^{\flat}\in\mathbb{N}^{\mathsf{d}}$, we define the sets 
\begin{equation}\label{1708242016} 
H^{(\vect{m}^{\sharp})}=\{0,\ldots,2\p[\vect{m}^{\sharp}]-1\}, \qquad
 \vect{R}^{(\vect{m}^{\flat})}= \bigtimes_{\substack{\vspace{-8pt}\\\mathsf{i}=1}}^{\substack{\mathsf{d}\\\vspace{-10pt}}} \{0,\ldots,m^{\flat}_{\mathsf{i}}-1\}.
\end{equation}
In Proposition~\ref{1509221521} we will obtain an
identification of the  elements of $\I^{(\vect{m})}_{\vect{\kappa}}$ with the elements of a class decomposition of the set 
$H^{(\vect{m}^{\sharp})}\times \vect{R}^{(\vect{m}^{\flat})}$.  The particular importance of this statement will be seen in Section \ref{1708201701} 
for the further characterization of the point sets $\LC^{(\vect{m})}_{\vect{\kappa}}$ and in Section \ref{17008191036} for the 
derivation of a discrete orthogonality structure. 
\begin{lemma} \label{1708311652}
Let  $\vect{m},\vect{m}^{\sharp},\vect{m}^{\flat}\in\mathbb{N}^{\mathsf{d}}$  satisfy \eqref{1509091200A}, 
\eqref{1509091200B}, \eqref{1509091200C} and \eqref{1509091200D}.

\medskip

If $\vectrho\in\mathbb{Z}^{\mathsf{d}}$ and if there is an integer $k\in\mathbb{Z}$ such that
\[\forall\,\mathsf{i} \in\{1,\ldots,\mathsf{d}\}:\quad k\equiv \rho_{\mathsf{i}}m^{\sharp}_{\mathsf{i}}\mod m_{\mathsf{i}},\]

then 
\[\forall\,\mathsf{i} \in\{1,\ldots,\mathsf{d}\}:\quad \rho_{\mathsf{i}}\equiv 0\mod m^{\flat}_{\mathsf{i}}.\]
\end{lemma}
\begin{proof}  Let $e_{\mathsf{j}}(p)$ be the exponent of a prime $p$ in the prime decomposition of  $m_{\mathsf{j}}$.  The assumption implies $\rho_{\mathsf{i}} m^{\sharp}_{\mathsf{i}}\equiv\rho_{\mathsf{j}} m^{\sharp}_{\mathsf{j}}\tmod \gcd\{m_{\mathsf{i}},m_{\mathsf{j}}\}$ for  all $\mathsf{i},\mathsf{j}$. Let $p$ be a prime that divides~$m^{\flat}_{\mathsf{i}}$. By \eqref{1509091200B}, the prime $p$ does not divide $m^{\sharp}_{\mathsf{i}}$. By \eqref{1509091200A},\eqref{1509091200C}, \eqref{1509091200D}, there exists $\mathsf{j}$ such that $p^{e_{\mathsf{j}}(p)}$ divides $m^{\sharp}_{\mathsf{j}}$ and 
$e_{\mathsf{j}}(p)\geq e_{\mathsf{i}}(p)$. Then, $p^{e_{\mathsf{i}}(p)}$ divides $\gcd\{m_{\mathsf{i}},m_{\mathsf{j}}\}$, and therefore, $p^{e_{\mathsf{i}}(p)}$ divides $\rho_{\mathsf{i}}$. We conclude the assertion.
\end{proof}
\begin{proposition}\label{1509221521} Let  $\vect{m},\vect{m}^{\sharp},\vect{m}^{\flat}\in\mathbb{N}^{\mathsf{d}}$  satisfy \eqref{1509091200A}, 
\eqref{1509091200B}, \eqref{1509091200C} and \eqref{1509091200D}.
\begin{enumerate}[a)]
 \item For all\, $(l,\vectrho)\in H^{(\vect{m}^{\sharp})}\times \vect{R}^{(\vect{m}^{\flat})}$, there exists a uniquely determined  element  $\vect{i}\in \I^{(\vect{m})}_{\vect{\kappa}}$ and a  (not necessarily  unique) $\vect{v}\in\{-1,1\}^{\mathsf{d}}$ such that 
\begin{equation}\label{1509221526}
\forall\,\mathsf{i} \in\{1,\ldots,\mathsf{d}\}:\quad i_{\mathsf{i}} \equiv v_{\mathsf{i}}\left(l - 2\rho_{\mathsf{i}}m^{\sharp}_{\mathsf{i}}-\kappa_{\mathsf{i}}\right) \mod 2m_{\mathsf{i}}.
\end{equation}
Therefore, a function \text{$\vectj\!:H^{(\vect{m}^{\sharp})}\times \vect{R}^{(\vect{m}^{\flat})}\to\I^{(\vect{m})}_{\vect{\kappa}}$} is well defined by $\vectj(l,\vectrho)=\vect{i}.$
\item For  $(l,\vectrho)\in H^{(\vect{m}^{\sharp})}\times \vect{R}^{(\vect{m}^{\flat})}$ and $\mathfrak{r}\in\{0,1\}:$  $\vectj(l,\vectrho)\in \I^{(\vect{m})}_{\vect{\kappa},\mathfrak{r}}$ $\Longleftrightarrow$ $l\equiv\mathfrak{r}\tmod 2$.
\item Using \eqref{1609080206}, for $\vect{i}\in \I^{(\vect{m})}_{\vect{\kappa},\mathsf{M}}$ with $\mathsf{M}\subseteq\{1,\ldots,\mathsf{d}\}$ we have \[\#\{\,(l,\vectrho)\in H^{(\vect{m}^{\sharp})}\times \vect{R}^{(\vect{m}^{\flat})}\,|\,\vectj(l,\vectrho)=\vect{i}\,\}=2^{\#\mathsf{M}}.\]
\end{enumerate}
\end{proposition}

Note that we have fixed  $\vect{m},\vect{m}^{\sharp},\vect{m}^{\flat}$ and that we have omitted the indication of the dependency on $\vect{m}^{\sharp}$, $\vect{m}^{\flat}$, $\vect{\kappa}$ in the notation of the function $\vectj$.

\medskip

\begin{proof} 
For $(l,\vectrho)$ we can find a tupel $\vect{i}$ with $0\leq i_{\mathsf{i}}\leq m_{\mathsf{i}}$ and $\vect{v} \in \{-1,1\}^{\mathsf{d}}$ satisfying~\eqref{1509221526}.
Further, we have $i_{\mathsf{i}}+\kappa_{\mathsf{i}}\equiv l\equiv  i_{\mathsf{j}}+\kappa_{\mathsf{j}} \tmod 2$ for all \text{$\mathsf{i}$, $\mathsf{j}$}. 
Therefore, we have $\vect{i}\in \I^{(\vect{m})}_{\vect{\kappa}}$ and the restriction $0\leq i_{\mathsf{i}}\leq m_{\mathsf{i}}$ implies the uniqueness of
$\vect{i}$. We have shown~a), and statement b) follows directly from \eqref{1509221526} and the definitions in \eqref{1509221632}. 

If $\vect{i}\in\I^{(\vect{m})}_{\vect{\kappa}}$ and $\vect{v}\in \{-1,1\}^{\mathsf{d}}$, then for $a_{\mathsf{i}}=v_{\mathsf{i}}i_{\mathsf{i}}+\kappa_{\mathsf{i}}$ and $\vect{k} = 2 \vect{m}^{\sharp}$ condition \eqref{1511161925} holds by \eqref{1509091200C} and
by the definition of $\I^{(\vect{m})}_{\vect{\kappa}}$. Now, the Chinese remainder theorem yields a unique number $l\in \{0,\ldots,2\p[\vect{m}^{\sharp}]-1\}$ satisfying
\[
\forall\,\mathsf{i}\in\{1,\ldots,\mathsf{d}\}:\quad v_{\mathsf{i}}i_{\mathsf{i}}+\kappa_{\mathsf{i}} \equiv l \mod 2m^{\sharp}_{\mathsf{i}}.
\]
Since \eqref{1509091200A} holds, we can also find an element $\vectrho\in \vect{R}^{(\vect{m}^{\flat})}$ such that \eqref{1509221526} holds. By Lemma \ref{1708311652} such an element $\vectrho\in \vect{R}^{(\vect{m}^{\flat})}$ is uniquely determined.
In this way, a 
function $\mathfrak{g}_{\vect{i}}:\{-1,1\}^{\mathsf{d}}\to H^{(\vect{m}^{\sharp})}\times \vect{R}^{(\vect{m}^{\flat})}$ is well-defined 
by $\mathfrak{g}_{\vect{i}}(\vect{v})=(l,\vectrho)$, where $(l,\vectrho)$ is chosen such that \eqref{1509221526} holds. If $\vect{i}\in \I^{(\vect{m})}_{\vect{\kappa},\mathsf{M}}$, then $\mathfrak{g}_{\vect{i}}(\vect{v}')= \mathfrak{g}_{\vect{i}}(\vect{v})$ holds if and only if 
\[\forall\,\mathsf{i}\in \{1,\ldots,\mathsf{d}\}:\quad (v'_{\mathsf{i}}-v_{\mathsf{i}})i_{\mathsf{i}}\equiv 0\mod 2m_{\mathsf{i}},\]
i.e. if and only if   $v'_{\mathsf{i}}=v_{\mathsf{i}}$ for all $\mathsf{i}\in\mathsf{M}$. We conclude statement c).
\end{proof}

\begin{example} \label{ex:1}
As a first example, we consider the bivariate node sets $\LC^{(2m,m)}_{(0,0)}$, with a parameter $m \in \mathbb{N}$. For $m = 5$, the set $\LC^{(10,5)}_{(0,0)}$ and the corresponding
index set $\I^{(10,5)}_{(0,0)}$ are illustrated in Figure \ref{fig:doublem-1}.  A possible decomposition of $\vect{m} = (2m,m)$ according to  Proposition \ref{1509061252} is given by 
$\vect{m}^{\sharp} = (2m,1)$, $\vect{m}^{\flat} = (1, m)$. Then, the sets in \eqref{1708242016} are 
\[H^{(2m,1)}=\{0,\ldots,4m-1\}, \qquad
\vect{R}^{(1, m)} = \{ 0 \} \times \{0,\ldots,m-1\}.\]
The number  of elements in $\LC^{(2m,m)}_{(0,0),\mathfrak{r}}$, $\mathfrak{r} \in \{0,1\}$, is given by 
\[ \# \LC^{(2m,m)}_{(0,0),\mathfrak{r}} = \dfrac1{2} \left\{ 
\begin{array}{cl}
 (m+1)(m+1-\mathfrak{r})  & \text{if $m$ is odd}, \\
 (m+1-\mathfrak{r})(m+2-2\mathfrak{r}) & \text{if $m$ is even}.
\end{array} \right.
\]
\end{example}

\setcounter{footnote}{-1}
\section[Characterizations of the node sets]{Characterizations of the node sets$^*$\footnote{$^*$The considerations  in the  Sections \ref{17008191035}, \ref{17008191036}, \ref{1609011843} and \ref{17008191038} are independent of this Section \ref{1708201701}.}}\label{1708201701}

For a first characterization of the node sets $\LC^{(\vect{m})}_{\vect{\kappa}}$, we consider the affine real algebraic 
Chebyshev variety $\mathcal{C}^{(\vect{m})} $ over the hypercube $[-1,1]^{\mathsf{d}}$ given by
\begin{equation} \label{1509222009} 
\mathcal{C}^{(\vect{m})}_{\vect{\kappa}} 
= \left\{\,\left.\vect{x} \in [-1,1]^{\mathsf{d}}\,\right|\,(-1)^{\kappa_1} T_{m_1}(x_1) = \ldots = (-1)^{\kappa_{\mathsf{d}}} T_{m_{\mathsf{d}}}(x_{\mathsf{d}}) \,\right\}.
\end{equation}
The name of this variety stems from the fact that $T_{m_{\mathsf{i}}}(x_{\mathsf{i}}) = \cos(m_{\mathsf{i}} \arccos x_{\mathsf{i}} )$ are the Chebyshev polynomials of the first kind with degree $m_{\mathsf{i}}$.
The Chebyshev variety $\mathcal{C}^{(\vect{m})}_{\vect{\kappa}}$ and its singularities are related to the considered node sets in the following way.

\begin{theorem} 
A point $\vect{x}\in \mathcal{C}^{(\vect{m})}_{\vect{\kappa}}$ is a  singular point of $\mathcal{C}^{(\vect{m})}_{\vect{\kappa}}$ if and only if it is an element of $\LC^{(\vect{m})}_{\vect{\kappa},\mathsf{M}}$ for $\mathsf{M}\subseteq \{1,\ldots,\mathsf{d}\}$ with $\#\mathsf{M}\geq 2$. Further, for $\mathfrak{r}\in \{0,1\}$ 
we have
\begin{equation}  \label{1509222041}
\LC^{(\vect{m})}_{\vect{\kappa},\mathsf{M},\mathfrak{r}} =\left\{ \, \left.\vect{x} \in \vect{F}^{\mathsf{d}}_{\mathsf{M}}\,\right|\, (-1)^{\kappa_1}T_{m_1}(x_1) = \ldots = (-1)^{\kappa_{\mathsf{d}}}T_{m_{\mathsf{d}}}(x_{\mathsf{d}}) = (-1)^{\mathfrak{r}} \,\right\}.
\end{equation}
\end{theorem}
The formula \eqref{1509222041} holds for all $\mathsf{M}\subseteq \{1,\ldots,\mathsf{d}\}$. 
If $\#\mathsf{M}\in \{0,1\}$, the elements of the sets in \eqref{1509222041} are regular points of the variety on the corners and edges of $[-1,1]^{\mathsf{d}}$.

\medskip

\begin{proof}
We have $\mathcal{C}^{(\vect{m})}_{\vect{\kappa}} = \left\{\,\left.\vect{x} \in [-1,1]^{\mathsf{d}}\,\right|\,f_1(\vect{x})=0,f_2(\vect{x})=0\ldots,f_{\mathsf{d}-1}(\vect{x})=0\right.\,\}$, where $f_{\mathsf{i}}(\vect{x})=S_{\mathsf{i}}(x_{\mathsf{i}})-S_{\mathsf{i}+1}(x_{\mathsf{i}+1})$ with $S_{\mathsf{i}}=(-1)^{\kappa_{\mathsf{i}}}T_{m_{\mathsf{i}}}$.
The singular points  are the points $\vect{x}^* \in [-1,1]^{\mathsf{d}}$ for which the Jacobian matrix of $(f_1,\ldots,f_{\mathsf{d}-1})$, given by
\[\begin{pmatrix} S'_1(x^*_1) & - S'_2(x^*_2)&  &  &   \\  & S'_2(x^*_2) & -S'_3(x^*_3)   & &  \\ && \ddots & \ddots & \\ &&& S'_{\mathsf{d}-1}(x^*_{\mathsf{d}-1})& - S'_{\mathsf{d}}(x^*_{\mathsf{d}})
                       \end{pmatrix},
\]
has not full rank $\mathsf{d}-1$. This is exactly the case if
$T'_{m_{\mathsf{i}}}(x_{\mathsf{i}}^*) = 0$ for at least two different indices $\mathsf{i}$.  
We have $T'_{m_{\mathsf{i}}}(x_{\mathsf{i}}^*)=0$ if and only if  $x_{\mathsf{i}}^*=z^{(m_{\mathsf{i}})}_{i_{\mathsf{i}}}$
for some $i_{\mathsf{i}}\in \{1,\ldots,m_{\mathsf{i}}-1\}$, and in this case  $T_{m_{\mathsf{i}}}(x_{\mathsf{i}}^*)\in \{-1,1\}$.
Since $\vect{x}^*\in \mathcal{C}^{(\vect{m})}_{\vect{\kappa}}$, there exists $\mathfrak{r}\in\{0,1\}$ such that
\[ (-1)^{\kappa_1}T_{m_1}(x_1^*) = \ldots = (-1)^{\kappa_{\mathsf{d}}}T_{m_{\mathsf{d}}}(x_{\mathsf{d}}^*)=(-1)^{\mathfrak{r}}.\]
Thus, we get $\vect{i}\in \I^{(\vect{m})}_{\vect{\kappa},\mathsf{M},\mathfrak{r}}$, $\vect{x}^*=\vect{z}^{(\vect{m})}_{\vect{i}}\in\LC^{(\vect{m})}_{\vect{\kappa},\mathsf{M},\mathfrak{r}}$ 
for some $\mathsf{M} \subseteq \{1,\ldots,\mathsf{d}\}$ with $\#M\geq 2$. On the other hand, for every point 
$\vect{z}^{(\vect{m})}_{\vect{i}}\in\LC^{(\vect{m})}_{\vect{\kappa},\mathsf{M},\mathfrak{r}}$
with $\#M\geq 2$, we have that $T'_{m_{\mathsf{i}}}(z_{i_{\mathsf{i}}}^{(m_{\mathsf{i}})}) = 0$ 
for all  $\mathsf{i} \in \mathsf{M}$ and the rank of the Jacobian matrix is $\mathsf{d}-\#M<\mathsf{d}-1$.

From the arguments above, the general formula \eqref{1509222041} can be derived in a straightforward way.
If~$\#\mathsf{M}\in \{0,1\}$, then the rank of the Jacobian matrix is $\mathsf{d}-1$ and the points in $\LC^{(\vect{m})}_{\vect{\kappa},\mathsf{M},\mathfrak{r}}$
are regular points on the edges (if $\#\mathsf{M} = 1$) or on the corners (if~$\#\mathsf{M} = 0$) of the hypercube $[-1,1]^{\mathsf{d}}$.
\end{proof}
 
\medskip

Using the class decomposition of  $H^{(\vect{m}^{\sharp})}\times \vect{R}^{(\vect{m}^{\flat})}$ according to Proposition~\ref{1509221521}, we can now characterize the node points $\LC^{(\vect{m})}_{\vect{\kappa}}$ and the 
Chebyshev variety $\mathcal{C}^{(\vect{m})}_{\vect{\kappa}}$ in terms of Lissajous curves. The statements are generalizations of results in \cite{DenckerErb2015a}. 

\medskip

 For $\vect{m}\in\mathbb{N}^{\mathsf{d}}$, $\vect{\xi}\in\mathbb{R}^{\mathsf{d}}$, $\vect{u}\in\{-1,1\}^{\mathsf{d}}$ 
 we define the curves $\vect{\ell}^{(\vect{m})}_{\vect{\xi},\vect{u}}\,:\,\mathbb{R}\to [-1,1]^{\mathsf{d}}$ by
\[\vect{\ell}^{(\vect{m})}_{\vect{\xi},\vect{u}}(t)=\left(u_1\cos \left(  \;\! \frac{\lcm[\vect{m}] \cdot t- \xi_1\pi}{m_1}\right), \cdots, u_{\mathsf{d}}\cos \left( \;\! \frac{\lcm[\vect{m}] \cdot t - \xi_{\mathsf{d}}\pi}{m_{\mathsf{d}}}\right) \right), \quad t \in \mathbb{R}.\]
By \cite[Theorem 1.1]{DenckerErb2015a} we know that the Lissajous curve $\vect{\ell}^{(\vect{m})}_{\vect{\xi},\vect{u}}$ is a closed curve with fundamental period $2\pi$.  
If $\vect{u}=\vect{1}$, we omit this parameter in the notation and write
\[\vect{\ell}^{(\vect{m})}_{\vect{\xi}}=\vect{\ell}^{(\vect{m})}_{\vect{\xi},\vect{1}}.\]
We consider now the following set of Lissajous curves
\begin{equation}\label{1609010824} 
\vect{\mathfrak{L}}^{(\vect{m}^{\sharp}, \, \vect{m}^{\flat})}_{\vect{\kappa}}=\left\{ \,\vect{\ell}^{(\vect{m})}_{\left(2\rho_1m^{\sharp}_1+\kappa_1,\ldots,2\rho_{\mathsf{d}}m^{\sharp}_{\mathsf{d}}+\kappa_{\mathsf{d}}\right)}\left|\ \vectrho\in \vect{R}^{(\vect{m}^{\flat})}\right.\right\}.
\end{equation}

\begin{theorem} \label{thm:decompositionlissajous} Let  $\vect{m},\vect{m}^{\sharp},\vect{m}^{\flat}\in\mathbb{N}^{\mathsf{d}}$  satisfy \eqref{1509091200A}, 
\eqref{1509091200B}, \eqref{1509091200C} and \eqref{1509091200D}.
\begin{enumerate}[a)]
 \item Using the sampling points $t^{(\vect{m})}_{l}$ given in \eqref{1608311013}, we have for $\mathfrak{r}\in \{0,1\}$ the identity
\[\LC^{(\vect{m})}_{\vect{\kappa},\mathfrak{r}}
= \left\{\,\vect{\ell}(t^{(\vect{m})}_{l})\,\left|\ \vect{\ell}\in \vect{\mathfrak{L}}^{(\vect{m}^{\sharp}, \, \vect{m}^{\flat})}_{\vect{\kappa}},\,l\in H^{(\vect{m}^{\sharp})}\ 
\text{with}\  l\equiv \mathfrak{r}\tmod 2\right.\right\}.\]
 \item The Chebyshev variety $\mathcal{C}^{(\vect{m})}_{\vect{\kappa}}$ can be written as 
$\displaystyle\mathcal{C}^{(\vect{m})}_{\vect{\kappa}} = \bigcup\limits_{\vect{\ell}\in\vect{\mathfrak{L}}^{(\vect{m}^{\sharp}, \, \vect{m}^{\flat})}_{\vect{\kappa}}}\vect{\ell}([0,2\pi)).$
\end{enumerate}
\end{theorem}

\vspace{-0.3em}

\begin{proof} Evaluating 
the Lissajous curves $\vect{\ell}^{(\vect{m})}_{(2\rho_1m^{\sharp}_1+\kappa_1,\ldots,2\rho_{\mathsf{d}}m^{\sharp}_{\mathsf{d}}+\kappa_{\mathsf{d}})}$ at the sampling
points $t^{(\vect{m})}_{l}$ given in \eqref{1608311013} and using \eqref{1509221526},  we obtain $ \vect{\ell}^{(\vect{m})}_{(2\rho_1m^{\sharp}_1+\kappa_1,\ldots,2\rho_{\mathsf{d}}m^{\sharp}_{\mathsf{d}}+
\kappa_{\mathsf{d}})}(t^{(\vect{m})}_{l})=z^{(\vect{m})}_{\vect{i}}.$
Therefore, Proposition \ref{1509221521} yields the characterization a).

Now, let $\vect{x}\in \mathcal{C}^{(\vect{m})}_{\vect{\kappa}}$. We choose $\theta_{\mathsf{i}}\in [0,\pi]$ and $t'\in \mathbb{R}$ such that 
$x_{\mathsf{i}}=\cos(\theta_{\mathsf{i}})$ and $(-1)^{\kappa_{\mathsf{i}}} T_{m_{\mathsf{i}}}(x_{\mathsf{i}})=\cos(\p[\vect{m}^{\sharp}]t')$ for all $\mathsf{i}$. 
Then, there are $\vect{v}\in \{-1,1\}^{\mathsf{d}}$, $\vect{h}'\in\mathbb{Z}$, such 
that $m_{\mathsf{i}}\theta_{\mathsf{i}}+\kappa_{\mathsf{i}}\pi= v_{\mathsf{i}}\p[\vect{m}^{\sharp}]t'+2 h'_{\mathsf{i}}\pi$ 
for all $\mathsf{i}$. Using this and \eqref{1509091200D}, we obtain
\[x_{\mathsf{i}}=\cos((\lcm[\vect{m}]t'+2h_{\mathsf{i}}\pi -\kappa_{\mathsf{i}}\pi)/m_{\mathsf{i}})\quad \text{for all $\mathsf{i}$},\] 
where $h_{\mathsf{i}}=v_{\mathsf{i}}h'_{\mathsf{i}}+(1-v_{\mathsf{i}})\kappa_{\mathsf{i}}/2$. By assumption \eqref{1509091200C} and the Chinese remainder theorem, we can find an $l$ such that $l\equiv 2h_{\mathsf{i}}\tmod 2m^{\sharp}_{\mathsf{i}}$.
Moreover, by  \eqref{1509091200A}, we can find a $\rho_{\mathsf{i}}\in \{0,\ldots,m^{\flat}_{\mathsf{i}}-1\}$ such that 
$l\equiv 2h_{\mathsf{i}}-2\rho_{\mathsf{i}}m^{\sharp}_{\mathsf{i}}\tmod 2m_{\mathsf{i}}$ for all $\mathsf{i}$.
Then, for  $t=t'+l\pi/\lcm[\vect{m}]$  
we get $x_{\mathsf{i}}=\cos((\lcm[\vect{m}]t-(\kappa_{\mathsf{i}}+2\rho_{\mathsf{i}}m^{\sharp}_{\mathsf{i}})\pi)/m_{\mathsf{i}})$   for all $\mathsf{i}$.
Therefore
\[ \vect{x} \in \vect{\ell}^{(\vect{m})}_{(2\rho_1m^{\sharp}_1+\kappa_1,\ldots,2\rho_{\mathsf{d}}m^{\sharp}_{\mathsf{d}}+\kappa_{\mathsf{d}})}(\mathbb{R})=
\vect{\ell}^{(\vect{m})}_{(2\rho_1m^{\sharp}_1+\kappa_1,\ldots,2\rho_{\mathsf{d}}m^{\sharp}_{\mathsf{d}}+\kappa_{\mathsf{d}})}([0,2\pi)).\] 
The relation $\vect{\ell}^{(\vect{m})}_{(2\rho_1m^{\sharp}_1+\kappa_1,\ldots,2\rho_{\mathsf{d}}m^{\sharp}_{\mathsf{d}}+\kappa_{\mathsf{d}})}([0,2\pi))\subseteq \mathcal{C}^{(\vect{m})}_{\vect{\kappa}}$ is easily verified
by inserting the curve in the definition \eqref{1509222009} of the 
Chebyshev variety $\mathcal{C}^{(\vect{m})}_{\vect{\kappa}}$.
\end{proof}

We are now going to refine the statement in Proposition \ref{thm:decompositionlissajous},~b). 
To this end, we introduce the equivalence relation \[\text{$\vect{\ell}' \simeq \vect{\ell}$\qquad if\quad $\vect{\ell}'(\mathbb{R}) = \vect{\ell}(\mathbb{R})$}\] 
for the set $\vect{\mathfrak{L}}^{(\vect{m}^{\sharp}, \, \vect{m}^{\flat})}_{\vect{\kappa}}$ given in \eqref{1609010824}, and 
we write $[\vect{\mathfrak{L}}^{(\vect{m}^{\sharp}, \, \vect{m}^{\flat})}_{\vect{\kappa}}]$ for the respective set of equivalence classes.
Note that for all $\vect{\ell} \in \vect{\mathfrak{L}}^{(\vect{m}^{\sharp}, \, \vect{m}^{\flat})}_{\vect{\kappa}}$ we have $\vect{\ell}(\mathbb{R})=\vect{\ell}([0,2\pi))$, 
since by \cite[Theorem~1.1]{DenckerErb2015a} the fundamental period of the  Lissajous curve $\vect{\ell}$ is $2\pi$. 

According to the definition in \cite{DenckerErb2015a}, we call a curve $\vect{\ell}\in \vect{\mathfrak{L}}^{(\vect{m}^{\sharp}, \, \vect{m}^{\flat})}_{\vect{\kappa}}$ {\it degenerate} if 
there exist $t'\in\mathbb{R}$, $\vect{u}\in\{-1,1\}^{\mathsf{d}}$ such that  
$\vect{\ell}(\,\cdot\,-t') = \vect{\ell}^{(\vect{m})}_{\vect{0},\vect{u}}$
and {\it non-degenerate} otherwise.

Let $\vect{m},\vect{m}^{\sharp},\vect{m}^{\flat}\in\mathbb{N}^{\mathsf{d}}$  satisfy  \eqref{1509091200A}, 
\eqref{1509091200B}, \eqref{1509091200C} and \eqref{1509091200D}.
By the Chinese remainder theorem there exists an unique element $l^{\dagger}\in \{0, \ldots, \p[\vect{m}^{\sharp}]-1\}$ satisfying
\begin{equation}\label{1708242340}
\forall\,\mathsf{i} \in\{1,\ldots,\mathsf{d}\}:\quad  l^{\dagger}\equiv \kappa_{\mathsf{i}} \mod m^{\sharp}_{\mathsf{i}}.
\end{equation}
For $\vectrho\in\vect{R}^{(\vect{m}^{\flat})}$, we define a corresponding element $\vectrho^{\dagger}=(\rho^{\dagger}_1,\ldots,\rho^{\dagger}_{\mathsf{d}})\in \vect{R}^{(\vect{m}^{\flat})}$ by
\begin{equation}\label{1708242341}
 \forall\,\mathsf{i} \in\{1,\ldots,\mathsf{d}\}:\quad \rho^{\dagger}_{\mathsf{i}}\equiv (l^{\dagger}-\kappa_{\mathsf{i}})/m^{\sharp}_{\mathsf{i}}-\rho_{\mathsf{i}}\mod m^{\flat}_{\mathsf{i}}.
\end{equation}

\begin{proposition}\label{201708201848} Let  $\vect{m},\vect{m}^{\sharp},\vect{m}^{\flat}\in\mathbb{N}^{\mathsf{d}}$  satisfy \eqref{1509091200A}, 
\eqref{1509091200B}, \eqref{1509091200C} and \eqref{1509091200D}.

The unique element $\vectrho^{\dagger}\in \vect{R}^{(\vect{m}^{\flat})}$ corresponding to $\vectrho\in\vect{R}^{(\vect{m}^{\flat})}$ can be characterized as follows.
For $\vectrho'\in\vect{R}^{(\vect{m}^{\flat})}$, we have $\vectrho'=\vectrho^{\dagger}$ if and only if there is some  $k\in\mathbb{Z}$ with
\begin{equation}\label{1708201840}
\forall\,\mathsf{i} \in\{1,\ldots,\mathsf{d}\}:\quad   k-2\rho'_{\mathsf{i}}m^{\sharp}_{\mathsf{i}}-\kappa_{\mathsf{i}}\equiv 2\rho_{\mathsf{i}}m^{\sharp}_{\mathsf{i}}+\kappa_{\mathsf{i}}\mod 2m_{\mathsf{i}}.
\end{equation}
\end{proposition}
\begin{proof} If $\vectrho'=\vectrho^{\dagger}$, then \eqref{1708201840} with $k=2l^{\dagger}$. Suppose \eqref{1708201840}. 
Then, $k$ is even, and \eqref{1708242341}, \eqref{1708201840} imply 
$k/2-l^{\dagger}\equiv (\rho'_{\mathsf{i}}-\rho^{\dagger}_{\mathsf{i}})m^{\sharp}_{\mathsf{i}}\tmod m_{\mathsf{i}}$ for all $\mathsf{i}$. 
Lemma \ref{1708311652} now yields $\vectrho'=\vectrho^{\dagger}$.
\end{proof}

\begin{theorem} \label{201708201819} Let  $\vect{m},\vect{m}^{\sharp},\vect{m}^{\flat}\in\mathbb{N}^{\mathsf{d}}$  satisfy \eqref{1509091200A}, 
\eqref{1509091200B}, \eqref{1509091200C} and \eqref{1509091200D}.

\begin{enumerate}[a)]
\item  For $\vectrho,\vectrho'\in\vect{R}^{(\vect{m}^{\flat})}$, we have 
\begin{align}  \label{1710311443}\vect{\ell}^{(\vect{m})}_{(2\rho'_1m^{\sharp}_1+\kappa_1,\ldots,2\rho'_{\mathsf{d}}m^{\sharp}_{\mathsf{d}}+\kappa_{\mathsf{d}})}&=\vect{\ell}^{(\vect{m})}_{(2\rho_1m^{\sharp}_1+\kappa_1,\ldots,2\rho_{\mathsf{d}}m^{\sharp}_{\mathsf{d}}+\kappa_{\mathsf{d}})}\quad\Longleftrightarrow\quad \vectrho'=\vectrho,\\
\label{1708221847} \vect{\ell}^{(\vect{m})}_{(2\rho'_1m^{\sharp}_1+\kappa_1,\ldots,2\rho'_{\mathsf{d}}m^{\sharp}_{\mathsf{d}}+\kappa_{\mathsf{d}})}&\simeq\vect{\ell}^{(\vect{m})}_{(2\rho_1m^{\sharp}_1+\kappa_1,\ldots,2\rho_{\mathsf{d}}m^{\sharp}_{\mathsf{d}}+\kappa_{\mathsf{d}})}\quad\Longleftrightarrow\quad \vectrho'\in\{\vectrho,\vectrho^\dagger\}.
\end{align}

\item  For $\vectrho\in\vect{R}^{(\vect{m}^{\flat})}$, the curve
$\vect{\ell}^{(\vect{m})}_{\left(2\rho_1m^{\sharp}_1+\kappa_1,\ldots,2\rho_{\mathsf{d}}m^{\sharp}_{\mathsf{d}}+\kappa_{\mathsf{d}}\right)}$ is degenerate $\Longleftrightarrow$ $\vectrho^\dagger=\vectrho.$\\
The classes in $[\vect{\mathfrak{L}}^{(\vect{m}^{\sharp}, \, \vect{m}^{\flat})}_{\vect{\kappa}}]$ of the degenerate curves contain precisely one element, 
the classes of the non-degenerate curves contain precisely two elements.
\item  We have $\#\vect{\mathfrak{L}}^{(\vect{m}^{\sharp}, \, \vect{m}^{\flat})}_{\vect{\kappa}}=\p[\vect{m}^{\flat}]$, and
the number  of degenerate curves 
in  $\vect{\mathfrak{L}}^{(\vect{m}^{\sharp}, \, \vect{m}^{\flat})}_{\vect{\kappa}}$ is 
\begin{equation}\label{1708231441} N_{\mathrm{deg}}  = \dfrac1{2}\#\I^{(\vect{m})}_{\vect{\kappa},\emptyset} = \left\{ 
\begin{array}{cl}
1 & \text{if $ M_0 = \emptyset$}, \\
2^{\# (K_0 \cap M_0) - 1} & \text{if $K_{0} \cap M_0 \neq \emptyset$ and $K_1 \cap M_0 = \emptyset$}, \\
2^{\# (K_1 \cap M_0) - 1} & \text{if $K_{0} \cap M_0 = \emptyset$ and $K_1 \cap M_0 \neq \emptyset$}, \\
0 & \text{if $K_{0} \cap M_0 \neq \emptyset$ and $K_1 \cap M_0 \neq \emptyset$},
\end{array} \right.
\end{equation}
where $\I^{(\vect{m})}_{\vect{\kappa},\emptyset}$ is defined in \eqref{1609080206} with $\mathsf{M}=\emptyset$ and for $\mathfrak{r}\in\{0,1\}$ we denote 
\[
M_{\mathfrak{r}} = \{\,\mathsf{i} \in \{1,\ldots, \mathsf{d}\} \ | \ m_{\mathsf{i}} \equiv \mathfrak{r} \tmod 2 \,\},\quad
K_{\mathfrak{r}} = \{\,\mathsf{i} \in \{1,\ldots, \mathsf{d}\} \ | \ \; \kappa_{\mathsf{i}} \equiv \mathfrak{r} \tmod 2 \,\}
.\]
\item Using \eqref{1708231441}, the cardinality of $[\vect{\mathfrak{L}}^{(\vect{m}^{\sharp}, \, \vect{m}^{\flat})}_{\vect{\kappa}}]$ is 
$\# [\vect{\mathfrak{L}}^{(\vect{m}^{\sharp}, \, \vect{m}^{\flat})}_{\vect{\kappa}}]  = \dfrac12(\p[\vect{m}^{\flat}]+N_{\mathrm{deg}}).$
\end{enumerate}
\end{theorem}

\begin{proof} Let $\mathsf{d}=1$. Then, $m^\flat=1$, $\vect{R}^{(\vect{m}^{\flat})}=\{\vect{0}\}$, and
 $\vect{\mathfrak{L}}^{(\vect{m}^{\sharp}, \, \vect{m}^{\flat})}_{\vect{\kappa}}$ contains precisely one element that is degenerate. Thus, a)-d) are trivially satisfied.  We suppose  $\mathsf{d}\geq 2$.

\medskip

a) Using \eqref{1708201840} with $\vectrho'=\vectrho^{\dagger}$, $\cos(-s)=\cos s$, and $t'=-t^{(\vect{m})}_{k}$, we obtain 
\[\vect{\ell}^{(\vect{m})}_{(2\rho^{\dagger}_1m^{\sharp}_1+\kappa_1,\ldots,2\rho^{\dagger}_{\mathsf{d}}m^{\sharp}_{\mathsf{d}}+\kappa_{\mathsf{d}})}\!(t-t')=
\vect{\ell}^{(\vect{m})}_{(-2\rho_1m^{\sharp}_1-\kappa_1,\ldots,-2\rho_{\mathsf{d}}m^{\sharp}_{\mathsf{d}}-\kappa_{\mathsf{d}})}\!(t) = 
\vect{\ell}^{(\vect{m})}_{(2\rho_1m^{\sharp}_1+\kappa_1,\ldots,2\rho_{\mathsf{d}}m^{\sharp}_{\mathsf{d}}+\kappa_{\mathsf{d}})}\!(-t).\]
Thus, we have the statement ``$\Longleftarrow$'' of \eqref{1708221847}. We show the statement ``$\Longrightarrow$'' of \eqref{1708221847}.\\
Let $\mathsf{d}\geq 2$, $\vectrho,\vectrho'\in\vect{R}^{(\vect{m}^{\flat})}$, and suppose the equivalence on the left hand side of~\eqref{1708221847}. 
We use mathematical induction for $\mathsf{d}\geq 2$ and  start with the base case $\mathsf{d}=2$. 

Let $(\rho_1,\rho_2),(\rho'_1,\rho'_2)\in\vect{R}^{(m^{\flat}_1,m^{\flat}_2)}$, $\vect{\ell}=\vect{\ell}^{(m_1,m_2)}_{(2\rho_1m^{\sharp}_1+\kappa_1,2\rho_2m^{\sharp}_2+\kappa_2)}$, $\vect{\ell}'=\vect{\ell}^{(m_1,m_2)}_{(2\rho'_1m^{\sharp}_1+\kappa_1,2\rho'_2m^{\sharp}_2+\kappa_2)}$.

We assume $\vect{\ell}'\equiv\vect{\ell}.$ Then, 
in particular, for $h\in\mathbb{Z}$ there exists $t'_h\in\mathbb{R}$ such that $\vect{\ell}'(t'_h)=\vect{\ell}(t^{(\vect{m})}_h).$ As in the proof of \cite[Theorem 1.4]{DenckerErb2015a}, we conclude that in this case $t'_h$ has to be of the form  $t'_h=t^{(\vect{m})}_{k_h}$ for some $k_h\in\mathbb{Z}$. Further, for $k_h\in\mathbb{Z}$, we have $\vect{\ell}'(t^{(\vect{m})}_{k_h})=\vect{\ell}(t^{(\vect{m})}_h)$  if and only if there exists  $(v_{h,1},v_{h,2})\in\{-1,1\}^2$ such that
\begin{equation}\label{1708221931}
\begin{split}
 k_h-2\rho'_1m^{\sharp}_1-\kappa_1&\equiv v_{h,1}(h-2\rho_1m^{\sharp}_1-\kappa_1) \mod 2m_1,\\
 k_h-2\rho'_2m^{\sharp}_2-\kappa_2&\equiv v_{h,2}(h-2\rho_2m^{\sharp}_2-\kappa_2) \mod 2m_2.
\end{split}
\end{equation}
 Let $(v_{h,1},v_{h,2})=(-1,-1)$. Then \eqref{1708221931} yields \eqref{1708201840} with $k=k_h+h$. Hence, Proposition~\ref{201708201848} implies $(\rho'_1,\rho'_2)=(\rho^\dagger_1,\rho^\dagger_2)$. If $(v_{h,1},v_{h,2})=(1,1)$,  then Lemma \ref{1708311652} implies $(\rho'_1,\rho'_2)=(\rho_1,\rho_2)$.
Now, we assume $v_{h,2}\neq v_{h,1}$ for all $h\in\mathbb{Z}$, and consider two cases.

(i) Let $(v_{h'+1,1},v_{h'+1,2})=  (v_{h',1},v_{h',2})$ for some $h'\in\mathbb{Z}$. Using \eqref{1708221931}, we conclude
\[k_{h'+1}-k_{h'}\equiv v_{h',1}\mod 2m_1,\qquad
 k_{h'+1}-k_{h'}\equiv  v_{h',2}=-v_{h',1}\mod 2m_2,\]
hence $-1\equiv 1 \tmod\lcm\{2m_1,2m_2\}$, and, thus, $\lcm\{m_1,m_2\}=1$, $\vect{m}^{\flat}=(1,1)$, $\vect{R}^{(\vect{m}^{\flat})}=\{(0,0)\}$. Therefore, we trivially have $(\rho'_1,\rho'_2)=(0,0)=(\rho_1,\rho_2)$. 

(ii) Let $(v_{h+1,1},v_{h+1,2})\neq (v_{h,1},v_{h,2})$ for all $h$. Then, since $v_{h,2}\neq v_{h,1}$ for all $h\in\mathbb{Z}$, 
we have $(v_{h+1,1},v_{h+1,2})= (-v_{h,1},-v_{h,2})$ for all $h$. Thus, in particular we have   \[(v_{2,1},v_{2,2})=(-v_{1,1},-v_{1,2})=(v_{0,1},v_{0,2}).\] 
As in case i),  using \eqref{1708221931},  we get  $-2\equiv 2 \tmod\lcm\{2m_1,2m_2\}$. Therefore, we have $\lcm\{m_1,m_2\}\in\{1,2\}.$ If  $\lcm\{m_1,m_2\}=1$,  
then  $(\rho'_1,\rho'_2)=(0,0)=(\rho_1,\rho_2)$. Let $\lcm\{m_1,m_2\}=2$, i.e. $\vect{m}^{\flat}\in\{(1,2),(2,1)\}$, and without restriction~$\vect{m}^{\flat}=(1,2)$.  
We have that $m^{\sharp}_2$ is odd and $\vect{R}^{(\vect{m}^{\flat})}=\{(0,0),(0,1)\}$. 
It is easy to verify the following: if $\kappa_2\equiv\kappa_1\tmod 2$, then  $(0,0)^{\dagger}=(0,0)$, $(0,1)^{\dagger}=(0,1)$, and the 
curves $\vect{\ell}^{(m_1,m_2)}_{(2\rho_1m^{\sharp}_1+\kappa_1,2\rho_2m^{\sharp}_2+\kappa_2)}$ corresponding to $\vectrho=(0,0)$ and to
$\vectrho=(0,1)$ are not equivalent; if  $\kappa_2\not\equiv\kappa_1\tmod 2$, then  $(0,0)^{\dagger}=(0,1)$ and $(0,1)^{\dagger}=(0,0)$. 

\medskip

In total, we obtain in all cases that $(\rho'_1,\rho'_2)\in\{(\rho_1,\rho_2),(\rho^\dagger_1,\rho^\dagger_2)\}$ if $\vect{\ell}'=\vect{\ell}$, i.e. the assertion for $\mathsf{d}=2$. 
We assume that the assertion holds for $\mathsf{d}-1$ instead of  $\mathsf{d}\geq 3$.

We suppose that the equivalence on the left hand side of~\eqref{1708221847} holds for $\vectrho'\neq\vectrho$. We will proof the induction hypothesis by showing 
$\vectrho'=\vectrho^{\dagger}$. Since $\vectrho'\neq\vectrho$, we have $\rho'_{\mathsf{i}}\neq\rho_{\mathsf{i}}$ for some~$\mathsf{i}$. Assume without restriction that $\rho'_2\neq\rho_2$. 
Further, 
we consider $(\rho_1,\ldots,\rho_{\mathsf{d}-1})^{\dagger}$ corresponding to $(\rho_1,\ldots,\rho_{\mathsf{d}-1})$ according to \eqref{1708242340}, \eqref{1708242341} with  $(m_1,\ldots,m_{\mathsf{d}-1})$,  $(m^{\sharp}_1,\ldots,m^{\sharp}_{\mathsf{d}-1})$, $(m^{\flat}_1,\ldots,m^{\flat}_{\mathsf{d}-1})$,  $(\kappa_1,\ldots,\kappa_{\mathsf{d}-1})$ instead of
$\vect{m},\vect{m}^{\sharp}$, $\vect{m}^{\flat}$, $\vect{\kappa}$. 
By Proposition~\ref{201708201848},  we have  \eqref{1708242341} for every $l^{\dagger}\in\mathbb{Z}$ satisfying~\eqref{1708242340}. We conclude
\begin{equation}\label{1708231307}(\rho_1,\ldots,\rho_{\mathsf{d}-1})^{\dagger}=(\rho^{\dagger}_1,\ldots,\rho^{\dagger}_{\mathsf{d}-1})\quad\text{with}\quad (\rho^{\dagger}_1,\ldots,\rho^{\dagger}_{\mathsf{d}})=\vectrho^{\dagger}.
\end{equation}

\noindent Denoting  $\operatorname{proj}_{\mathbb{R}^{\mathsf{d}-1}}\vect{x}=(x_1,\ldots,x_{\mathsf{d}-1})$ for $\vect{x}=(x_1,\ldots,x_{\mathsf{d}-1},x_{\mathsf{d}})\in\mathbb{R}^{\mathsf{d}},$ we can write
\[ \operatorname{proj}_{\mathbb{R}^{\mathsf{d}-1}} \vect{\ell}^{(\vect{m})}_{(2\rho_1m^{\sharp}_1+\kappa_1,\ldots,2\rho_{\mathsf{d}}m^{\sharp}_{\mathsf{d}}+\kappa_{\mathsf{d}})}(t/m^{\sharp}_{\mathsf{d}})
=\vect{\ell}^{(m_1,\ldots,m_{\mathsf{d}-1})}_{(2\rho_1m^{\sharp}_1+\kappa_1,\ldots,2\rho_{\mathsf{d}-1}m^{\sharp}_{\mathsf{d}-1}+\kappa_{\mathsf{d}-1})}(t),\qquad t\in\mathbb{R},\]
and we have an analogous identity for the Lissajous curve with $\rho'_{\mathsf{i}}$ instead of $\rho_{\mathsf{i}}$.
Therefore, the equivalence on the left hand side of~\eqref{1708221847} implies 
\[\vect{\ell}^{(m_1,\ldots,m_{\mathsf{d}-1})}_{(2\rho'_1m^{\sharp}_1+\kappa_1,\ldots,2\rho'_{\mathsf{d}-1}m^{\sharp}_{\mathsf{d}-1}+\kappa_{\mathsf{d}-1})}\simeq\vect{\ell}^{(m_1,\ldots,m_{\mathsf{d}-1})}_{(2\rho_1m^{\sharp}_1+\kappa_1,\ldots,2\rho_{\mathsf{d}-1}m^{\sharp}_{\mathsf{d}-1}+\kappa_{\mathsf{d}-1})}.\]
We have $(\rho'_1,\ldots,\rho'_{\mathsf{d}-1})\in\{(\rho_1,\ldots,\rho_{\mathsf{d}-1}),(\rho^{\dagger}_1,\ldots,\rho^{\dagger}_{\mathsf{d}-1})\}$ by the
 induction assumption and by \eqref{1708231307}. Since $\rho'_2\neq\rho_2$, this implies
$(\rho'_1,\ldots,\rho'_{\mathsf{d}-1})=(\rho^{\dagger}_1,\ldots,\rho^{\dagger}_{\mathsf{d}-1})$. 
In the very same way we conclude that $(\rho'_2,\ldots,\rho'_{\mathsf{d}})=(\rho^{\dagger}_2,\ldots,\rho^{\dagger}_{\mathsf{d}})$ holds. 
Combining these two identites, we obtain $\vectrho'=\vectrho^{\dagger}$.

\medskip

b) As in the proof of \cite[Theorem 1.4]{DenckerErb2015a}, the definition $\vect{\ell}(\,\cdot\,-t') = \vect{\ell}^{(\vect{m})}_{\vect{0},\vect{u}}$ of degeneracy implies $t'$ to be of the form $t'=-t^{(\vect{m})}_{h}$ 
for some $h\in\mathbb{Z}$. Further, $\vect{\ell}(\,\cdot\,-t') = \vect{\ell}^{(\vect{m})}_{\vect{0},\vect{u}}$ with $t'=-t^{(\vect{m})}_{h}$ 
and some $\vect{u}\in\{-1,1\}^{\mathsf{d}}$ is equivalent to 
\begin{equation}\label{1708231405}
\forall\,\mathsf{i} \in\{1,\ldots,\mathsf{d}\}:\quad  h\equiv 2\rho_{\mathsf{i}}m^{\sharp}_{\mathsf{i}}+\kappa_{\mathsf{i}} \mod m_{\mathsf{i}}.
\end{equation} 
Now, if $\vectrho^\dagger=\vectrho$, then the integer $k$ in \eqref{1708201840} with $\vectrho'=\vectrho^\dagger$ is even, and
\eqref{1708231405} is satisfied with $h=k/2$. On the other hand, 
the relation \eqref{1708231405} implies \eqref{1708201840} with $\vectrho'=\vectrho$ and $k=2h.$ Therefore, by Proposition~\ref{201708201848} we obtain statement b).

\medskip

c) As in the proof of \cite[Theorem 1.4]{DenckerErb2015a}, we see: $\vect{\ell}^{(\vect{m})}_{\left(2\rho_1m^{\sharp}_1+\kappa_1,\ldots,2\rho_{\mathsf{d}}m^{\sharp}_{\mathsf{d}}+\kappa_{\mathsf{d}}\right)}(t')\in\{-1,1\}^{\mathsf{d}}$
if and only if $t'=t^{(\vect{m})}_{h}$  for some $h\in\mathbb{Z}$ satisfying \eqref{1708231405}.
In the same way as in the proof of part~b) above, we see that the image of a Lissajous  curve $\vect{\ell}\in\vect{\mathfrak{L}}^{(\vect{m}^{\sharp}, \, \vect{m}^{\flat})}_{\vect{\kappa}}$ 
contains an element of $\vect{F}^{\mathsf{d}}_{\emptyset}=\{-1,1\}^{\mathsf{d}}$ if and only if $\vect{\ell}$ is degenerate. Furthermore, we see as in the proof of by \cite[Theorem 1.4]{DenckerErb2015a} that the image of every degenerate Lissajous curve $\vect{\ell}\in\vect{\mathfrak{L}}^{(\vect{m}^{\sharp}, \, \vect{m}^{\flat})}_{\vect{\kappa}}$  contains precisely two elements of $\vect{F}^{\mathsf{d}}_{\emptyset}=\{-1,1\}^{\mathsf{d}}$. 

Let $\vectrho,\vectrho'\in\vect{R}^{(\vect{m}^{\flat})}$, $h,h'\in\mathbb{Z}$. If
\begin{equation}\label{1708231922}\vect{\ell}^{(\vect{m})}_{\left(2\rho'_1m^{\sharp}_1+\kappa_1,\ldots,2\rho'_{\mathsf{d}}m^{\sharp}_{\mathsf{d}}+\kappa_{\mathsf{d}}\right)}(t^{(\vect{m})}_{h'})=\vect{\ell}^{(\vect{m})}_{\left(2\rho_1m^{\sharp}_1+\kappa_1,\ldots,2\rho_{\mathsf{d}}m^{\sharp}_{\mathsf{d}}+\kappa_{\mathsf{d}}\right)}(t^{(\vect{m})}_{h})=\vect{u}\in\{-1,1\}^{\mathsf{d}},
\end{equation}
then, using $l_{\mathsf{i}}=(1+u_{\mathsf{i}})/2$,  we have  the congruence relations
\begin{equation}\label{1708231917}
\forall\,\mathsf{i} \in\{1,\ldots,\mathsf{d}\}:\quad  h\equiv 2\rho_{\mathsf{i}}m^{\sharp}_{\mathsf{i}}+\kappa_{\mathsf{i}}+l_{\mathsf{i}}m_{\mathsf{i}} \mod 2m_{\mathsf{i}},
\end{equation}
and the analogous  congruence relations
\begin{equation}\label{1708231916}
\forall\,\mathsf{i} \in\{1,\ldots,\mathsf{d}\}:\quad  h'\equiv 2\rho'_{\mathsf{i}}m^{\sharp}_{\mathsf{i}}+\kappa_{\mathsf{i}} + l_{\mathsf{i}}m_{\mathsf{i}} \mod 2m_{\mathsf{i}}.
\end{equation}
By Lemma \ref{1708311652}, the relations \eqref{1708231917} and \eqref{1708231916} imply $\vectrho'=\vectrho$. Thus, we have that \eqref{1708231922} implies $\vectrho'=\vectrho$. 
Combining all these statements, we obtain the first equality in \eqref{1708231441}. Furthermore, \eqref{1509221441} and \eqref{1708240203} imply the second equality in \eqref{1708231441}.

\medskip

d) The proven statements a) and b)  easily imply \eqref{1710311443} for $\vectrho,\vectrho'\in\vect{R}^{(\vect{m}^{\flat})}$. In view  of the definitions in \eqref{1609010824} and \eqref{1708242016}, we get $\#\vect{\mathfrak{L}}^{(\vect{m}^{\sharp}, \, \vect{m}^{\flat})}_{\vect{\kappa}}=\#\vect{R}^{(\vect{m}^{\flat})}=\p[\vect{m}^{\flat}]$. 
By~b), the equivalence classes of the degenerate curves in $\vect{\mathfrak{L}}^{(\vect{m}^{\sharp}, \, \vect{m}^{\flat})}_{\vect{\kappa}}$ contain 
exactly one element, the other classes contain exactly two elements. 
The  number of elements in $[\vect{\mathfrak{L}}^{(\vect{m}^{\sharp}, \, \vect{m}^{\flat})}_{\vect{\kappa}}]$ is therefore given by
$\# [\vect{\mathfrak{L}}^{(\vect{m}^{\sharp}, \, \vect{m}^{\flat})}_{\vect{\kappa}}]  = \p[\vect{m}^{\flat}]-(\p[\vect{m}^{\flat}]-N_{\mathrm{deg}})/2$.
\end{proof}

\medskip

\begin{example}\label{1710311355}
We consider the bivariate setting given in Example \ref{ex:1}.
According to Theorem \ref{thm:decompositionlissajous},~b) the Chebyshev variety $\mathcal{C}^{(2m,m)}_{(0,0)}$ can be written as the union of the images of the curves  
$\vect{\ell}^{(2m,m)}_{(0,2\rho)}$, $\rho \in \{0, \ldots, m-1\}$. Since $\vect{\ell}^{(2m,m)}_{(0,m-2\rho)}(t)= \vect{\ell}^{(2m,m)}_{(0,2\rho)}(-t)$, it 
follows that already the union 
of all curves $\vect{\ell}^{(2m,m)}_{(0,2\rho)}([0,2\pi))$, $\rho \in \{0, \ldots, \lfloor m/2 \rfloor\}$, gives the 
entire variety $\mathcal{C}^{(10,5)}_{(0,0)}$, see also Figure \ref{fig:doublem-1}. This observation
corresponds to the statement in Theorem \ref{201708201819}. Namely, we have $N_{\mathrm{deg}} = 1$ if $m$ is odd and $N_{\mathrm{deg}} = 2$ if
$m$ is even, and the total number of pairwise different sets $\vect{\ell}^{(2m,m)}_{0,2\rho}([0,2\pi))$ is given by the number of classes
$\# [\vect{\mathfrak{L}}^{((2m,1), \, (1,m))}_{(0,0)}] = (m+N_{\mathrm{deg}})/2$. 
\end{example}


\section{The spectral index sets}
\label{17008191035}

The sets we investigate in this section play a crucial role in the definition of the polynomial spaces in Section \ref{1609011843}. They describe
the spectral domain of the polynomial spaces in which the interpolation problem is solved. 

\medskip

For $\vect{m}\in\mathbb{N}^{\mathsf{d}}$, $\vect{\kappa}\in\mathbb{N}^{\mathsf{d}}$, $\mathfrak{r}\in\{0,1\}$, we define the sets
\begin{equation}\label{1608271930}
\vect{\Gamma}^{(\vect{m})}_{\vect{\kappa},\mathfrak{r}} = 
\left\{\vectgamma\in\mathbb{N}_0^{\mathsf{d}}\left|\begin {array}{l} 
\forall\,\mathsf{i}\ \text{with}\ \kappa_{\mathsf{i}}\equiv \mathfrak{r}\tmod 2:\ 2\gamma_{\mathsf{i}}\leq m_{\mathsf{i}},\\
\forall\,\mathsf{i}\ \text{with}\  \kappa_{\mathsf{i}}\not\equiv \mathfrak{r}\tmod 2:\ 2\gamma_{\mathsf{i}}< m_{\mathsf{i}}
\end{array}\right.
\right\},
\end{equation}
and
\[\overline{\vect{\Gamma}}^{(\vect{m})}_{\vect{\kappa}}=\left\{\,\vectgamma\in\mathbb{N}_0^{\mathsf{d}}\ \left|\begin {array}{ll} 
\forall\,\mathsf{i}\in\{1,\ldots,\mathsf{d}\}:& \!\gamma_{\mathsf{i}}\leq m_{\mathsf{i}},\\
\forall\,\mathsf{i},\mathsf{j}\ \text{with}\ \mathsf{i}\neq\mathsf{j}:& \!\gamma_{\mathsf{i}}/m_{\mathsf{i}}+\gamma_{\mathsf{j}}/m_{\mathsf{j}}\leq  1,\\
\forall\,\mathsf{i},\mathsf{j}\ \text{with}\ \kappa_{\mathsf{i}}\not \equiv \kappa_{\mathsf{j}}\tmod 2:& \!(\gamma_{\mathsf{i}},\gamma_{\mathsf{j}}) \neq (m_{\mathsf{i}}/2, m_{\mathsf{j}}/2)
\end{array}\right.\right\},\]
and
\[\vectGammacircvect{m} =\left\{\,\vectgamma\in\mathbb{N}_0^{\mathsf{d}}\ \left|\begin {array}{ll} 
\forall\,\mathsf{i}\in\{1,\ldots,\mathsf{d}\}:& \!\gamma_{\mathsf{i}} < m_{\mathsf{i}},\\
\forall\,\mathsf{i},\mathsf{j}\ \text{with}\ \mathsf{i}\neq\mathsf{j}:& \!\gamma_{\mathsf{i}}/m_{\mathsf{i}}+\gamma_{\mathsf{j}}/m_{\mathsf{j}}<  1
\end{array}\right.\right\}.\]
The sets $\vect{\Gamma}^{(\vect{m})}_{\vect{\kappa},0}$ and $\vect{\Gamma}^{(\vect{m})}_{\vect{\kappa},1}$ contain, with small modifications on the boundary,
all nonnegative integer vectors $\vectgamma$ inside
the $\mathsf{d}$-dimensional rectangle $\displaystyle\bigtimes_{\substack{\vspace{-8pt}\\\mathsf{i}=1}}^{\substack{\mathsf{d}\\\vspace{-10pt}}}[0,m_{\mathsf{i}}/2]$. Note that 
$\vect{\Gamma}^{(\vect{m})}_{\vect{0},1} \subseteq \vect{\Gamma}^{(\vect{m})}_{\vect{\kappa},0} \subseteq 
\vect{\Gamma}^{(\vect{m})}_{\vect{0},0}$ and $\vect{\Gamma}^{(\vect{m})}_{\vect{0},1} \subseteq \vect{\Gamma}^{(\vect{m})}_{\vect{\kappa},1} \subseteq 
\vect{\Gamma}^{(\vect{m})}_{\vect{0},0}$, as well as 
\[\vect{\Gamma}^{(\vect{m})}_{\vect{0},1}=\vect{\Gamma}^{(\vect{m})}_{\vect{\kappa},0} \cap \vect{\Gamma}^{(\vect{m})}_{\vect{\kappa},1}\qquad\text{for all}\quad \vect{\kappa}\in\mathbb{Z}^{\mathsf{d}}.\]

\medskip

From the particular cross product structure of the sets $\vect{\Gamma}^{(\vect{m})}_{\vect{\kappa},\mathfrak{r}}$ in \eqref{1608271930}, we immediately obtain their cardinality. 
Using the explicit values from \eqref{1509221434}, we can write
\begin{equation}\label{1609170228} 
\#\vect{\Gamma}^{(\vect{m})}_{\vect{\kappa},\mathfrak{r}}=\#\I^{(\vect{m})}_{\vect{\kappa},\mathfrak{r}} = \#\LC^{(\vect{m})}_{\vect{\kappa},\mathfrak{r}}.
\end{equation}
Further, the sets $\overline{\vect{\Gamma}}^{(\vect{m})}_{\vect{\kappa}}$ and $\vectGammacircvect{m}$ contain with small modifications on the boundary,
all nonnegative integer vectors $\vectgamma$ inside
the $\mathsf{d}$-dimensional polyhedral region bounded by the hyperplanes $\{\,\vect{\xi}\in\mathbb{R}^{\mathsf{d}}\,|\,\xi_{\mathsf{i}}=0\,\}$ and the hyperplanes $\{\,\vect{\xi}\in\mathbb{R}\,|\,\xi_{\mathsf{i}}/m_{\mathsf{i}}+ \xi_{\mathsf{j}}/m_{\mathsf{j}} = 1\,\}$ with $\mathsf{j}\neq\mathsf{i}$. 
For $\mathsf{d}=3$, two such index sets are illustrated in Figure \ref{fig:MPX-3} and \ref{fig:LCpoints542}.

\medskip

When introducing later the fundamental basis for the polynomial interpolation special attention has to be given 
to the elements $\vectgamma\in\overline{\vect{\Gamma}}^{(\vect{m})}_{\vect{\kappa}}$ lying on some hyperplane $\{\,\vect{\xi}\in\mathbb{R}^{\mathsf{d}}\,|\,\xi_{\mathsf{i}}/m_{\mathsf{i}}+ \xi_{\mathsf{j}}/m_{\mathsf{j}} = 1\,\}$ with $\mathsf{j}\neq\mathsf{i}$.  To handle these elements, a particular class decomposition 
of the set $\overline{\vect{\Gamma}}^{(\vect{m})}_{\vect{\kappa}}$
turns out to be useful.  

We denote
\begin{equation}\label{1608251942}
\mathsf{K}^{(\vect{m})}(\vectgamma)=\left\{\,\mathsf{j}\in\{1,\ldots,\mathsf{d}\}\,\left|\,\gamma_{\mathsf{j}}/m_{\mathsf{j}}=\max\nolimits^{(\vect{m})}[\vectgamma]\right.\,\right\}
\end{equation}
where $\max\nolimits\nolimits^{(\vect{m})}[\vectgamma]=\max\left\{\,\gamma_{\mathsf{i}}/m_{\mathsf{i}}\,\left|\,\mathsf{i}\in\{1,\ldots,\mathsf{d}\}\right.\,\right\}$.
Further, using 
\begin{equation}\label{1608251943}
\mathfrak{s}^{(\vect{m})}_{\mathsf{j}}(\vectgamma)=(\gamma_1,\ldots,\gamma_{\mathsf{j}-1}, m_{\mathsf{j}}-\gamma_{\mathsf{j}},\gamma_{\mathsf{j}+1}, \ldots,\gamma_{\mathsf{d}}),
\end{equation}
we define the sets 
\begin{equation}\label{1608251944}
\mathfrak{S}^{(\vect{m})}(\vectgamma)=\left\{\,\mathfrak{s}^{(\vect{m})}_{\mathsf{j}}(\vectgamma)\,\left|\ \mathsf{j}\in \mathsf{K}^{(\vect{m})}(\vectgamma)\right.\,\right\}.
\end{equation}
Now, we introduce the set $\left[\overline{\vect{\Gamma}}^{(\vect{m})}_{\vect{\kappa}}\right]$ as
\begin{equation}\label{1608241504} 
\left[\overline{\vect{\Gamma}}^{(\vect{m})}_{\vect{\kappa}}\right]=\left\{\,\{\vectgamma\}\,\left|\,\vectgamma\in\vect{\Gamma}^{(\vect{m})}_{\vect{\kappa},0} \cup \vect{\Gamma}^{(\vect{m})}_{\vect{\kappa},1} \right.\,\right\}\cup \left\{\,\mathfrak{S}^{(\vect{m})}(\vecteta)\,\left|\,\vecteta\in\vect{\Gamma}^{(\vect{m})}_{\vect{0},1}\right.\,\right\}.
\end{equation}
and use the notation
\begin{equation}\label{1608291955}
\vect{\Lambda}^{(\vect{m}),1}_{\vect{\kappa}}=\vectGammacircvect{m} \cup \vect{\Gamma}^{(\vect{m})}_{\vect{\kappa},0} \cup \vect{\Gamma}^{(\vect{m})}_{\vect{\kappa},1}.
\end{equation}

\begin{proposition} \label{201512041410} The following statements hold true.
\begin{enumerate}[a)]
 \item We have
\begin{equation}\label{1608292002}  \vect{\Lambda}^{(\vect{m}),1}_{\vect{\kappa}}\subsetneqq \overline{\vect{\Gamma}}^{(\vect{m})}_{\vect{\kappa}}
\end{equation}
and
\begin{equation}\label{1608251930}
\mathfrak{S}^{(\vect{m})}(\vectgamma)\subseteq \overline{\vect{\Gamma}}^{(\vect{m})}_{\vect{\kappa}}\qquad \text{for all}\quad \vectgamma \in\overline{\vect{\Gamma}}^{(\vect{m})}_{\vect{\kappa}}.
\end{equation}
 \item For $\vecteta\in \vect{\Gamma}^{(\vect{m})}_{\vect{0},1}$, we have $\mathfrak{S}^{(\vect{m})}(\vecteta)\cap \left(\vect{\Gamma}^{(\vect{m})}_{\vect{\kappa},0} 
 \cup \vect{\Gamma}^{(\vect{m})}_{\vect{\kappa},1}\right)=\emptyset$ and \[\#\mathfrak{S}^{(\vect{m})}(\vecteta)=\#\mathsf{K}^{(\vect{m})}(\vecteta).\]
 \item For $\vectgamma\in \overline{\vect{\Gamma}}^{(\vect{m})}_{\vect{\kappa}}\setminus\left(\vect{\Gamma}^{(\vect{m})}_{\vect{\kappa},0} 
 \cup \vect{\Gamma}^{(\vect{m})}_{\vect{\kappa},1}\right)$, there is a unique $\mathsf{k}\in\{1,\ldots,\mathsf{d}\}$ such that 
\begin{equation}\label{1608300914}
\dfrac{\gamma_{\mathsf{i}}}{m_{\mathsf{i}}}< \dfrac12<\dfrac{\gamma_{\mathsf{k}}}{m_{\mathsf{k}}}.
\end{equation}
Then, for  $\vecteta = \mathfrak{s}^{(\vect{m})}_{\mathsf{k}}(\vectgamma)$ we have $\vecteta\in \vect{\Gamma}^{(\vect{m})}_{\vect{0},1}$, $\mathsf{k}\in \mathsf{K}^{(\vect{m})}(\vecteta)$,
and $\vectgamma = \mathfrak{s}^{(\vect{m})}_{\mathsf{k}}(\vecteta)$.

 \item The set $\left[\overline{\vect{\Gamma}}^{(\vect{m})}_{\vect{\kappa}}\right]$ 
is a class decomposition of $\overline{\vect{\Gamma}}^{(\vect{m})}_{\vect{\kappa}}$, and the number  of classes is
\begin{equation}\label{B1509231303}
\#\left[\overline{\vect{\Gamma}}^{(\vect{m})}_{\vect{\kappa}}\right]=\#\vect{\Gamma}^{(\vect{m})}_{\vect{\kappa},0}+\#\vect{\Gamma}^{(\vect{m})}_{\vect{\kappa},1}=\#\I^{(\vect{m})}_{\vect{\kappa},0}+\#\I^{(\vect{m})}_{\vect{\kappa},1}=\#\I^{(\vect{m})}_{\vect{\kappa}} =\#\LC^{(\vect{m})}_{\vect{\kappa}}.
\end{equation}

 \item For $\vectgamma\in \overline{\vect{\Gamma}}^{(\vect{m})}_{\vect{\kappa}}$, we 
 denote the class of $\vectgamma$ in $\left[\overline{\vect{\Gamma}}^{(\vect{m})}_{\vect{\kappa}}\right]$ by $[\vectgamma]$.\\
If $\mathsf{d}=1$, then  we have $\#[\vectgamma]=1$ for all
$\vectgamma\in\overline{\vect{\Gamma}}^{(\vect{m})}_{\vect{\kappa}}$. If $\mathsf{d}\geq2$, then  
\begin{equation}\label{1608241517} \#[\vectgamma]>1\quad \Longleftrightarrow\quad \vectgamma\notin\vect{\Lambda}^{(\vect{m}),1}_{\vect{\kappa}}\qquad\text{for all}\quad \vectgamma\in \overline{\vect{\Gamma}}^{(\vect{m})}_{\vect{\kappa}}.
\end{equation}
\end{enumerate}
\end{proposition}

\noindent 

\medskip

For $\mathsf{d}\geq 2$ there is at least one class that contains more than one element, in fact
\begin{equation}\label{1609170212}
\mathfrak{S}^{(\vect{m})}(\vect{0})=\{(m_1,0,0,\ldots,0),\,(0,m_2,0,\ldots,0),\,\ldots,\,(0,0,\ldots,0,m_{\mathsf{d}})\},
\end{equation}
and thus we have $\#\mathfrak{S}^{(\vect{m})}(\vect{0})=\mathsf{d}$. There are interesting cases in which all other classes consist of precisely one element. 
These cases are characterized in Proposition~\ref{201512151534}. 

\medskip

\begin{proof} a) By the definitions, we see that 
$\vectGammacircvect{m}\!\!,\, \vect{\Gamma}^{(\vect{m})}_{\vect{\kappa},0}, \vect{\Gamma}^{(\vect{m})}_{\vect{\kappa},1} \subseteq \overline{\vect{\Gamma}}^{(\vect{m})}_{\vect{\kappa}}$, 
and therefore $\vect{\Lambda}^{(\vect{m}),1}_{\vect{\kappa}}\subseteq \overline{\vect{\Gamma}}^{(\vect{m})}_{\vect{\kappa}}$.
Further, $\mathfrak{S}^{(\vect{m})}(\vect{0})\cap \vect{\Lambda}^{(\vect{m}),1}_{\vect{\kappa}} =\emptyset$ but 
$\mathfrak{S}^{(\vect{m})}(\vect{0})\subseteq \overline{\vect{\Gamma}}^{(\vect{m})}_{\vect{\kappa}}$. We get \eqref{1608292002}.

Now, we show \eqref{1608251930}.
Let $\vectgamma \in\overline{\vect{\Gamma}}^{(\vect{m})}_{\vect{\kappa}}$, $\mathsf{k}\in \mathsf{K}^{(\vect{m})}(\vectgamma)$, 
and  $\vectgamma' = \mathfrak{s}_{\mathsf{k}}^{(\vect{m})}(\vectgamma)$.  We show that $\vectgamma'\in \overline{\vect{\Gamma}}^{(\vect{m})}_{\vect{\kappa}}$.  
By the definition in \eqref{1608251942}, we have $\gamma_{\mathsf{i}}/m_{\mathsf{i}}\leq \gamma_{\mathsf{k}}/m_{\mathsf{k}}$ for all $\mathsf{i}$, and 
by the definitions in \eqref{1608251943}, \eqref{1608251944}, we have $\gamma'_{\mathsf{k}}=m_{\mathsf{k}}-\gamma_{\mathsf{k}}$ and $\gamma'_{\mathsf{i}}=\gamma_{\mathsf{i}}$ for $\mathsf{i}\neq \mathsf{k}$. We obtain
\begin{equation}\label{A1603131454}
\forall\,\mathsf{i}\in \{1,\ldots,\mathsf{d}\}\setminus\{\mathsf{k}\}:\quad \gamma'_{\mathsf{i}}/m_{\mathsf{i}}+\gamma'_{\mathsf{k}}/m_{\mathsf{k}}= \gamma_{\mathsf{i}}/m_{\mathsf{i}}-\gamma_{\mathsf{k}}/m_{\mathsf{k}}+1\leq 1.
\end{equation}
Further, since $\vectgamma\in\overline{\vect{\Gamma}}^{(\vect{m})}_{\vect{\kappa}}$, and since $\gamma'_{\mathsf{i}}=\gamma_{\mathsf{i}}$ for all $\mathsf{i}\neq \mathsf{k}$, we get
\begin{equation}\label{A1603131503}
\forall\,\mathsf{i},\mathsf{j}\in \{1,\ldots,\mathsf{d}\}\setminus\{\mathsf{k}\}\ \text{with}\ \mathsf{i}\neq\mathsf{j}:\quad\gamma'_{\mathsf{i}}/m_{\mathsf{i}}+\gamma'_{\mathsf{j}}/m_{\mathsf{j}}\leq 1.
\end{equation}
Assume that $\kappa_{\mathsf{i}}\not \equiv \kappa_{\mathsf{j}}\tmod 2$ 
and  $(\gamma'_{\mathsf{i}},\gamma'_{\mathsf{j}}) = (m_{\mathsf{i}}/2, m_{\mathsf{j}}/2)$. If $\mathsf{i},\mathsf{j}\in \{1,\ldots,\mathsf{d}\}\setminus\{\mathsf{k}\}$, then $(\gamma_{\mathsf{i}},\gamma_{\mathsf{j}}) = (\gamma'_{\mathsf{i}},\gamma'_{\mathsf{j}})=(m_{\mathsf{i}}/2, m_{\mathsf{j}}/2)$, $\kappa_{\mathsf{i}}\not \equiv \kappa_{\mathsf{j}}\tmod 2$,  and  $\vectgamma\in\overline{\vect{\Gamma}}^{(\vect{m})}_{\vect{\kappa}}$ imply a contradiction. If $\mathsf{j}=\mathsf{k}$, then we have $\gamma_{\mathsf{k}}=m_{\mathsf{k}}-\gamma'_{\mathsf{i}}=m_{\mathsf{k}}-m_{\mathsf{k}}/2=\gamma'_{\mathsf{k}}$, thus  $\vectgamma'=\vectgamma$, and, therefore, $(\gamma_{\mathsf{i}},\gamma_{\mathsf{k}}) = (m_{\mathsf{i}}/2, m_{\mathsf{k}}/2)$, $\kappa_{\mathsf{i}}\not \equiv \kappa_{\mathsf{k}}\tmod 2$,  and  $\overline{\vect{\Gamma}}^{(\vect{m})}_{\vect{\kappa}}$ imply a contradiction.

Thus, we can conclude that if $\kappa_{\mathsf{i}}\not \equiv \kappa_{\mathsf{j}}\tmod 2$, then $(\gamma'_{\mathsf{i}},\gamma'_{\mathsf{j}}) \neq (m_{\mathsf{i}}/2, m_{\mathsf{j}}/2)$ for all $\mathsf{i}\neq \mathsf{j}$. 
This statement combined with \eqref{A1603131454}, \eqref{A1603131503} yields exactly $\vectgamma'\in \overline{\vect{\Gamma}}^{(\vect{m})}$. 

\medskip

b) Let $\vecteta\in \vect{\Gamma}^{(\vect{m})}_{\vect{0},1}$, $\vectgamma\in  \mathfrak{S}^{(\vect{m})}(\vecteta)$.
There exists $\mathsf{k}\in \mathsf{K}^{(\vect{m})}(\vecteta)$ such that $\vectgamma = \mathfrak{s}^{(\vect{m})}_{\mathsf{k}}(\vecteta)$.
Since $\eta_{\mathsf{k}}<m_{\mathsf{k}}/2$, we have  $\gamma_{\mathsf{k}}=m_{\mathsf{k}}-\eta_{\mathsf{k}}>m_{\mathsf{k}}/2$.
Thus, $\vectgamma\notin \vect{\Gamma}^{(\vect{m})}_{\vect{\kappa},0}\cup\vect{\Gamma}^{(\vect{m})}_{\vect{\kappa},1}$.
We get 
 \begin{equation}\label{1608292115}
\{\vectgamma\}\cap \mathfrak{S}^{(\vect{m})}(\vecteta)=\emptyset\quad \text{for $\vectgamma\in \vect{\Gamma}^{(\vect{m})}_{\vect{\kappa},0}\cup\vect{\Gamma}^{(\vect{m})}_{\vect{\kappa},1}$ 
and $\vecteta\in \vect{\Gamma}^{(\vect{m})}_{\vect{0},1}$}.
 \end{equation}
Now, let $\vecteta\in \vect{\Gamma}^{(\vect{m})}_{\vect{0},1}$,  
and $\vectgamma = \mathfrak{s}^{(\vect{m})}_{\mathsf{k}}(\vecteta)$, $\vectgamma' = \mathfrak{s}^{(\vect{m})}_{\mathsf{k}'}(\vecteta)$ for two indices
$\mathsf{k}, \mathsf{k'}\in\mathsf{K}^{(\vect{m})}(\vecteta)$.
By the same argument as above, we get $\gamma_{\mathsf{k}}>m_{\mathsf{k}}/2$ and $\gamma'_{\mathsf{k}'}>m_{\mathsf{k}'}/2$. Thus, 
by a), we have  $\gamma_{\mathsf{i}}<m_{\mathsf{k}}/2$ for $\mathsf{i}\neq \mathsf{k}$ and $\gamma'_{\mathsf{j}}<m_{\mathsf{k}'}/2$ for
$\mathsf{j}\neq \mathsf{k}'$. Therefore,  $\vectgamma'=\vectgamma$ implies 
$\mathsf{k}'=\mathsf{k}$, and thus, the statement in b). 

\medskip

c) Let 

\vspace{-2em}

\begin{equation}\label{1608292236}
\vectgamma \in \overline{\vect{\Gamma}}^{(\vect{m})}_{\vect{\kappa}} \setminus \left( \vect{\Gamma}^{(\vect{m})}_{\vect{\kappa},0} \cup \vect{\Gamma}^{(\vect{m})}_{\vect{\kappa},1} \right).
\end{equation}
Since $\vectgamma \notin \vect{\Gamma}^{(\vect{m})}_{\vect{\kappa},0} \cup \vect{\Gamma}^{(\vect{m})}_{\vect{\kappa},1}$, there exists an index $\mathsf{k}$ such that
$\gamma_{\mathsf{k}}/m_{\mathsf{k}} > 1/2$. By the definition of the set $\overline{\vect{\Gamma}}^{(\vect{m})}_{\vect{\kappa}}$, we have 
$\gamma_{\mathsf{i}}/m_{\mathsf{i}}+\gamma_{\mathsf{k}}/m_{\mathsf{k}}\leq 1$ 
and we necessarily obtain $\gamma_{\mathsf{i}}/m_{\mathsf{i}} < 1/2$ for all other indices $\mathsf{i} \neq \mathsf{k}$. Thus, 
there exists a unique $\mathsf{k}$ satisfying \eqref{1608300914}. Further, for $\vecteta=\mathfrak{s}^{(\vect{m})}_{\mathsf{k}}(\vectgamma)$ 
we have $\eta_{\mathsf{k}}=m_{\mathsf{k}}-\gamma_{\mathsf{k}}<m_{\mathsf{k}}/2$ and $\eta_{\mathsf{i}}=\gamma_{\mathsf{i}}$ for $\mathsf{i}\neq \mathsf{k}$.
Therefore, $\eta_{\mathsf{k}}/m_{\mathsf{k}}<1/2$ and, since $\gamma_{\mathsf{i}}/m_{\mathsf{i}}+\gamma_{\mathsf{k}}/m_{\mathsf{k}}\leq 1$, we 
also get $\eta_{\mathsf{i}}/m_{\mathsf{i}}\leq \eta_{\mathsf{k}}/m_{\mathsf{k}}<1/2$ for all $\mathsf{i} \neq \mathsf{k}$. 
All together, we obtain $\vecteta\in \vect{\Gamma}^{(\vect{m})}_{\vect{0},1}$ and 
$\vectgamma=\mathfrak{s}^{(\vect{m})}_{\mathsf{k}}(\mathfrak{s}^{(\vect{m})}_{\mathsf{k}}(\vectgamma))=\mathfrak{s}^{(\vect{m})}_{\mathsf{k}}(\vecteta)$ and 
$\mathsf{k}\in \mathsf{K}^{(\vect{m})}(\vecteta)$, and thus, the statements in c).

\medskip

d) Next, we show that
 \begin{equation}\label{1608292116}
 \mathfrak{S}^{(\vect{m})}(\vecteta')\cap\mathfrak{S}^{(\vect{m})}(\vecteta)=\emptyset\qquad \text{for}\quad \text{$\vecteta,\vecteta'\in\vect{\Gamma}^{(\vect{m})}_{\vect{0},1}$ with $\vecteta' \neq\vecteta$}.
 \end{equation}
Let $\vectgamma\in\mathfrak{S}^{(\vect{m})}(\vecteta')\cap\mathfrak{S}^{(\vect{m})}(\vecteta)$. Since
$\vectgamma\in \mathfrak{S}^{(\vect{m})}(\vecteta)$, there exists $\mathsf{k}\in \mathsf{K}^{(\vect{m})}(\vecteta)$ 
such that $\vectgamma = \mathfrak{s}^{(\vect{m})}_{\mathsf{k}}(\vecteta)$. Since $\eta_{\mathsf{k}}<m_{\mathsf{k}}/2$, 
we have  $\gamma_{\mathsf{k}}=m_{\mathsf{k}}-\eta_{\mathsf{k}}>m_{\mathsf{k}}/2$, and further $\gamma_{\mathsf{i}}=\eta_{\mathsf{i}}<m_{\mathsf{k}}/2<\gamma_{\mathsf{k}}$ 
for $\mathsf{i}\neq \mathsf{k}$. In the same way, there exists $\mathsf{j}\in \mathsf{K}^{(\vect{m})}(\vecteta')$ such that $\vectgamma = \mathfrak{s}^{(\vect{m})}_{\mathsf{j}}(\vecteta')$
and
 $\gamma_{\mathsf{i}}=\eta'_{\mathsf{i}}<m_{\mathsf{j}}/2<\gamma_{\mathsf{j}}$ for $\mathsf{i}\neq \mathsf{j}$.
 We conclude that $\mathsf{j}=\mathsf{k}$ and
 $\vecteta'=\mathfrak{s}^{(\vect{m})}_{\mathsf{j}}(\mathfrak{s}^{(\vect{m})}_{\mathsf{j}}(\vecteta'))=\mathfrak{s}^{(\vect{m})}_{\mathsf{j}}(\vectgamma)
 =\mathfrak{s}^{(\vect{m})}_{\mathsf{k}}(\mathfrak{s}^{(\vect{m})}_{\mathsf{k}}(\vecteta))=\vecteta$. Thus, we have shown \eqref{1608292116}. 

\medskip
 
Now, combining the statements in c) with \eqref{1608251930}, \eqref{1608292115}, \eqref{1608292116}, and  \eqref{1609170228}, we obtain the statement in d).  

\medskip

e) If $\mathsf{d}=1$, we clearly have $\#[\vectgamma]=1$ for all $\vectgamma\in\overline{\vect{\Gamma}}^{(\vect{m})}_{\vect{\kappa}}$. Suppose that $\mathsf{d}\geq 2$. 
If $\vectgamma\in \vect{\Gamma}^{(\vect{m})}_{\vect{\kappa},0} \cup \vect{\Gamma}^{(\vect{m})}_{\vect{\kappa},1}$, then $\#[\vectgamma]=1$ by the definition of 
$\vectgamma\in\overline{\vect{\Gamma}}^{(\vect{m})}_{\vect{\kappa}}$ in \eqref{1608241504}.  Let
\[\vectgamma \in \vectGammacircvect{m} \setminus \left( \vect{\Gamma}^{(\vect{m})}_{\vect{\kappa},0} \cup \vect{\Gamma}^{(\vect{m})}_{\vect{\kappa},1} \right).\]
Then \eqref{1608292236} holds, and we can repeat the considerations below \eqref{1608292236}. 
But in this case, we have the stronger condition $\gamma_{\mathsf{i}}/m_{\mathsf{i}}+\gamma_{\mathsf{k}}/m_{\mathsf{k}}< 1$  for all $\mathsf{i} \neq \mathsf{k}$. 
Therefore, $\eta_{\mathsf{i}}/m_{\mathsf{i}}< \eta_{\mathsf{k}}/m_{\mathsf{k}}<1/2$  for all $\mathsf{i} \neq \mathsf{k}$. 
This implies  $\mathsf{K}^{(\vect{m})}(\vecteta)=\{\mathsf{k}\}$, and thus \[\#[\vectgamma]=\#\mathfrak{S}^{(\vect{m})}(\vecteta)=\#\mathsf{K}^{(\vect{m})}(\vecteta)=1.\]
 If $\vectgamma\in \mathfrak{S}^{(\vect{m})}(\vect{0})$, then $[\vectgamma]=\mathfrak{S}^{(\vect{m})}(\vect{0})$, and therefore $\#[\vectgamma]=\mathsf{d}\geq 2$. 
 Finally, let   
\[\vectgamma\notin\overline{\vect{\Gamma}}^{(\vect{m})}_{\vect{\kappa}}\setminus\left(\vectGammacircvect{m} \cup\mathfrak{S}^{(\vect{m})}(\vect{0})\cup \vect{\Gamma}^{(\vect{m})}_{\vect{\kappa},0} \cup \vect{\Gamma}^{(\vect{m})}_{\vect{\kappa},1}\right).\]

\vspace{-0.7em}

\noindent Again, we can repeat the considerations below \eqref{1608292236}. Since $\vectgamma\notin\vectGammacircvect{m}$, we have  $\gamma_{\mathsf{j}}/m_{\mathsf{j}}+\gamma_{\mathsf{j'}}/m_{\mathsf{j'}}= 1$ 
for some $\mathsf{j},\mathsf{j}'$ with  $\mathsf{j} \neq \mathsf{j}'$.
Since $\gamma_{\mathsf{i}}/m_{\mathsf{i}} < 1/2$  for $\mathsf{i} \neq \mathsf{k}$, we have $\mathsf{j}=\mathsf{k}$ or $\mathsf{j}'=\mathsf{k}$. Without restriction let $\mathsf{j}'=\mathsf{k}$. 
Then, $\eta_{\mathsf{j}}/m_{\mathsf{j}}=\gamma_{\mathsf{j}}/m_{\mathsf{j}}=1-\gamma_{\mathsf{k}}/m_{\mathsf{k}}=\eta_{\mathsf{k}}/m_{\mathsf{k}}$, and therefore,  $\mathsf{j},\mathsf{k}\in \mathsf{K}^{(\vect{m})}(\vecteta)$. We  conclude, using statement b),  that $\#[\vectgamma]=\#\mathfrak{S}^{(\vect{m})}(\vecteta) =\#\mathsf{K}^{(\vect{m})}(\vecteta)\geq \#\{\mathsf{j},\mathsf{k}\}=2$.
\end{proof}

In the next sections, we will also consider sets $\vect{\Gamma}^{(\vect{m})}_{\vect{\kappa}}\subseteq\overline{\vect{\Gamma}}^{(\vect{m})}_{\vect{\kappa}}$ 
of representatives of the classes in $\left[\overline{\vect{\Gamma}}^{(\vect{m})}_{\vect{\kappa}}\right]$, i.e.
sets that contain precisely one element of each class, i.e. 
\begin{equation}\label{1509301005}
\left[\overline{\vect{\Gamma}}^{(\vect{m})}_{\vect{\kappa}}\right]=\left\{\,[\vectgamma]\left|\,\vectgamma\in\vect{\Gamma}^{(\vect{m})}_{\vect{\kappa}} \right.\right\},\quad \#\vect{\Gamma}^{(\vect{m})}_{\vect{\kappa}}=\#\left[\overline{\vect{\Gamma}}^{(\vect{m})}_{\vect{\kappa}}\right]=\#\I^{(\vect{m})}_{\vect{\kappa}}.
\end{equation}
Then, Proposition \ref{201512041410} yields
$\vect{\Lambda}^{(\vect{m}),1}_{\vect{\kappa}}\subsetneqq \vect{\Gamma}^{(\vect{m})}_{\vect{\kappa}}\subseteq \overline{\vect{\Gamma}}^{(\vect{m})}_{\vect{\kappa}}$.
Further, if $\mathsf{d}\geq2$, we have $\vect{\Gamma}^{(\vect{m})}_{\vect{\kappa}} \subsetneqq \overline{\vect{\Gamma}}^{(\vect{m})}_{\vect{\kappa}}$, whereas for $\mathsf{d}=1$
we have $\vect{\Gamma}^{(\vect{m})}_{\vect{\kappa}} = \overline{\vect{\Gamma}}^{(\vect{m})}_{\vect{\kappa}}$. 
The unique element in the one-element set  $\vect{\Gamma}^{(\vect{m})}_{\vect{\kappa}}\cap \mathfrak{S}^{(\vect{m})}(\vect{0})$ will be denoted by $\vectgamma^{\ast}$, i.e.
\begin{equation}\label{1509301421}
\vect{\Gamma}^{(\vect{m})}_{\vect{\kappa}}\cap \mathfrak{S}^{(\vect{m})}(\vect{0})=\{\vectgamma^{\ast}\}.
\end{equation}

\begin{proposition} \label{201512151534}
Let $\vect{\kappa}\in\mathbb{Z}^{\mathsf{d}}$. The following statements are equivalent.
\begin{enumerate}[i)]
 \item We have $\gcd\{m_{\mathsf{i}},m_{\mathsf{j}}\}\leq 2$ for all $\mathsf{i},\mathsf{j}\in\{1,\ldots,\mathsf{d}\}$ with $\mathsf{i}\neq\mathsf{j}$.
 \item All classes in  $\left[\overline{\vect{\Gamma}}^{(\vect{m})}_{\vect{\kappa}}\right]\setminus\{\mathfrak{S}^{(\vect{m})}(\vect{0})\}$ consist of precisely one element.
\end{enumerate}
\noindent In this case, using notation \eqref{1608291955}, we can write $\overline{\vect{\Gamma}}^{(\vect{m})}_{\vect{\kappa}} = \vect{\Lambda}^{(\vect{m}),1}_{\vect{\kappa}}\cup \mathfrak{S}^{(\vect{m})}(\vect{0})$.
\end{proposition}
\noindent In the setup of Proposition \ref{201512151534} we have $\vect{\Gamma}^{(\vect{m})}_{\vect{\kappa}} = \vect{\Lambda}^{(\vect{m}),1}_{\vect{\kappa}} \cup \{\vectgamma^{\ast}\}$ for  $\vect{\Gamma}^{(\vect{m})}_{\vect{\kappa}}\subseteq\overline{\vect{\Gamma}}^{(\vect{m})}_{\vect{\kappa}}$ with \eqref{1509301005},
where $\vectgamma^{\ast}$ is the representative of the class $\mathfrak{S}^{(\vect{m})}(\vect{0})$, i.e. \eqref{1509301421}. We only have to choose
$\vectgamma^{\ast} \in \mathfrak{S}^{(\vect{m})}(\vect{0})$ in order to  specify a set of representatives.

\medskip

\begin{proof} Both, i) and ii), are true for $\mathsf{d}=1$ in any case. Suppose that $\mathsf{d}\geq 2$.
By Proposition~\ref{201512041410}, condition ii) holds true if and only if \text{$\#\mathsf{K}^{(\vect{m})}(\vecteta)=1$ for all $\vecteta\in \vect{\Gamma}^{(\vect{m})}_{\vect{0},1}\setminus\{\vect{0}\}$.}

Assume that i) holds and $\vecteta\in\vect{\Gamma}^{(\vect{m})}_{\vect{0},1}\setminus\{\vect{0}\}$ and $\#\mathsf{K}^{(\vect{m})}(\vecteta)\geq 2$.  Let $\mathsf{j},\mathsf{k}\in \mathsf{K}^{(\vect{m})}(\vecteta)$ with $\mathsf{k}\neq \mathsf{j}$. We have $\eta_{\mathsf{j}},\eta_{\mathsf{k}}>0$ and \text{$\eta_{\mathsf{k}}/m_{\mathsf{k}} = \eta_{\mathsf{j}}/m_{\mathsf{j}}$}.
Let \text{$g=\gcd\{m_{\mathsf{j}},m_{\mathsf{k}}\}$}, $\mu_{\mathsf{j}}=m_{\mathsf{j}}/g$, $\mu_{\mathsf{k}}=m_{\mathsf{k}}/g$.
By i) we have $g\leq 2$, and, therefore  $0<\eta_{\mathsf{k}}<m_{\mathsf{k}}/2=g\mu_{\mathsf{k}}/2\leq \mu_{\mathsf{k}}$. The relation $\eta_{\mathsf{k}}\mu_{\mathsf{j}}=\eta_{\mathsf{j}}\mu_{\mathsf{k}}$ and $\gcd\{\mu_{\mathsf{j}},\mu_{\mathsf{k}}\}=1$ imply that  $\eta_{\mathsf{k}}$  is an integer multiple of $\mu_{\mathsf{k}}$. This is a contradiction to $0<\eta_{\mathsf{k}}<\mu_{\mathsf{k}}$.

Now, assume that i) is not satisfied. Then, there are $\mathsf{j},\mathsf{k}$ with $\mathsf{j}\neq \mathsf{k}$ such that \text{$g=\gcd\{m_{\mathsf{j}},m_{\mathsf{k}}\}$} satisfies $g>2$. Let $\eta_j=m_{\mathsf{j}}/g$, $\eta_k=m_{\mathsf{k}}/g$, and
$\eta_{\mathsf{i}}=0$ for $\mathsf{i}\notin\{\mathsf{j},\mathsf{k}\}$. Then, we have $\eta_{\mathsf{k}}/m_{\mathsf{k}}=\eta_{\mathsf{j}}/m_{\mathsf{j}}>\eta_{\mathsf{i}}/m_{\mathsf{i}}$ for $\mathsf{i}\notin\{\mathsf{j},\mathsf{k}\}$.  We have $\vecteta\in \vect{\Gamma}^{(\vect{m})}_{\vect{0},1}$, since $g>2$. We conclude $\vecteta\in \vect{\Gamma}^{(\vect{m})}_{\vect{0},1}\setminus\{\vect{0}\}$ and $\mathsf{K}^{(\vect{m})}(\vecteta)=\{\mathsf{j},\mathsf{k}\}$ with $\mathsf{k}\neq \mathsf{j}$, in particular $\#\mathsf{K}^{(\vect{m})}(\vecteta)=2$. Therefore, if i) is not satisfied, then condition ii) is not satisfied.
\end{proof}


\section{Discrete orthogonality structure}
\label{17008191036}

For the proof of the interpolation and quadrature formulas on the Lissajous-Chebyshev node points, the key ingredient 
in this work is, as in  \cite{DenckerErb2015a}, a discrete orthogonality structure on the index set $\I^{(\vect{m})}_{\vect{\kappa}}.$ 
For $\vectgamma\in \mathbb{N}_0^{\mathsf{d}}$, we define $\dchi^{(\vect{m})}_{\indexvectgamma}:\,\I^{(\vect{m})}_{\vect{\kappa}}\to\mathbb{R}$ by
\begin{equation}\label{1509241334}
\displaystyle\dchi^{(\vect{m})}_{\indexvectgamma}(\vect{i})=\tprod_{\mathsf{i}=1}^{\mathsf{d}}\cos(\gamma_{\mathsf{i}}i_{\mathsf{i}}\pi/m_{\mathsf{i}}).
\end{equation}
We remark that the considered domain $\I^{(\vect{m})}_{\vect{\kappa}}$ depends on $\vect{\kappa}$ and therefore also the functions $\dchi^{(\vect{m})}_{\indexvectgamma}$. 
We omit the explicit indication of this dependency.

We remind that for  $\vectgamma\in\overline{\vect{\Gamma}}^{(\vect{m})}_{\vect{\kappa}}$,
the class of $\vectgamma$ in \eqref{1608241504} is denoted by 
$[\vectgamma]$, and that $\#[\vectgamma]=1$ if $\vectgamma \in\vect{\Gamma}^{(\vect{m})}_{\vect{\kappa},0}\cup\vect{\Gamma}^{(\vect{m})}_{\vect{\kappa},1}$. We 
study the dependencies of $\displaystyle\dchi^{(\vect{m})}_{\indexvectgamma}$ in terms of different representatives $\vectgamma$ of a class $[\vectgamma]$. For $\vectgamma\in \overline{\vect{\Gamma}}^{(\vect{m})}_{\vect{\kappa}}\setminus\left(\vect{\Gamma}^{(\vect{m})}_{\vect{\kappa},0} \cup \vect{\Gamma}^{(\vect{m})}_{\vect{\kappa},1}\right)$,  we define
\[\vect{\kappa}[\vectgamma]=\kappa_{\mathsf{k}}\qquad \text{where the index $\mathsf{k}$ is uniquely given by \eqref{1608300914}}.\]
For all $\vectgamma'\in [\vectgamma]$, $\vectgamma\in\overline{\vect{\Gamma}}^{(\vect{m})}_{\vect{\kappa}}$, we further introduce
\begin{equation}\label{1609080610}
\vect{\kappa}[\vectgamma,\vectgamma']=
 \left\{ \begin{array}{cl}  0,\; & \text{if}\quad \vectgamma'=\vectgamma \in\vect{\Gamma}^{(\vect{m})}_{\vect{\kappa},0}\cup\vect{\Gamma}^{(\vect{m})}_{\vect{\kappa},1},\\
   \vect{\kappa}[\vectgamma']-\vect{\kappa}[\vectgamma],\; & \text{if}\quad \vectgamma'\in [\vectgamma],\ \vectgamma\in\overline{\vect{\Gamma}}^{(\vect{m})}_{\vect{\kappa}}\setminus\left(\vect{\Gamma}^{(\vect{m})}_{\vect{\kappa},0}\cup\vect{\Gamma}^{(\vect{m})}_{\vect{\kappa},1}\right).
\end{array} \right. 
\end{equation} 
Note that for all $\vectgamma'\in [\vectgamma]$, $\vectgamma\in\overline{\vect{\Gamma}}^{(\vect{m})}_{\vect{\kappa}}$ we have
\begin{equation}\label{1609080605}
\vect{\kappa}[\vectgamma,\vectgamma']=0\qquad \text{if}\quad   \#[\vectgamma]=1.
\end{equation}

\begin{proposition}\label{1609080801}
We have  
$\dchi^{(\vect{m})}_{\indexvectgamma'}(\vect{i})=(-1)^{\vect{\kappa}[\indexvectgamma,\indexvectgamma']}\dchi^{(\vect{m})}_{\indexvectgamma}(\vect{i})$, $\vect{i}\in \I^{(\vect{m})}_{\vect{\kappa}}$, $\vectgamma'\in[\vectgamma]$, $\vectgamma\in\overline{\vect{\Gamma}}^{(\vect{m})}_{\vect{\kappa}}$.
\end{proposition}
\begin{proof} If $\#[\vectgamma]=1$, then the assertion is trivial. Let $\#[\vectgamma]>1$. Then, by  Proposition~\ref{201512041410} we have
$\vectgamma,\vectgamma'\in \overline{\vect{\Gamma}}^{(\vect{m})}_{\vect{\kappa}}\setminus\left(\vect{\Gamma}^{(\vect{m})}_{\vect{\kappa},0}\cup\vect{\Gamma}^{(\vect{m})}_{\vect{\kappa},1}\right)$, and 
$\vectgamma=\mathfrak{s}^{(\vect{m})}_{\mathsf{k}}(\vecteta)$, $\vectgamma'=\mathfrak{s}^{(\vect{m})}_{\mathsf{k}'}(\vecteta)$ with $\mathsf{k},\mathsf{k}'\in\mathsf{K}^{(\vect{m})}(\vecteta)$ for some 
$\vecteta\in \vect{\Gamma}^{(\vect{m})}_{\vect{0},1}$.
 For $\vect{i}\in\I^{(\vect{m})}_{\vect{\kappa},\mathfrak{r}}$, $\mathfrak{r}\in\{0,1\}$, we derive the identities 
\text{$\dchi^{(\vect{m})}_{\indexvectgamma'}(\vect{i})=(-1)^{r+\kappa_{\mathsf{k}'}}\dchi^{(\vect{m})}_{\indexvecteta}(\vect{i})$} and
$\dchi^{(\vect{m})}_{\indexvectgamma}(\vect{i})=(-1)^{r+\kappa_{\mathsf{k}}}\dchi^{(\vect{m})}_{\indexvecteta}(\vect{i})$, and get the assertion.
\end{proof}

For $\vect{i}\in\I^{(\vect{m})}_{\vect{\kappa}}$, using \eqref{1609080206}, we define the weights $\mathfrak{w}^{(\vect{m})}_{\vect{\kappa},\vect{i}}$ by 
\begin{equation}\label{1608311812}
\mathfrak{w}^{(\vect{m})}_{\vect{\kappa},\vect{i}}=2^{\#\mathsf{M}}/(2\p[\vect{m}]) \quad \text{if}\ \ \vect{i}\in\I^{(\vect{m})}_{\vect{\kappa},\mathsf{M}}.
\end{equation}
A measure $\omega^{(\vect{m})}_{\vect{\kappa}}$ on the power set $\mathcal{P}(\I^{(\vect{m})}_{\vect{\kappa}})$  of $\I^{(\vect{m})}_{\vect{\kappa}}$  is
well-defined by  $\omega^{(\vect{m})}_{\vect{\kappa}}(\{\vect{i}\})=\mathfrak{w}^{(\vect{m})}_{\vect{\kappa},\vect{i}}$ 
for the one-element sets $\{\vect{i}\}\in\mathcal{P}(\I^{(\vect{m})}_{\vect{\kappa}})$. For a function $h \in \mathcal{L}(\I^{(\vect{m})}_{\vect{\kappa}})$, we have
\[ \int h \, \mathrm{d}\rule{1pt}{0pt}\omega^{(\vect{m})}_{\vect{\kappa}} = \sum_{\vect{i} \in \I^{(\vect{m})}_{\vect{\kappa}}} h(\vect{i}) \ \mathfrak{w}^{(\vect{m})}_{\vect{\kappa},\vect{i}}. \]

\begin{proposition}\label{1509231358}
Let  $\vectgamma\in\mathbb{N}_0^{\mathsf{d}}$ and $\displaystyle\displaystyle\dchi^{(\vect{m})}_{\indexvectgamma}\in \mathcal{L}(\I^{(\vect{m})}_{\vect{\kappa}})$. If $\tint\displaystyle\dchi^{(\vect{m})}_{\indexvectgamma}\mathrm{d}\rule{1pt}{0pt}\omega^{(\vect{m})}_{\vect{\kappa}}\neq 0$, then

\vspace{-1em}

\begin{equation}\label{1509222116}
\text{there exists $\vect{h}\in \mathbb{N}_0^{\mathsf{d}}$ with $\gamma_{\mathsf{i}}=h_{\mathsf{i}}m_{\mathsf{i}}$, $\mathsf{i}=1,\ldots,\mathsf{d}$,  and
$\tsum_{\mathsf{i}=1}^{\mathsf{d}}h_{\mathsf{i}}\in 2\mathbb{N}_0$}.
\end{equation}

\vspace{-0.5em}

\noindent If \eqref{1509222116} is satisfied, then
$\tint\displaystyle\dchi^{(\vect{m})}_{\indexvectgamma}\mathrm{d}\rule{1pt}{0pt}\omega^{(\vect{m})}_{\vect{\kappa}}=(-1)^{\vartheta_{\vect{\kappa}}(\indexvectgamma)}$ with $\vartheta_{\vect{\kappa}}(\vectgamma)=\tsum_{\mathsf{i}=1}^{\mathsf{d}}h_{\mathsf{i}}\kappa_{\mathsf{i}}$.
\end{proposition}

\begin{proof} We use the well-known trigonometric identities
\begin{equation}\label{1509230525}
\tprod_{\mathsf{i}=1}^{\mathsf{r}}\cos(\vartheta_{\mathsf{i}})=
\dfrac1{2^{\mathsf{r}}}\tsum_{\vect{v}\in\{-1,1\}^{\mathsf{r}}}\cos(v_1\vartheta_1+\cdots+v_{\mathsf{r}}\vartheta_{\mathsf{r}}), \quad \mathsf{r} \in \mathbb{N},
\end{equation}
and
\begin{equation}\label{1509230526} 
\tsum_{l=0}^N \cos(l\vartheta-\vartheta_0)=\dfrac{\sin\left((N+1)\vartheta/2\right)\, \cos \left(N\vartheta/2-\vartheta_0\right) }{\sin(\vartheta/2)},\quad \vartheta\notin 2\pi\mathbb{Z},\  N\in\mathbb{N}_0.
\end{equation}
Using \eqref{1509230525} with $\mathsf{r} = \mathsf{d}$, we obtain
\begin{align*}
\tint\dchi^{(\vect{m})}_{\indexvectgamma}\mathrm{d}\rule{1pt}{0pt}\omega^{(\vect{m})}_{\vect{\kappa}}
&= \dfrac1{2\p[\vect{m}]}\tsum_{\mathsf{M}\subseteq \{1, \ldots,\mathsf{d}\}}2^{\#\mathsf{M}}\tsum_{\vect{i}\in\I^{(\vect{m})}_{\vect{\kappa},\mathsf{M}}}\tprod_{\mathsf{i}=1}^{\mathsf{d}}\cos(\gamma_{\mathsf{i}}i_{\mathsf{i}}\pi/m_{\mathsf{i}})\\
 &=\dfrac1{2^{\mathsf{d}}2\p[\vect{m}]}\tsum_{\vect{v}'\in \{-1,1\}^{\mathsf{d}}} \tsum_{\mathsf{M}\subseteq \{1, \ldots,\mathsf{d}\}}
2^{\#\mathsf{M}}\tsum_{\vect{i}\in\I^{(\vect{m})}_{\vect{\kappa},\mathsf{M}}}
\cos\left(\pi\tsum_{\mathsf{i}=1}^{\mathsf{d}}v'_{\mathsf{i}}\gamma_{\mathsf{i}}i_{\mathsf{i}}/m_{\mathsf{i}}\right).
\end{align*}
Let  $\vect{m}^{\sharp},\vect{m}^{\flat}\in\mathbb{N}^{\mathsf{d}}$  satisfy \eqref{1509091200A}, \eqref{1509091200B}, \eqref{1509091200C} and \eqref{1509091200D}.
We remind the statements of Proposition \ref{1509221521}. If $\vectj(l,\vectrho)=\vect{i}$, then \eqref{1509221526} is satisfied for some (not necessarily  uniquely determined) $\vect{v}=\vect{v}(l,\vectrho)\in\{-1,1\}^{\mathsf{d}}$. Hence, for all $\vect{v}'\in\{-1,1\}^{\mathsf{d}}$ we get
\[\cos\left(\pi\tsum_{\mathsf{i}=1}^{\mathsf{d}}v'_{\mathsf{i}}\gamma_{\mathsf{i}}i_{\mathsf{i}}/m_{\mathsf{i}}\right)=
\cos\left(\pi\tsum_{\mathsf{i}=1}^{\mathsf{d}}v'_{\mathsf{i}}v_{\mathsf{i}}(l,\vectrho)\gamma_{\mathsf{i}}\left(l-2\rho_{\mathsf{i}}m^{\sharp}_{\mathsf{i}}-\kappa_{\mathsf{i}}\right)/m_{\mathsf{i}}\right).\] 
Using  Proposition~\ref{1509221521},~c), we obtain for $\tint\displaystyle\dchi^{(\vect{m})}_{\indexvectgamma}\mathrm{d}\rule{1pt}{0pt}\omega^{(\vect{m})}_{\vect{\kappa}}$ the value
\[\dfrac1{2^{\mathsf{d}}2\p[\vect{m}]}\tsum_{\vect{v}\in \{-1,1\}^{\mathsf{d}}}\tsum_{\rho_1=0}^{m^{\flat}_1-1}\cdots\tsum_{\rho_{\mathsf{d}}=0}^{m^{\flat}_{\mathsf{d}}-1}\tsum_{l=0}^{2\p[\vect{m}^{\sharp}]-1}
\cos\left(\pi l Q(\vectgamma,\vect{v})-2\pi\tsum_{\mathsf{i}=1}^{\mathsf{d}} \rho_{\mathsf{i}}\dfrac{v_{\mathsf{i}}\gamma_{\mathsf{i}}}{m^{\flat}_{\mathsf{i}}}-\pi \vartheta_{\vect{\kappa}}(\vectgamma,\vect{v})\right),\]
with
\[Q(\vectgamma,\vect{v})= \tsum_{\mathsf{i}=1}^{\mathsf{d}} v_{\mathsf{i}}\gamma_{\mathsf{i}}/m_{\mathsf{i}},\qquad \vartheta_{\vect{\kappa}}(\vectgamma,\vect{v})=\tsum_{\mathsf{i}=1}^{\mathsf{d}} v_{\mathsf{i}}\gamma_{\mathsf{i}}\kappa_{\mathsf{i}}/m_{\mathsf{i}}.\]

By the identity \eqref{1509230526}, the value $\tint\displaystyle\dchi^{(\vect{m})}_{\indexvectgamma}\mathrm{d}\rule{1pt}{0pt}\omega^{(\vect{m})}_{\vect{\kappa}}$  is zero if for all  $\vect{v}\in\{-1,1\}^{\mathsf{d}}$  the number 
$Q(\vectgamma,\vect{v})$ is not an element of  $2\mathbb{Z}$, or if one of the $\gamma_{\mathsf{i}}/m^{\flat}_{\mathsf{i}}$ 
is not an element of~$\mathbb{Z}$. 
Assume that $\tint\displaystyle\dchi^{(\vect{m})}_{\indexvectgamma}\mathrm{d}\rule{1pt}{0pt}\omega^{(\vect{m})}_{\vect{\kappa}}\neq 0$. 
Then, we can conclude that  there exists $\vect{h}'\in\mathbb{N}^{\mathsf{d}}$ such that $\gamma_{\mathsf{i}}=h'_{\mathsf{i}}m^{\flat}_{\mathsf{i}}$ for all $\mathsf{i}$. Further, there exists $\vect{v}$ such that 
\begin{equation}\label{1509231215}
\tsum_{\mathsf{i}=1}^{\mathsf{d}} v_{\mathsf{i}}h'_{\mathsf{i}}/m_{\mathsf{i}}^{\sharp}=Q(\vectgamma,\vect{v})\in 2\mathbb{Z}.
\end{equation}
In particular, the value of the left hand side of \eqref{1509231215} is an element of $\mathbb{Z}$.  Since the  $m^{\sharp}_{\mathsf{i}}$, $\mathsf{i}\in \{1,\ldots,\mathsf{d}\}$, 
are pairwise relatively prime, there is a $\vect{h}\in\mathbb{N}_0^{\mathsf{d}}$ such that $h'_{\mathsf{i}}=h_{\mathsf{i}}m^{\sharp}_{\mathsf{i}}$, $\mathsf{i}\in \{1,\ldots,\mathsf{d}\}$, and further, $v_1h_1+\ldots+v_{\mathsf{d}}h_{\mathsf{d}}\in 2\mathbb{Z}$. We can conclude  \eqref{1509222116}.

Suppose \eqref{1509222116}. Then, $\vartheta_{\vect{\kappa}}(\vectgamma,\vect{v})=\tsum_{\mathsf{i}=1}^{\mathsf{d}}v_{\mathsf{i}}h_{\mathsf{i}}\kappa_{\mathsf{i}}\equiv -\vartheta_{\vect{\kappa}}(\vectgamma) \tmod 2$. Since  $Q(\vectgamma,\vect{v})\in 2\mathbb{Z}$ for all $\vect{v}$ and $\gamma_{\mathsf{i}}/m^{\flat}_{\mathsf{i}}\in \mathbb{N}_0$ for all $\mathsf{i}$, the   value of $\tint\displaystyle\dchi^{(\vect{m})}_{\indexvectgamma}\mathrm{d}\rule{1pt}{0pt}\omega^{(\vect{m})}_{\vect{\kappa}}$
equals $(-1)^{\vartheta_{\vect{\kappa}}(\indexvectgamma)}$. 
\end{proof}

In the following, we use the notation 
\begin{equation}\label{1608311800}
\mathfrak{e}(\vectgamma)=\#\{\,\mathsf{i}\in\{1,\ldots,\mathsf{d}\}\,|\,\gamma_{\mathsf{i}}>0\,\},
\end{equation}
\begin{equation}\label{1608311805}
\mathfrak{f}^{(\vect{m})}(\vectgamma) =\left\{ \begin{array}{cl} 
\#\{\,\mathsf{i}\,|\,\gamma_{\mathsf{i}}=m_{\mathsf{i}}/2\,\}-1,\; & \text{if}\;\, \exists\,\mathsf{i}\in\{1,\ldots,\mathsf{d}\}:\,\gamma_{\mathsf{i}}=m_{\mathsf{i}}/2, \\
0, \; & \text{otherwise}.
\end{array} \right. 
\end {equation}

\begin{theorem}\label{1509232032} \begin{enumerate}[a)] \item Let $\vectgamma,\vectgamma'\in\overline{\vect{\Gamma}}^{(\vect{m})}_{\vect{\kappa}}$.
We have $\tint\displaystyle\dchi^{(\vect{m})}_{\indexvectgamma}\dchi^{(\vect{m})}_{\indexvectgamma'}\mathrm{d}\rule{1pt}{0pt}\omega^{(\vect{m})}_{\vect{\kappa}}\neq 0$ 
if and only if the indices  $\vectgamma$ and $\vectgamma'$ belong to the same class  in the  decomposition \eqref{1608241504}.  In this case
\begin{equation}\label{1609080736}
\tint\displaystyle\dchi^{(\vect{m})}_{\indexvectgamma} \dchi^{(\vect{m})}_{\indexvectgamma'}\mathrm{d}\rule{1pt}{0pt}\omega^{(\vect{m})}_{\vect{\kappa}}=
(-1)^{\vect{\kappa}[\indexvectgamma,\indexvectgamma']}\|\displaystyle\dchi^{(\vect{m})}_{\indexvectgamma}\|_{\omega^{(\vect{m})}_{\vect{\kappa}}}^2 \quad  
\text{for}\quad \vectgamma'\in[\vectgamma],\ \vectgamma\in\overline{\vect{\Gamma}}^{(\vect{m})}_{\vect{\kappa}},
\end{equation}
where 

\vspace{-2em}

\begin{equation}\label{1509231405}
\|\displaystyle\dchi^{(\vect{m})}_{\indexvectgamma}\|_{\omega^{(\vect{m})}_{\vect{\kappa}}}^2   =  
 \left\{ \begin{array}{cl}  2^{-\mathfrak{e}(\indexvectgamma)+\mathfrak{f}^{(\vect{m})}(\indexvectgamma)},\; & \text{if}\quad \vectgamma \in\overline{\vect{\Gamma}}^{(\vect{m})}_{\vect{\kappa}}\setminus\mathfrak{S}^{(\vect{m})}(\vect{0}),\\
   1,\; & \text{if}\quad \vectgamma \in\mathfrak{S}^{(\vect{m})}(\vect{0}).
\end{array} \right. 
\end{equation}
\item Let $\vect{\Gamma}^{(\vect{m})}_{\vect{\kappa}}$  be a set of unique representatives corresponding to the class decomposition 
$\left[\overline{\vect{\Gamma}}^{(\vect{m})}_{\vect{\kappa}}\right]$, i.e. suppose  $\vect{\Gamma}^{(\vect{m})}_{\vect{\kappa}}\subseteq\overline{\vect{\Gamma}}^{(\vect{m})}_{\vect{\kappa}}$ 
satisfies \eqref{1509301005}.
Then, $\displaystyle\dchi^{(\vect{m})}_{\indexvectgamma}$, $\vectgamma\in \vect{\Gamma}^{(\vect{m})}_{\vect{\kappa}}$, form an  orthogonal basis of $(\mathcal{L}(\I^{(\vect{m})}_{\vect{\kappa}}),\langle\,\cdot,\cdot\,\rangle_{\omega^{(\vect{m})}_{\vect{\kappa}}})$.
\end{enumerate}
\end{theorem}
\begin{proof} We fix $\vectgamma,\vectgamma'\in \overline{\vect{\Gamma}}^{(\vect{m})}_{\vect{\kappa}}$.
Using the identity \eqref{1509230525} with $\mathsf{r}=2$, we obtain 
\begin{equation}\label{1509231402}
\displaystyle\dchi^{(\vect{m})}_{\indexvectgamma}\displaystyle\dchi^{(\vect{m})}_{\indexvectgamma'}=\dfrac1{2^{\mathsf{d}}}\tsum_{\vect{v}\in\{-1,1\}^{\mathsf{d}}}\displaystyle\dchi^{(\vect{m})}_{(|\gamma_1+v_1\gamma'_1|,\ldots,|\gamma_{\mathsf{d}}+v_{\mathsf{d}}\gamma'_{\mathsf{d}}|)}.
\end{equation}
We assume that $[\vectgamma']\neq [\vectgamma]$ and  $\tint\displaystyle\dchi^{(\vect{m})}_{\indexvectgamma}\dchi^{(\vect{m})}_{\indexvectgamma'}\mathrm{d}\rule{1pt}{0pt}\omega^{(\vect{m})}_{\vect{\kappa}}\neq 0$.  Then, by \eqref{1509231402}, there exists  $\vect{v}\in\{-1,1\}^{\mathsf{d}}$ such that
$\displaystyle\dchi^{(\vect{m})}_{(|\gamma_1+v_1\gamma'_1|,\ldots,|\gamma_{\mathsf{d}}+v_{\mathsf{d}}\gamma'_{\mathsf{d}}|)}\neq0$. Thus, by  Proposition~\ref{1509231358}, 
\begin{equation}\label{15090890727}
\text{there exists $\vect{h}\in \mathbb{N}_0^{\mathsf{d}}$ with $\gamma_{\mathsf{i}}+v_{\mathsf{i}}\gamma'_{\mathsf{i}}=h_{\mathsf{i}}m_{\mathsf{i}}$, $\mathsf{i}=1,\ldots,\mathsf{d}$,  and
$\tsum_{\mathsf{i}=1}^{\mathsf{d}}h_{\mathsf{i}}\in 2\mathbb{N}_0$}.
\end{equation}
Let $\mathsf{P}=\{\,\mathsf{i}\,|\,h_{\mathsf{i}}>0\,\}$. Since  $\vectgamma'\neq \vectgamma$, the set $\mathsf{P}$ is not empty.  For all $\mathsf{i}\notin \mathsf{P}$, we have $\gamma'_{\mathsf{i}}=\gamma_{\mathsf{i}}$.
Since $[\vectgamma']\neq [\vectgamma]$, at least one of the tupels $\vectgamma$ and $\vectgamma'$ is not an element of $\mathfrak{S}^{(\vect{m})}(\vect{0})$. Since $\vectgamma,\vectgamma'\in \overline{\vect{\Gamma}}^{(\vect{m})}_{\vect{\kappa}}$ and by the choice of
 $\mathsf{P}$ this implies  $h_{\mathsf{i}}=1$ for  $\mathsf{i}\in \mathsf{P}$ and $v_{\mathsf{i}}=1$ for  $\mathsf{i}\in \mathsf{P}$.
Since  \eqref{15090890727}, we  have that $\#\mathsf{P}$ is even, in particular
$\#\mathsf{P}\geq 2$.  

We have
$\gamma_{\mathsf{i}}\geq m_{\mathsf{i}}/2$ or $\gamma'_{\mathsf{i}}\geq m_{\mathsf{i}}/2$ for all $\mathsf{i}\in \mathsf{P}$.
Further, there is a $\mathsf{k}\in \mathsf{P}$ such that 
$\gamma_{\mathsf{k}}>m_{\mathsf{k}}/2$ or $\gamma'_{\mathsf{k}}>m_{\mathsf{k}}/2$, since otherwise 
$\gamma_{\mathsf{i}}=\gamma'_{\mathsf{i}}=m_{\mathsf{i}}/2$ for all $\mathsf{i}\in \mathsf{P}$, and thus $\vectgamma'=\vectgamma$.
Without loss of generality, we assume that $\gamma_{\mathsf{k}}>m_{\mathsf{k}}/2$. Then,
we have $\gamma_{\mathsf{i}}<m_{\mathsf{i}}/2$ for all $\mathsf{i}\neq \mathsf{k}$, since $\vectgamma\in\vect{\Gamma}^{(\vect{m})}_{\vect{\kappa}}$. Further, there exists a $\mathsf{k}'\in \mathsf{P}$ such that $\gamma'_{\mathsf{k}'} > m_{\mathsf{k}'}/2$, for  otherwise  $\mathsf{P}=\{\mathsf{k}\}$ and $\#\mathsf{P}=1$. Again,
we have  $\gamma'_{\mathsf{i}}<m_{\mathsf{k}'}/2$ for all $\mathsf{i}\neq \mathsf{k}'$,  hence $\#\mathsf{P}\leq 2$. Summarizing, we have $\mathsf{P}=\{\mathsf{k}, \mathsf{k}'\}$ with $\mathsf{k}'\neq\mathsf{k}$ and
\begin{equation}\label{1608301127}
\gamma_{\mathsf{k}}+\gamma'_{\mathsf{k}}=m_{\mathsf{k}},\qquad \gamma_{\mathsf{k}'}+\gamma'_{\mathsf{k}'}=m_{\mathsf{k}'},\qquad \gamma'_{\mathsf{i}}=\gamma_{\mathsf{i}}\quad \text{for $\mathsf{i}\notin\{\mathsf{k},\mathsf{k}'\}$}.
\end{equation}
Since $\vectgamma,\vectgamma'\in \overline{\vect{\Gamma}}^{(\vect{m})}_{\vect{\kappa}}\setminus\left(\vect{\Gamma}^{(\vect{m})}_{\vect{\kappa},0} \cup \vect{\Gamma}^{(\vect{m})}_{\vect{\kappa},1}\right)$, by Proposition \ref{201512041410} there  exists  $\vecteta,\vecteta'\in \vect{\Gamma}^{(\vect{m})}_{\vect{0},1}$ with\[\mathsf{k}\in\mathsf{K}[\vecteta],\quad \mathsf{k}'\in\mathsf{K}[\vecteta'],\quad
 \vectgamma=\mathfrak{s}^{(\vect{m})}_{\mathsf{k}}(\vecteta),\quad\vectgamma'=\mathfrak{s}^{(\vect{m})}_{\mathsf{k}'}(\vecteta').\]
From the definitions in \eqref{1608251942} and \eqref{1608251943} we get
\begin{align}
\label{1509241130A}
\gamma_{\mathsf{k}'}/m_{\mathsf{k}'}=\eta_{\mathsf{k}'}/m_{\mathsf{k}'}& \leq \ \eta_{\mathsf{k}}/m_{\mathsf{k}}=1-\gamma_{\mathsf{k}}/m_{\mathsf{k}},\\
\label{1509241130B}\gamma'_{\mathsf{k}}/m_{\mathsf{k}}= \eta'_{\mathsf{k}}/m_{\mathsf{k}}\ &\leq\eta'_{\mathsf{k}'}/m_{\mathsf{k}'}=1-\gamma'_{\mathsf{k}'}/m_{\mathsf{k}'}.
\end{align}
Since 
\[(\gamma_{\mathsf{k}}+\gamma'_{\mathsf{k}})/m_{\mathsf{k}}+(\gamma_{\mathsf{k}'}+\gamma'_{\mathsf{k}'})/m_{\mathsf{k}'}=h_{\mathsf{k}}+h_{\mathsf{k}'}=2,\]
we have equality in \eqref{1509241130A} and  in \eqref{1509241130B}. Combining this with \eqref{1608301127} implies 
\begin{equation}\label{1508311343}
\begin{split}
\gamma'_{\mathsf{k}}/m_{\mathsf{k}}&=1-\gamma_{\mathsf{k}}/m_{\mathsf{k}},\quad
\gamma'_{\mathsf{k}'}/m_{\mathsf{k}'}=1-\gamma_{\mathsf{k}'}/m_{\mathsf{k}'},\\
\gamma_{\mathsf{k}'}/m_{\mathsf{k}'}&=1-\gamma_{\mathsf{k}}/m_{\mathsf{k}},\quad
\ \gamma'_{\mathsf{k}}/m_{\mathsf{k}}=1-\gamma'_{\mathsf{k}'}/m_{\mathsf{k}'}.
\end{split}
\end{equation}
Let $g=\gcd\{m_{\mathsf{k}},m_{\mathsf{k}'}\}$ and $\mu_{\mathsf{k}}=m_{\mathsf{k}}/g$, $\mu_{\mathsf{k}'}=m_{\mathsf{k}'}/g$. Since  \eqref{1508311343} implies  that \[\gamma'_{\mathsf{k}'}\mu_{\mathsf{k}}=\gamma_{\mathsf{k}}\mu_{\mathsf{k}'},\quad \gamma'_{\mathsf{k}}\mu_{\mathsf{k}'}=\gamma_{\mathsf{k}'}\mu_{\mathsf{k}},\]
and since $\gcd\{\mu_{\mathsf{k}},\mu_{\mathsf{k}'}\}=1$, there exists  $\alpha,\beta\in\mathbb{N}_0$ satisfying
\[\gamma_{\mathsf{k}}=\alpha\mu_{\mathsf{k}},\quad \gamma_{\mathsf{k}'}=\beta\mu_{\mathsf{k}'},\quad \gamma'_{\mathsf{k}}=\beta\mu_{\mathsf{k}},\quad \gamma'_{\mathsf{k}'}=\alpha\mu_{\mathsf{k}'}. \]
Further, \eqref{1608301127} yields $\alpha+\beta=g$, and thus, \eqref{1509241130A} and \eqref{1509241130B} imply
\[\eta'_{\mathsf{k}}=\beta\mu_{\mathsf{k}}=(g-\alpha)\mu_{\mathsf{k}}=\eta_{\mathsf{k}},\quad \eta'_{\mathsf{k}'}=(g-\alpha)\mu_{\mathsf{k}'}=\beta\mu_{\mathsf{k}'}=\eta_{\mathsf{k}'}.\]
We have  $\vecteta'=\vecteta$ and, thus, $\vectgamma$ and $\vectgamma'$ belong to the same class of  \eqref{1608241504}, a contradiction.

\medskip

Now, we consider the case $\vectgamma'=\vectgamma$.  If $\vectgamma\in\mathfrak{S}^{(\vect{m})}(\vect{0})$, then $\|\displaystyle\dchi^{(\vect{m})}_{\indexvectgamma}\|_{\omega^{(\vect{m})}_{\vect{\kappa}}}^2=1$, i.e. we have \eqref{1509231405} for $\vectgamma \in \mathfrak{S}^{(\vect{m})}(\vect{0})$.
Let $\vectgamma'=\vectgamma\notin \mathfrak{S}^{(\vect{m})}(\vect{0})$
and consider the set
\begin{equation}\label{1509231403}
\left\{\,\vect{v}\in\{-1,1\}^{\mathsf{d}}\,|\,\vect{v}\ \text{satisfies \eqref{15090890727}}\right\}.
\end{equation}
Using the notation $\mathsf{N}=\{\,\mathsf{i}\,|\,\gamma_{\mathsf{i}}>0\}$, $\mathsf{M}=\{\,\mathsf{i}\,|\,\gamma_{\mathsf{i}} = m_{\mathsf{i}}/2\}$,
 $\mathsf{M}(\vect{v})=\{\,\mathsf{i}\in \mathsf{M}\,|\,v_{\mathsf{i}}=1\}$, the set \eqref{1509231403} can be reformulated as
\begin{equation}\label{1509231404}
\{\,\vect{v}\in\{-1,1\}^{\mathsf{d}}\,|\,\text{$v_{\mathsf{i}}=-1$ for all $\mathsf{i}\in \mathsf{N}\setminus \mathsf{M}$ and $\#\mathsf{M}(\vect{v})$ is even}\,\}.
\end{equation}
The number $A$ of elements in \eqref{1509231404} is given by
\[A=2^{\mathsf{d}-\#\mathsf{N}}\tsum_{\substack{m=0\\m\equiv 0\tmod 2}}^{\#\mathsf{M}}\displaystyle\binom{\#\mathsf{M}}{m}
= \left\{\begin{array}{rl} 2^{\mathsf{d}-\#\mathsf{N}}& \text{if $\mathsf{M}=\emptyset$,} \\
                           2^{\#\mathsf{M}-1} 2^{\mathsf{d}-\#\mathsf{N}}& \text{otherwise.}        
         \end{array} \right.\]
Let $\vect{v}$ be an element of \eqref{1509231404}. Using  the notation given in Proposition~\ref{1509231358}, we have
\begin{equation}\label{1609082217}
\vartheta_{\vect{\kappa}}(|\gamma_1+v_1\gamma'_1|,\ldots,|\gamma_{\mathsf{d}}+v_{\mathsf{d}}\gamma'_{\mathsf{d}}|)=\tsum_{\mathsf{i}\in \mathsf{M}(\vect{v})}\kappa_{\mathsf{i}}.
\end{equation}
In \eqref{1609082217} the sum over an empty set is as usual considered to be $0$. 
Suppose $\mathsf{M}(\vect{v})\neq \emptyset$.
 Then, since  $\mathsf{M}\neq \emptyset$, we have $\vectgamma\in \vect{\Gamma}^{(\vect{m})}_{\vect{\kappa},0} \cup \vect{\Gamma}^{(\vect{m})}_{\vect{\kappa},1}$.
 By the definition in \eqref{1608271930}, we have $\kappa_{\mathsf{i}}\equiv\kappa_{\mathsf{j}}\tmod 2$ for all $\mathsf{i},\mathsf{j}\in \mathsf{M}$.
Since  $\#\mathsf{M}(\vect{v})$ is even, the integer \eqref{1609082217} is even. Therefore, in any case \eqref{1609082217} is an even integer, and by Proposition~\ref{1509231358}, we get $ \int \displaystyle\dchi^{(\vect{m})}_{\indexvectgamma(\vect{v})} \mathrm{d}\rule{1pt}{0pt}\omega^{(\vect{m})}_{\vect{\kappa}} =1$. 
Using \eqref{1509231402}, we get $\|\displaystyle\dchi^{(\vect{m})}_{\indexvectgamma}\|_{\omega^{(\vect{m})}_{\vect{\kappa}}}^2=A/2^{\mathsf{d}}$. Therefore, we have \eqref{1509231405} for $\vectgamma\in \overline{\vect{\Gamma}}^{(\vect{m})}_{\vect{\kappa}}\setminus\mathfrak{S}^{(\vect{m})}(\vect{0})$. Proposition \ref{1609080801} and  \eqref{1509231405} immediately imply  \eqref{1609080736}.

\medskip

The just shown statement a)  implies that
$\displaystyle\dchi^{(\vect{m})}_{\indexvectgamma}\neq 0$ for all $\vectgamma\in \vect{\Gamma}^{(\vect{m})}_{\vect{\kappa}}$. 
Since the functions $\displaystyle\dchi^{(\vect{m})}_{\indexvectgamma}$, $\indexvectgamma \in \Gamma^{(\vect{m})}_{\vect{\kappa}}$, are pairwise orthogonal, 
they are linearly independent. Further, by \eqref{1509301005} we have
 $\#\vect{\Gamma}^{(\vect{m})}_{\vect{\kappa}}= \#\I^{(\vect{m})}_{\vect{\kappa}}$. We conclude statement b). 
\end{proof}


\section{Multivariate polynomial interpolation}\label{1609011843}

In this central part of the article, we connect the characterizations of the node sets $\LC^{(\vect{m})}_{\vect{\kappa}}$, and 
the results for the functions $\displaystyle\dchi^{(\vect{m})}_{\indexvectgamma}$ on $\I^{(\vect{m})}_{\vect{\kappa}}$ of the previous section to 
prove a quadrature rule and the uniqueness of polynomial interpolation on $\LC^{(\vect{m})}_{\vect{\kappa}}$.

\medskip

For $\gamma\in\mathbb{N}_0$, the univariate Chebyshev polynomials  of the first kind are defined as 
\[T_{\gamma}:\,[-1,1]\to[-1,1],\qquad T_{\gamma}(x) = \cos (\gamma \arccos x).\]
Then, for  $\vectgamma\in\mathbb{N}_0^{\mathsf{d}}$, we consider
the {\it $\mathsf{d}$-variate Chebyshev polynomials}
 \[T_{\indexvectgamma}:\,[-1,1]^{\mathsf{d}}\to[-1,1],\qquad T_{\indexvectgamma}(\vect{x})=  T_{\gamma_1}(x_1) \cdot \ldots \cdot T_{\gamma_{\mathsf{d}}}(x_{\mathsf{d}}).\]

We denote by $\Pi^{\mathsf{d}}$ the complex vector space of all $\mathsf{d}$-variate polynomial functions from $[-1,1]^{\mathsf{d}}$ to $\mathbb{C}$. 
The Chebyshev polynomials $T_{\indexvectgamma}(\vect{x})$, $\vectgamma\in\mathbb{N}_0^{\mathsf{d}}$, form an orthogonal basis of the space $\Pi^{\mathsf{d}}$
with respect to the inner product given by 
 \begin{equation} \label{1509060507}
  \langle P,Q \rangle = \frac{1}{\pi^{\mathsf{d}}} \int_{[-1,1]^{\mathsf{d}}} P(\vect{x}) \overline{Q(\vect{x})} w_{\mathsf{d}}(\vect{x})\,\mathrm{d}\vect{x}, \quad w_{\mathsf{d}}(\vect{x})  = \tprod_{\mathsf{i} = 1}^{\mathsf{d}}  \displaystyle\frac{1}{\sqrt{1-x_{\mathsf{i}}^2}}.
 \end{equation}
In  particular, we have 
\begin{equation}
\label{1609081436} 
\langle T_{\indexvectgamma}, T_{\indexvectgamma'}\rangle=0\qquad \text{if}\quad \vectgamma'\neq \vectgamma.
\end{equation}
For simplicity of notation, the norm corresponding to  \eqref{1509060507} is denoted by $\|\cdot\|$, i.e.
 \[\|P\|=\sqrt{ \langle P,P \rangle}.\]
Using the notation \eqref{1608311800}, we have 
for $\vectgamma\in\mathbb{N}_0^{\mathsf{d}}$ the identity
\begin{equation} \label{1509082017}
\|T_{\indexvectgamma}\|^2  = 2^{-\mathfrak{e}(\indexvectgamma)}.
\end{equation}

\medskip

For $\vectgamma\in\mathbb{N}_0^{\mathsf{d}}$ and for $\vect{i}\in \I^{(\vect{m})}_{\vect{\kappa}}$, we have the fundamental relation
\begin{equation}\label{1509082018}
T_{\indexvectgamma}(\vect{z}^{(\vect{m})}_{\vect{i}}) = \displaystyle\dchi^{(\vect{m})}_{\indexvectgamma}(\vect{i}).
\end{equation}
This relation gives a direct link between the Chebyshev polynomials $T_{\indexvectgamma}$, the node points $\vect{z}^{(\vect{m})}_{\vect{i}} \in \LC^{(\vect{m})}_{\vect{\kappa}}$, 
and the functions $\displaystyle\dchi^{(\vect{m})}_{\indexvectgamma}$ on $\I^{(\vect{m})}_{\vect{\kappa}}$ defined in \eqref{1509241334}.
As a first application, we can  formulate a quadrature rule formula for $\mathsf{d}$-variate polynomials.

\begin{theorem} \label{1509082004} Let $P$ be a  $\mathsf{d}$-variate polynomial function from $[-1,1]^{\mathsf{d}}$ to $\mathbb{C}$. 
 If
\[\text{$\langle P, T_{\indexvectgamma} \rangle = 0$ for all $\vectgamma \in\mathbb{N}_0^{\mathsf{d}} \setminus \{\vect{0}\}$ satisfying condition \eqref{1509222116},}\]
then 

\vspace*{-2.3em}

\[\frac{1}{\pi^{\mathsf{d}}} \int_{[-1,1]^{\mathsf{d}}} P(\vect{x}) w(\vect{x}) \,\mathrm{d}\vect{x} = \tsum_{\vect{i} \in \I^{(\vect{m})}_{\vect{\kappa}}} \mathfrak{w}^{(\vect{m})}_{\vect{\kappa},\vect{i}} P(\vect{z}^{(\vect{m})}_{\vect{i}}). \]
\end{theorem}

\medskip

\noindent The simple proof of this Theorem follows the lines of the proof of \cite[Theorem 2.10]{DenckerErb2015a}.

\medskip

We turn to multivariate interpolation problems based on the node sets $\LC^{(\vect{m})}_{\vect{\kappa}}$. 
For given $h(\vect{i}) \in \mathbb{R}$, $\vect{i} \in \I^{(\vect{m})}_{\vect{\kappa}}$, we want to find 
an polynomial $P^{(\vect{m})}_{\vect{\kappa},h}$ that satisfies
\begin{equation}\label{1509082000}
 P^{(\vect{m})}_{\vect{\kappa},h} (\vect{z}^{(\vect{m})}_{\vect{i}}) = h({\vect{i}}) \quad \text{for all}\quad \vect{i} \in \I^{(\vect{m})}_{\vect{\kappa}}.
\end{equation}
In order to make this problem uniquely solvable, we have to specify an appropriate polynomial space that has vector space 
dimension $\#\I^{(\vect{m})}_{\vect{\kappa}}$. 

Relation \eqref{1509082018} motivates the introduction of polynomial spaces 
that are spanned by basis polynomials $T_{\indexvectgamma}$. We have $\#\overline{\vect{\Gamma}}^{(\vect{m})}_{\vect{\kappa}}>\#\I^{(\vect{m})}_{\vect{\kappa}}$ if $\mathsf{d}\geq 2$. 
However, by Proposition \ref{201512041410} d) the number of classes in the decomposition \eqref{1608241504} is precisely $\#\I^{(\vect{m})}_{\vect{\kappa}}$. 
Thus, a first idea for a spectral index set is to use a set $\vect{\Gamma}^{(\vect{m})}_{\vect{\kappa}}$ that contains precisely one element from each class, i.e. 
we suppose $\vect{\Gamma}^{(\vect{m})}_{\vect{\kappa}}\subseteq\overline{\vect{\Gamma}}^{(\vect{m})}_{\vect{\kappa}}$ and \eqref{1509301005}. 
By Proposition~\ref{201512151534}, this
approach is implicitly also used for the special Lissajous-Chebyshev nodes considered in \cite{DenckerErb2015a}, and in 
particular for the Padua points. We define
\begin{equation} \label{1608271750}
\Pi^{(\vect{m})}_{\vect{\kappa}} = \vspan \left\{\, T_{\indexvectgamma}\,\left|\, \vectgamma \in \vect{\Gamma}^{(\vect{m})}_{\vect{\kappa}} \right. \right\}.
\end{equation}
Clearly, the polynomials $T_{\indexvectgamma}$, $\vectgamma \in \vect{\Gamma}^{(\vect{m})}_{\vect{\kappa}}$, 
form an orthogonal basis of the vector space $\Pi^{(\vect{m})}_{\vect{\kappa}}$ with respect to the inner product \eqref{1509060507} and we have
$\dim\Pi^{(\vect{m})}_{\vect{\kappa}}= \#\vect{\Gamma}^{(\vect{m})}_{\vect{\kappa}}=\# \I^{(\vect{m})}_{\vect{\kappa}}.$
Using the definition \eqref{1608311805} for $\mathfrak{f}^{(\vect{m})}(\vectgamma)$ and \eqref{1509301421} for $\vectgamma^{\ast}$, we define on $[-1,1]^{\mathsf{d}}\times [-1,1]^{\mathsf{d}}$ 
\[L^{(\vect{m})}_{\vect{\kappa}}(\vect{z},\vect{x}) = \tsum_{\indexvectgamma \in \vect{\Gamma}^{(\vect{m})}_{\vect{\kappa}}} 
\dfrac{2^{-\mathfrak{f}^{(\vect{m})}(\indexvectgamma)}}{\|T_{\indexvectgamma}\|^2}  
T_{\indexvectgamma}(\vect{z})T_{\indexvectgamma}(\vect{x})-T_{\indexvectgamma^{\ast}}(\vect{z})T_{\indexvectgamma^{\ast}}(\vect{x}). \]
Further, using the  weights $\mathfrak{w}^{(\vect{m})}_{\vect{\kappa},\vect{i}}$ given in \eqref{1608311812}, we introduce for $\vect{i} \in \I^{(\vect{m})}_{\vect{\kappa}}$ the functions 
\begin{equation} \label{1509082001}
 L^{(\vect{m})}_{\vect{\kappa},\vect{i}} = \mathfrak{w}^{(\vect{m})}_{\vect{\kappa},\vect{i}}L^{(\vect{m})}_{\vect{\kappa}}(\vect{z}^{(\vect{m})}_{\vect{i}},\cdot\,).
\end{equation}

\begin{theorem}\label{1509082002}  Let $\vect{\Gamma}^{(\vect{m})}_{\vect{\kappa}}$ 
be a set of unique representatives corresponding to the class decomposition $\left[\overline{\vect{\Gamma}}^{(\vect{m})}_{\vect{\kappa}}\right]$ 
given in \eqref{1608241504}.
For $h\in\mathcal{L}(\I^{(\vect{m})}_{\vect{\kappa}})$, the interpolation problem \eqref{1509082000} has a unique solution 
in the polynomial space $\Pi^{(\vect{m})}_{\vect{\kappa}}$ given by 
\begin{equation}\label{1509082003} 
P^{(\vect{m})}_{\vect{\kappa},h} = \tsum_{\vect{i} \in \I^{(\vect{m})}_{\vect{\kappa}}} h({\vect{i}}) L^{(\vect{m})}_{\vect{\kappa},\vect{i}}.
\end{equation}
Furthermore, we have $\left\{\,P^{(\vect{m})}_{\vect{\kappa},h}\,\left|\,h\in \mathcal{L}(\I^{(\vect{m})}_{\vect{\kappa}})\right.\,\right\} = \Pi^{(\vect{m})}_{\vect{\kappa}}$, and the  polynomial functions  $L^{(\vect{m})}_{\vect{\kappa},\vect{i}}$, $\vect{i} \in \I^{(\vect{m})}_{\vect{\kappa}}$, in \eqref{1509082001} form a basis of the vector space $\Pi^{(\vect{m})}_{\vect{\kappa}}$.
\end{theorem}

The polynomials  $L^{(\vect{m})}_{\vect{\kappa},\vect{i}}$ are the 
unique solution of the interpolation problem \eqref{1509082000} in $\Pi^{(\vect{m})}_{\vect{\kappa}}$ if $h$ is 
the Dirac delta function $h=\delta_{\vect{i}}$ on $\I^{(\vect{m})}_{\vect{\kappa}}$. For this reason, 
the polynomials $L^{(\vect{m})}_{\vect{\kappa},\vect{i}}$   are called the 
{\it fundamental solutions} of the interpolation problem \eqref{1509082000} in $\Pi^{(\vect{m})}_{\vect{\kappa}}$.

\bigskip

\begin{proof} For $h\in \mathcal{L}(\I^{(\vect{m})}_{\vect{\kappa}})$, we define $P^{(\vect{m})}_{\vect{\kappa},h}$ by \eqref{1509082003}. 
Since $L^{(\vect{m})}_{\vect{\kappa},\vect{i}} \in \Pi^{(\vect{m})}_{\vect{\kappa}}$, $\vect{i}\in \I^{(\vect{m})}_{\vect{\kappa}}$, also $P^{(\vect{m})}_{\vect{\kappa},h}\in \Pi^{(\vect{m})}_{\vect{\kappa}}$. For $\vect{i}\in\I^{(\vect{m})}_{\vect{\kappa}}$, we define $h_{\vect{i}}\in\mathcal{L}(\I^{(\vect{m})}_{\vect{\kappa}})$ by  $h_{\vect{i}}(\vect{i}')=L^{(\vect{m})}_{\vect{\kappa},\vect{i}}(\vect{z}_{\vect{i}'}^{(\vect{n})})$, $\vect{i}'\in\I^{(\vect{m})}_{\vect{\kappa}}$. 
Using the relation \eqref{1509082018} and the formulas \eqref{1509231405}, \eqref{1509082017}, we obtain
\[\begin{split}
h_{\vect{i}}(\vect{i}')&=\mathfrak{w}^{(\vect{m})}_{\vect{\kappa},\vect{i}} \left(\, \tsum_{\indexvectgamma \in \vect{\Gamma}^{(\vect{m})}_{\vect{\kappa}}} \dfrac{2^{-\mathfrak{f}^{(\vect{m})}(\indexvectgamma)}}{\|T_{\indexvectgamma}\|^2} T_{\indexvectgamma}(\vect{z}_{\vect{i}}^{(\vect{m})})\, T_{\indexvectgamma}(\vect{z}_{\vect{i}'}^{(\vect{m})})- T_{\indexvectgamma^{\ast}}(\vect{z}^{(\vect{m})}_{\vect{i}})\,T_{\indexvectgamma^{\ast}}(\vect{z}_{\vect{i}'}^{(\vect{m})}) \right)\\
&=\mathfrak{w}^{(\vect{m})}_{\vect{\kappa},\vect{i}} \tsum_{\indexvectgamma \in \vect{\Gamma}^{(\vect{m})}_{\vect{\kappa}}} \dfrac1{\|\dchi^{(\vect{m})}_{\indexvectgamma}\|_{\omega^{(\vect{m})}_{\vect{\kappa}}}^2}\dchi^{(\vect{m})}_{\indexvectgamma}(\vect{i})\dchi^{(\vect{m})}_{\indexvectgamma}(\vect{i}').
\end{split}\]
Now, for every $\vectgamma\in\vect{\Gamma}^{(\vect{m})}_{\vect{\kappa}}$ the Theorem \ref{1509232032} implies  
$\langle h_{\vect{i}},\dchi^{(\vect{m})}_{\indexvectgamma}\rangle_{\omega^{(\vect{m})}_{\vect{\kappa}}}=\mathfrak{w}^{(\vect{m})}_{\vect{\kappa},\vect{i}}\dchi^{(\vect{m})}_{\indexvectgamma}(\vect{i})$.
By the definition of  $\langle\,\cdot,\cdot\,\rangle_{\omega^{(\vect{m})}_{\vect{\kappa}}}$, we also have $\langle \delta_{\vect{i}},\dchi^{(\vect{m})}_{\indexvectgamma}\rangle_{\omega^{(\vect{m})}_{\vect{\kappa}}}=\mathfrak{w}^{(\vect{m})}_{\vect{\kappa},\vect{i}}\dchi^{(\vect{m})}_{\indexvectgamma}(\vect{i})$ for $\vectgamma\in\vect{\Gamma}^{(\vect{m})}_{\vect{\kappa}}$. 
Thus, since $\dchi^{(\vect{m})}_{\indexvectgamma}$, $\vectgamma\in\vect{\Gamma}^{(\vect{m})}_{\vect{\kappa}}$, form a basis of $\mathcal{L}(\I^{(\vect{m})}_{\vect{\kappa}})$,
we get $h_{\vect{i}}=\delta_{\vect{i}}$ for $\vect{i}\in \mathcal{L}(\I^{(\vect{m})}_{\vect{\kappa}})$. Hence, the function  $P^{(\vect{m})}_{\vect{\kappa},h}$ satisfies the interpolation condition \eqref{1509082000}. 

The homomorphism  $h\mapsto P^{(\vect{m})}_{\vect{\kappa},h}$  from the vector space $\mathcal{L}(\I^{(\vect{m})}_{\vect{\kappa}})$ 
into the vector space $\Pi^{(\vect{m})}_{\vect{\kappa}}$  is obviously injective.
Since $\dim \mathcal{L}(\I^{(\vect{m})}_{\vect{\kappa}}) = \#\I^{(\vect{m})}_{\vect{\kappa}}=\#\vect{\Gamma}^{(\vect{m})}_{\vect{\kappa}}=\dim \Pi^{(\vect{m})}_{\vect{\kappa}}$ by \eqref{1509301005},
this homomorphism is an isomorphism from $\mathcal{L}(\I^{(\vect{m})}_{\vect{\kappa}})$ onto  $\Pi^{(\vect{m})}_{\vect{\kappa}}$.
\end{proof}

\medskip

The definition \eqref{1608271750} of the polynomial spaces $\Pi^{(\vect{m})}_{\vect{\kappa}}$ requires the  selection 
of elements $\vectgamma$ from a class $[\vectgamma]$ in \eqref{1608241504} if this class has more than one element. 
For $\mathsf{d}\geq2$, this selection process concerns sometimes only one class, see Proposition~\ref{201512151534}.
In general, it seems however to be useful to overcome the pure arbitrariness in the selection of class elements. 
Therefore, we consider a second approach to obtain a suitable interpolation space of dimension  $\#\I^{(\vect{m})}_{\vect{\kappa}}$. We define
\[\overline{\Pi}^{(\vect{m})}_{\vect{\kappa}} = \vspan \left\{\, T_{\indexvectgamma}\,|\, \vectgamma \in \overline{\vect{\Gamma}}^{(\vect{m})}_{\vect{\kappa}}  \right\}.\]
Clearly, the Chebyshev polynomials $T_{\indexvectgamma}$, $\vectgamma \in \overline{\vect{\Gamma}}^{(\vect{m})}_{\vect{\kappa}}$, 
form an orthogonal basis of $\overline{\Pi}^{(\vect{m})}_{\vect{\kappa}}$ with respect to the inner product \eqref{1509060507}. 
Thus, this space has dimension $\#\overline{\vect{\Gamma}}^{(\vect{m})}_{\vect{\kappa}}$ and this dimension is larger 
than $\#\I^{(\vect{m})}_{\vect{\kappa}}$ if $\mathsf{d}\geq 2$.  If $\vect{\kappa}=\vect{0}$, we can pass over to a 
subspace by restricting to polynomials $P$ for which $\langle P, T_{\indexvectgamma'}\rangle$, $\vectgamma'\in [\vectgamma]$, is constant 
on each class $[\vectgamma]$. For general $\vect{\kappa}\in\mathbb{Z}^{\mathsf{d}}$, we may have to take 
into account anti-symmetries. Here, the following definition turns out to be appropriate. 
This approach was already used in the bivariate setting for the Morrow-Patterson-Xu points, see Example \ref{1609031131}.

\medskip

We use the definition \eqref{1609080610} for $\vect{\kappa}[\vectgamma,\vectgamma']$,  
and define the averaged polynomials
\begin{equation}
\label{1608252121}
 \oversim{T}_{\indexvectgamma}=\dfrac1{\#[\vectgamma]} \tsum_{\indexvectgamma'\in[\indexvectgamma]}(-1)^{\vect{\kappa}[\indexvectgamma,\indexvectgamma']}T_{\indexvectgamma'}\qquad\text{for}\quad \vectgamma\in\overline{\vect{\Gamma}}^{(\vect{m})}_{\vect{\kappa}}.
\end{equation}
By \eqref{1608241517} and \eqref{1609080605}, we have for $\vectgamma\in\overline{\vect{\Gamma}}^{(\vect{m})}_{\vect{\kappa}}$ the relation
\[\oversim{T}_{\indexvectgamma}=T_{\indexvectgamma}\qquad\Longleftrightarrow\qquad \vectgamma\in \vect{\Lambda}^{(\vect{m}),1}_{\vect{\kappa}}\quad \text{or}\quad \mathsf{d}=1.\]
In view of \eqref{1609081436}, \eqref{1509082017} and \eqref{1608252121} we have for  $\vectgamma,\vectgamma'\in\overline{\vect{\Gamma}}^{(\vect{m})}_{\vect{\kappa}}$
the identities
\begin{equation}\label{1608301335}
\langle \oversim{T}_{\indexvectgamma}, \oversim{T}_{\indexvectgamma'}\rangle=
\left\{\begin{array}{cl} 0 & \text{if $[\vectgamma']\neq [\vectgamma]$,} \\
                        (-1)^{\vect{\kappa}[\indexvectgamma,\indexvectgamma']} 2^{-\mathfrak{e}(\indexvectgamma)}  &  \text{if $[\vectgamma']= [\vectgamma]$}, \end{array} \right.
\end{equation}
and, in particular,
\begin{equation}
\label{1608301758}
\|\oversim{T}_{\indexvectgamma}\|^2  =\|T_{\indexvectgamma}\|^2  = 2^{-\mathfrak{e}(\indexvectgamma)}.
\end{equation}
Proposition \ref{1609080801} and relation \eqref{1509082018} give
for $\vectgamma\in\overline{\vect{\Gamma}}^{(\vect{m})}_{\vect{\kappa}}$ and  $\vect{i}\in \I^{(\vect{m})}_{\vect{\kappa}}$ the relation 
\begin{equation}\label{1608301730}
\oversim{T}_{\indexvectgamma}(\vect{z}^{(\vect{m})}_{\vect{i}}) = \displaystyle\dchi^{(\vect{m})}_{\indexvectgamma}(\vect{i}).
\end{equation}
Now, we define our second polynomial space for the interpolation as
\begin{equation}
\label{1608271751}
\oversim{\Pi}^{(\vect{m})}_{\vect{\kappa}} = \vspan \left\{\, \oversim{T}_{\indexvectgamma}\,|\, \vectgamma \in \overline{\vect{\Gamma}}^{(\vect{m})}_{\vect{\kappa}}  \right\}.
\end{equation}
The identities \eqref{1608301335}, \eqref{1608301758} and \eqref{B1509231303} immediately imply
\begin{equation}
\label{1608312133}
 \dim \oversim{\Pi}^{(\vect{m})}_{\vect{\kappa}}=\#\left[\overline{\vect{\Gamma}}^{(\vect{m})}_{\vect{\kappa}}\right]= \# \I^{(\vect{m})}_{\vect{\kappa}}.
\end{equation}
In view of \eqref{1608252121} and \eqref{1608301335}, the polynomial space \eqref{1608271751} can be characterized by
\[\oversim{\Pi}^{(\vect{m})}_{\vect{\kappa}}= \left\{\,P\in \overline{\Pi}^{(\vect{m})}_{\vect{\kappa}}\,\left|\ \langle P, T_{\indexvectgamma'
}\rangle = (-1)^{\vect{\kappa}[\indexvectgamma,\indexvectgamma']}\langle P, T_{\indexvectgamma}\rangle\quad\text{for all}\quad  \vectgamma'\in[\vectgamma],\ \vectgamma\in \overline{\vect{\Gamma}}^{(\vect{m})}_{\vect{\kappa}}\, \right.\right\}.\]
Using the definition \eqref{1608311805} for $\mathfrak{f}^{(\vect{m})}(\indexvectgamma)$ and \eqref{1609170212} for the set $\mathfrak{S}^{(\vect{m})}(\vect{0})$, 
we define 
\[\oversim{L}^{(\vect{m})}_{\vect{\kappa}}(\vect{z},\vect{x})=\tsum_{\indexvectgamma \in \overline{\vect{\Gamma}}^{(\vect{m})}_{\vect{\kappa}}}
\dfrac{2^{-\mathfrak{f}^{(\vect{m})}(\indexvectgamma)}}{\#[\vectgamma]\|{T}_{\indexvectgamma}\|^2}\, T_{\indexvectgamma}(\vect{z})T_{\indexvectgamma}(\vect{x}) 
- \dfrac1{\mathsf{d}} \tsum_{\indexvectgamma \in \mathfrak{S}^{(\vect{m})}(\vect{0})} T_{\indexvectgamma}(\vect{z})T_{\indexvectgamma}(\vect{x}),\]
where $\#[\vectgamma]$ denotes the number of elements of the class $[\vectgamma]$.  
For all $\vect{i} \in \I^{(\vect{m})}_{\vect{\kappa}}$, we set
\begin{equation} \label{1608270850}
 \oversim{L}^{(\vect{m})}_{\vect{\kappa},\vect{i}} =
 \mathfrak{w}^{(\vect{m})}_{\vect{\kappa},\vect{i}} \oversim{L}^{(\vect{m})}_{\vect{\kappa}}(\vect{z}^{(\vect{m})}_{\vect{i}},\,\cdot\,).
\end{equation}

\begin{proposition}  Let $\vect{\Gamma}^{(\vect{m})}_{\vect{\kappa}}$ be a set of unique representatives of $\left[\overline{\vect{\Gamma}}^{(\vect{m})}_{\vect{\kappa}}\right]$ given in \eqref{1608241504}.
\begin{enumerate}[a)]
\item The polynomials  $\oversim{T}_{\indexvectgamma}$, $\vectgamma \in \vect{\Gamma}^{(\vect{m})}_{\vect{\kappa}}$, 
 form an orthogonal basis of $\oversim{\Pi}^{(\vect{m})}_{\vect{\kappa}}$ with respect to the inner product \eqref{1509060507}.
\item Using the element $\vectgamma^{\ast}\in \vect{\Gamma}^{(\vect{m})}_{\vect{\kappa}}$ given in \eqref{1509301421}, for all $\vect{i} \in \I^{(\vect{m})}_{\vect{\kappa}}$ we can write
\begin{equation}\label{1608971650}
 \oversim{L}^{(\vect{m})}_{\vect{\kappa},\vect{i}}= \mathfrak{w}^{(\vect{m})}_{\vect{\kappa},\vect{i}} \left(\tsum_{\indexvectgamma \in \vect{\Gamma}^{(\vect{m})}_{\vect{\kappa}}}\!\!\dfrac{2^{-\mathfrak{f}^{(\vect{m})}(\indexvectgamma)}}{\|\oversim{T}_{\indexvectgamma}\|^2}  \oversim{T}_{\indexvectgamma}( \vect{z}^{(\vect{m})}_{\vect{i}})\, \oversim{T}_{\indexvectgamma} -
\oversim{T}_{\indexvectgamma^*}(\vect{z}^{(\vect{m})}_{\vect{i}})\,\oversim{T}_{\indexvectgamma^*}\right).
\end {equation}
\end{enumerate}
\end{proposition}
\begin{proof}
By \eqref{1608301335}, the polynomials $\oversim{T}_{\indexvectgamma}$, 
$\vectgamma\in \vect{\Gamma}^{(\vect{m})}_{\vect{\kappa}}$, are pairwise orthogonal. 
Thus, since $\|\oversim{T}_{\indexvectgamma}\|\neq 0$ by \eqref{1608301758}, the polynomials $\oversim{T}_{\indexvectgamma}$, $\vectgamma\in \vect{\Gamma}^{(\vect{m})}_{\vect{\kappa}}$,
are linearly independent. We have $\dim \oversim{\Pi}^{(\vect{m})}_{\vect{\kappa}}=\#\vect{\Gamma}^{(\vect{m})}_{\vect{\kappa}}$ by \eqref{1608312133} and \eqref{1509301005}
and can therefore conclude a). 

For every $\vectgamma'\in \overline{\vect{\Gamma}}^{(\vect{m})}_{\vect{\kappa}}$ there is by assumption \eqref{1509301005} and definition \eqref{1608252121} precisely one 
$\vectgamma\in \vect{\Gamma}^{(\vect{m})}_{\vect{\kappa}}$ such that $\langle \oversim{T}_{\indexvectgamma},T_{\indexvectgamma'}\rangle\neq 0$. In fact, this unique element is given by
$\vectgamma\in \vect{\Gamma}^{(\vect{m})}_{\vect{\kappa}}$ with $\vectgamma\in [\vectgamma']$.
Further, we have $\mathfrak{f}^{(\vect{m})}(\vectgamma')=\mathfrak{f}^{(\vect{m})}(\vectgamma)$ if $\indexvectgamma'\in [\vectgamma]$.
A simple computation using  \eqref{1608301758}, \eqref{1608252121}, \eqref{1608301730}, and Proposition \ref{1609080801}  gives statement b).
\end{proof}

In the new polynomial space $\oversim{\Pi}^{(\vect{m})}_{\vect{\kappa}}$, we want to find the interpolating polynomial $\oversim{P}^{(\vect{m})}_{\vect{\kappa},h}$ that
satisfies the interpolation conditions
\begin{equation}
\label{1608302037} 
\oversim{P}^{(\vect{m})}_{\vect{\kappa},h} (\vect{z}^{(\vect{m})}_{\vect{i}}) = h({\vect{i}}) \quad \text{for all}\quad \vect{i} \in \I^{(\vect{m})}_{\vect{\kappa}}.
\end{equation}

\begin{theorem} \label{201512131945}
Let $h\in\mathcal{L}(\I^{(\vect{m})}_{\vect{\kappa}})$. 
The interpolation problem  \eqref{1608302037}  has a unique solution 
in the polynomial space $\oversim{\Pi}^{(\vect{m})}_{\vect{\kappa}}$ given by 
\begin{equation} \label{201513121708} 
\oversim{P}^{(\vect{m})}_{\vect{\kappa},h} = \tsum_{\vect{i} \in \I^{(\vect{m})}_{\vect{\kappa}}} h({\vect{i}}) \oversim{L}^{(\vect{m})}_{\vect{\kappa},\vect{i}}.
\end{equation}
Furthermore, we have $\left\{\,\oversim{P}^{(\vect{m})}_{\vect{\kappa},h}\,\left|\,h\in \mathcal{L}(\I^{(\vect{m})}_{\vect{\kappa}})\right.\,\right\} = \oversim{\Pi}^{(\vect{m})}_{\vect{\kappa}}$, and the  polynomial functions  $\oversim{L}^{(\vect{m})}_{\vect{\kappa},\vect{i}}$, $\vect{i} \in \I^{(\vect{m})}_{\vect{\kappa}}$, form a basis of the vector space $\oversim{\Pi}^{(\vect{m})}_{\vect{\kappa}}$. 
\end{theorem}

The fundamental polynomials $\oversim{L}^{(\vect{m})}_{\vect{\kappa},\vect{i}}$ are the unique solutions of \eqref{1608302037} for $h=\delta_{\vect{i}}$. 

\medskip

\begin{proof}
Choose a set $\vect{\Gamma}^{(\vect{m})}_{\vect{\kappa}}\subseteq\overline{\vect{\Gamma}}^{(\vect{m})}_{\vect{\kappa}}$ satisfying \eqref{1509301005}. 
Using \eqref{1608971650}, \eqref{1608301730}, and \eqref{1608301758}, the proof is accomplished by following the lines of the proof of Theorem~\ref{1509082002}. 
\end{proof}


\section{Computing the interpolating polynomials}
\label{17008191038}

In order to obtain a numerical scheme for the computation, we consider the expansions
\begin{equation}\label{1509241909B}
P^{(\vect{m})}_{\vect{\kappa},h}(\vect{x}) = \sum_{\indexvectgamma \in \vect{\Gamma}^{(\vect{m})}_{\vect{\kappa}}} c_{\indexvectgamma}(h)\;\!T_{\indexvectgamma}(\vect{x})
\end{equation}
and
\begin{equation}\label{1509241909C}
\oversim{P}^{(\vect{m})}_{\vect{\kappa},h}(\vect{x}) = \sum_{\indexvectgamma \in \overline{\vect{\Gamma}}^{(\vect{m})}_{\vect{\kappa}}} 
\frac{{c}_{\indexvectgamma}(h)}{\#[\vectgamma]} \;\!
T_{\indexvectgamma}(\vect{x}) 
\end{equation}
of the interpolating polynomials $P^{(\vect{m})}_{\vect{\kappa},h}$ and 
$\oversim{P}^{(\vect{m})}_{\vect{\kappa},h}$ in terms of Chebyshev polynomials $T_{\indexvectgamma}$, $\vectgamma \in \vect{\Gamma}^{(\vect{m})}_{\vect{\kappa}}$. 
This expansion enables us to compute the 
interpolating polynomials by first computing the expansion coefficients $c_{\indexvectgamma}(h)$ and then
evaluating the sums in \eqref{1509241909B} and \eqref{1509241909C}. Both steps can be conducted efficiently by using fast Fourier algorithms. 
The idea to evaluate polynomial interpolants in this way is very common for spectral methods. For the Padua points, it was studied in \cite{CaliariDeMarchiSommarivaVianello2011}. For 
polynomial interpolation on the node points of Lissajous curves, it was also already introduced in \cite{DenckerErb2015a,ErbKaethnerAhlborgBuzug2015}.  

\medskip

We first take a look at the computation of the coefficients $c_{\indexvectgamma}(h)$. By Theorem \ref{1509082002} and definition \eqref{1509082001} for the expansion \eqref{1509241909B}, 
as well as by Theorem \ref{201512131945} and definition \eqref{1608270850} for the second expansion \eqref{1509241909C}, we obtain the following identity. 

\begin{corollary} \label{cor:161008}
For $h\in\mathcal{L}(\I^{(\vect{m})}_{\vect{\kappa}})$, the unique coefficients $c_{\indexvectgamma}(h)$ in \eqref{1509241909B} and \eqref{1509241909C} are
\begin{align*}
c_{\indexvectgamma}(h) &= \dfrac1{\|\dchi^{(\vect{m})}_{\indexvectgamma}\|_{\omega^{(\vect{m})}}^2}\,\langle\;\! h,\dchi^{(\vect{m})}_{\indexvectgamma}\rangle_{\omega^{(\vect{m})}},\quad \vectgamma \in \overline{\vect{\Gamma}}^{(\vect{m})}_{\vect{\kappa}}.
\end{align*}
\end{corollary}
Using Corollary \ref{cor:161008}, the coefficients $c_{\indexvectgamma}(h)$ can be computed efficiently using discrete cosine transforms along the $\mathsf{d}$ dimensions of the index set 
\[ \J^{(\vect{m})}= \bigtimes_{\substack{\vspace{-8pt}\\\mathsf{i}=1}}^{\substack{\mathsf{d}\\\vspace{-10pt}}} \{0,\ldots,m_{\mathsf{i}}\}.\]
We introduce
\begin{align*} g^{(\vect{m})}_{\vect{\kappa}}(\vect{i}) &= \left\{ \begin{array}{cl} \mathfrak{w}^{(\vect{m})}_{\vect{i}} h(\vect{i}), \quad  & \text{if}\; \vect{i}\in\I^{(\vect{m})}_{\vect{\kappa}},\rule[-0.65em]{0pt}{1em}\\ 
       0, \quad 
  & \text{if}\; \vect{i}\in\J^{(\vect{m})} \setminus \I^{(\vect{m})}_{\vect{\kappa}}. \end{array} \right. 
\end{align*}
Now, we define the $\mathsf{d}$-dimensional discrete cosine transform $\hat{g}^{(\vect{m})}_{\vect{\kappa},\indexvectgamma}$ of $g^{(\vect{m})}_{\vect{\kappa}}$ starting with
\[ \hat{g}^{(\vect{m})}_{\vect{\kappa},(\gamma_1)}(i_{2},\ldots,i_{\mathsf{d}}) = \sum_{i_{\mathsf{i}}=0}^{m_{\mathsf{1}}} g^{(\vect{m})}_{\vect{\kappa}}(\vect{i})\cos(\gamma_{\mathsf{1}}i_{1}\pi/m_{1}).\]
and, then proceeding recursively for $\mathsf{i}=2,\ldots,\mathsf{d}$ with
\[\hat{g}^{(\vect{m})}_{\vect{\kappa},(\gamma_1,\ldots,\gamma_{\mathsf{i}})}(i_{\mathsf{i}+1},\ldots,i_{\mathsf{d}})= \sum_{i_{\mathsf{i}}=0}^{n_{\mathsf{i}}}
\hat{g}^{(\vect{m})}_{\vect{\kappa},(\gamma_1,\ldots,\gamma_{\mathsf{i}-1})}(i_{\mathsf{i}},\ldots,i_{\mathsf{d}})\cos(\gamma_{\mathsf{i}}i_{\mathsf{i}}\pi/n_{\mathsf{i}}).\]
Including the formula \eqref{1509231405} for the norm $\|\dchi^{(\vect{m})}_{\indexvectgamma}\|_{\omega^{(\vect{m})}_{\vect{\kappa}}}^2$, we obtain in this way 
\[c_{\indexvectgamma}(h) = 
\left\{ \begin{array}{rl}  2^{\mathfrak{e}(\indexvectgamma)-\mathfrak{f}^{(\vect{m})}(\indexvectgamma)} \hat{g}^{(\vect{m})}_{\vect{\kappa}, \indexvectgamma} ,\; & 
\text{if}\quad \vectgamma
 \in\overline{\vect{\Gamma}}^{(\vect{m})}_{\vect{\kappa}}\setminus\mathfrak{S}^{(\vect{m})}(\vect{0}),\\[0.3em]
   \hat{g}^{(\vect{m})}_{\vect{\kappa}, \indexvectgamma},\; & \text{if}\quad \vectgamma \in\mathfrak{S}^{(\vect{m})}(\vect{0}). \end{array} \right.\] 
As a composition of $\mathsf{d}$ fast cosine transform, the computation of the set of coefficients $c_{\indexvectgamma}(h)$ can be conducted in
$\mathcal{O}\!\left(\!\!\;\p[\vect{m}] \ln \!\!\;\p[\vect{m}] \right)$ arithmetic operations. 
Once the coefficients $c_{\indexvectgamma}(h)$ are calculated, the evaluations of the interpolating polynomials
$P^{(\vect{m})}_{\vect{\kappa},h}(\vect{x})$ and $\oversim{P}^{(\vect{m})}_{\vect{\kappa},h}(\vect{x})$ at different points $\vect{x} \in [-1,1]^{\mathsf{d}}$ 
can be computed using \eqref{1509241909B} and \eqref{1509241909C}. If the evalution points are lying on a suitable grid in $[-1,1]^{\mathsf{d}}$ 
also the sum evaluations in \eqref{1509241909B} and \eqref{1509241909C} can be carried out efficiently using fast Fourier methods.


\section{Examples}\label{1609031100}

\subsection{Degenerate Lissajous curves and Padua points}

An important class of node points that fit into the framework of this article are interpolation points that are generated by equidistant
sampling along a degenerate Lissajous curves. In the bivariate setting, the most prominent examples are the Padua points. They were introduced in
\cite{CaliariDeMarchiVianello2005} and investigated profoundly in a series of papers \cite{BosDeMarchiVianelloXu2006,BosDeMarchiVianelloXu2007,CaliariDeMarchiVianello2008}. Recently,
this theory was extended to  multivariate Lissajous curves \cite{DenckerErb2015a,Erb2015,ErbKaethnerDenckerAhlborg2015}.

\medskip

The respective point sets are given by $\LC^{(\vect{n})}_{\vect{\kappa}}$, where $\vect{n}=(n_1,\ldots,n_{\mathsf{d}})$ is such that
\begin{equation}\label{16009031509} 
n_1,\ldots,n_{\mathsf{d}}\in \mathbb{N}\; \quad \text{are pairwise relatively prime}.
\end{equation}
The interpolation theory 
corresponding to these sets is described in \cite{DenckerErb2015a} for $\vect{\kappa}=\vect{0}$. Note that the sets $\LC^{(\vect{n})}_{\vect{\kappa}}$ coincide with $\LC^{(\vect{n})}_{\vect{0}}$ up to reflections with respect to the coordinate axis. We give a short summary using the statements of this work.

\medskip

Since the entries of $\vect{n}$ are pairwise relatively prime, the integer vectors $\vect{n}^{\sharp}$ and $\vect{n}^{\flat}$ according to Proposition~\ref{1509061252} are unique and $\vect{n}^{\sharp} = \vect{n}$, $\vect{n}^{\flat} = \vect{1}$.
Furthermore, we need only to consider the set $H^{(\vect{n})} = H^{(\vect{n}^{\sharp})} = \{0,\ldots,2 \p[\vect{n}]-1\}$ in Proposition \ref{1509221521} whereas $R^{(\vect{n}^{\flat})} = \{ \vect{0} \}$ plays no role in this case. We have
\[ \# \LC^{(\vect{n})}_{\vect{\kappa}} = \# \I^{(\vect{n})}_{\vect{\kappa}} = \# \vect{\Gamma}^{(\vect{n})}_{\vect{\kappa}} = \dfrac1{2^{\mathsf{d}-1}}\p[\vect{n}+\vect{1}].
\]
The set $\vect{\mathfrak{L}}^{(\vect{n}, \, \vect{1})}_{\vect{\kappa}}$ given in \eqref{1609010824} contains only the degenerate Lissajous curve 
$\vect{\ell}^{(\vect{n})}_{\vect{\kappa}}$. In \cite[Theorem 1.1]{DenckerErb2015a} it is shown that the degenerate curve $\vect{\ell}^{(\vect{n})}_{\vect{\kappa}}(t)$ intersects every point $\vect{z}_{\vect{i}}^{(\vect{n})}$ with $\vect{i} \in \I^{(\vect{n})}_{\vect{\kappa},\mathsf{M}}$ and \text{$\mathsf{M} \subseteq \{1,\ldots \mathsf{d}\}$} at precisely $2^{ \# \mathsf{M}}$ points $t \in [0,2\pi)$.

If $\mathsf{d}\geq 2$, exactly one class in the decomposition $\left[\overline{\vect{\Gamma}}^{(\vect{n})}_{\vect{\kappa}}\right]$ defined in \eqref{1608241504}
consists of more than one element, namely the class $\mathfrak{S}^{(\vect{n})}(\vect{0})$. Thus, to choose a set of representatives with \eqref{1509301005} we only have to specify a
single element $\vectgamma^{\ast}$ satisfying \eqref{1509301421}. In \cite{DenckerErb2015a} this element is chosen as $\vectgamma^{\ast}=(0, \ldots, 0, n_{\mathsf{d}})$.

The relations between the sets $\vectGammacircvect{m}$, $\overline{\vect{\Gamma}}^{(\vect{n})}_{\vect{\kappa}}$ and $\vect{\Gamma}^{(\vect{n})}_{\vect{\kappa}}$ are given by
\[
\vect{\Gamma}^{(\vect{n})}_{\vect{\kappa}} = \vectGammacircvect{n} \cup \{\vectgamma^{\ast}\}, \qquad 
\overline{\vect{\Gamma}}^{(\vect{n})}_{\vect{\kappa}} = \vectGammacircvect{n} \cup \mathfrak{S}^{(\vect{n})}(\vect{0}). 
\]
In particular, the sets $\overline{\vect{\Gamma}}^{(\vect{n})}_{\vect{\kappa}}$, $\vect{\Gamma}^{(\vect{n})}_{\vect{\kappa}}$ do not depend on the parameter $\vect{\kappa}$. For the relation to the Padua points and further concrete examples, we refer to \cite{DenckerErb2015a}.

\subsection{The Morrow-Patterson-Xu points}
\begin{figure}[htb]
	\centering
	\subfigure[\hspace*{1em} $\LC^{(4,4)}_{(0,0)}$ and $\mathcal{C}^{(4,4)}_{(0,0)} = \displaystyle \bigcup_{k \in \{0,2,4\}} \vect{\ell}^{(4,4)}_{(0,k)}(\mathbb{R})$
	]{\includegraphics[scale=0.8]{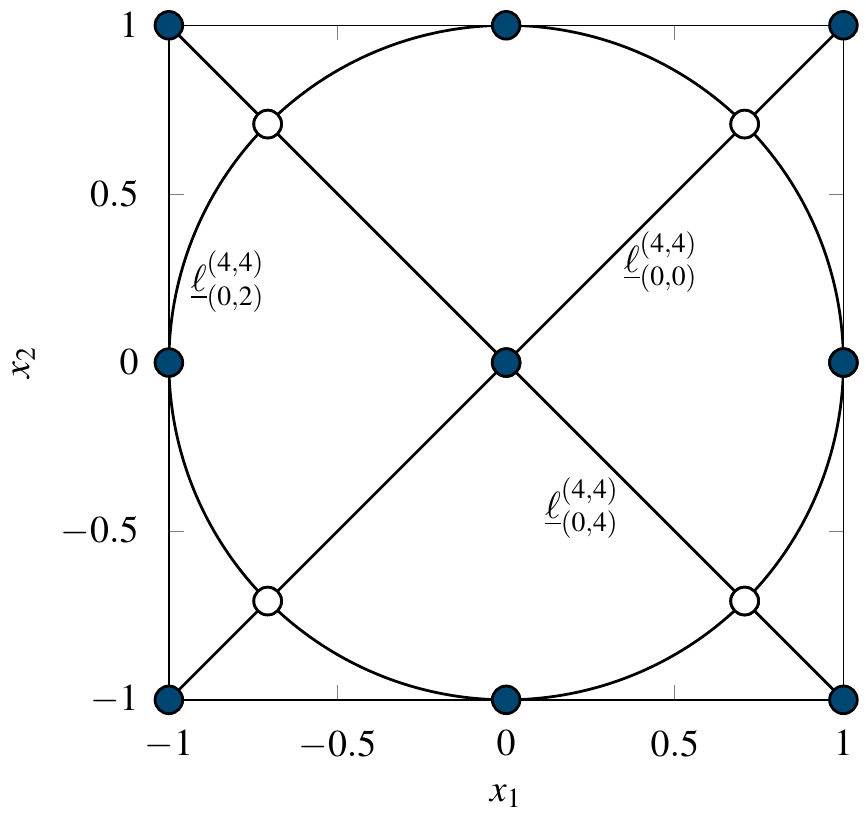}}
	\hfill	
	\subfigure[\hspace*{1em} $\LC^{(5,5)}_{(0,0)}$ and $\mathcal{C}^{(5,5)}_{(0,0)} = \displaystyle \bigcup_{k \in \{0,2,4\}} \vect{\ell}^{(5,5)}_{(0,k)}(\mathbb{R})$
	]{\includegraphics[scale=0.8]{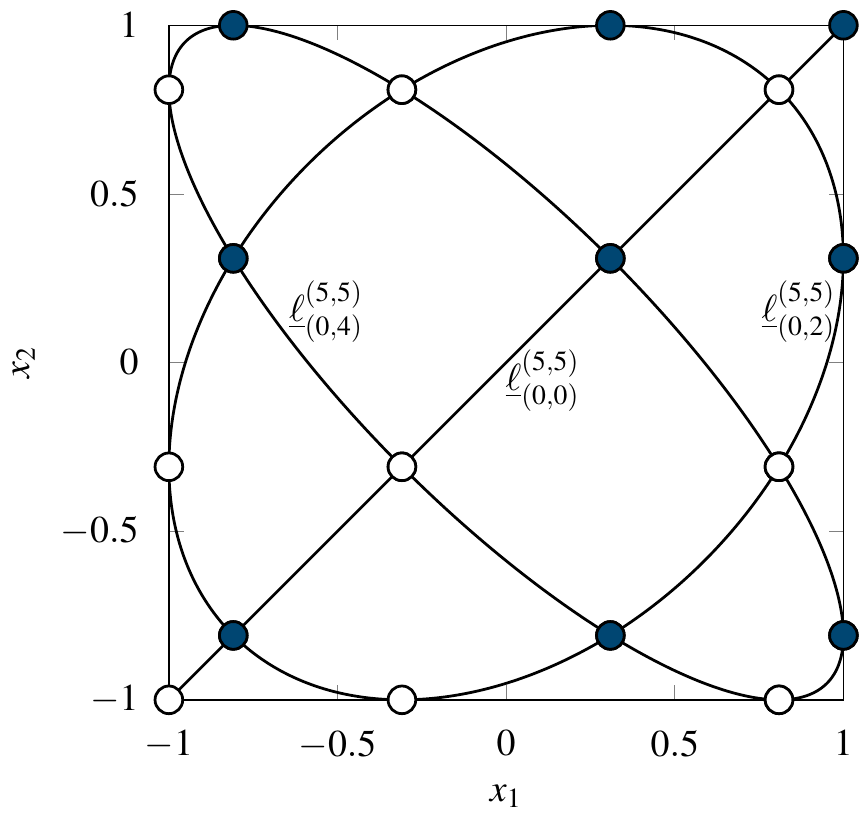}} 
  	\caption{Illustration of the bivariate Morrow-Patterson-Xu points $\LC^{(m,m)}_{(0,0)}$ and the corresponding Chebyshev variety $\mathcal{C}^{(m,m)}_{(0,0)}$.
  	The subsets $\LC^{(m,m)}_{(0,0),0}$ and $\LC^{(m,m)}_{(0,0),1}$ are colored in blue and 
  	white, respectively. 
  	}
	\label{fig:MPX-1}
\end{figure}

 \begin{figure}[htb]
	\centering
	\subfigure[$\overline{\vect{\Gamma}}^{(4,4)}_{(0,0)}$ and $\vect{\Gamma}^{(4,4)}_{(0,0)} \! = \overline{\vect{\Gamma}}^{(4,4)}_{(0,0)} \! \setminus \! \{(3,1), (4,0)\}$
	]{\includegraphics[scale=0.79]{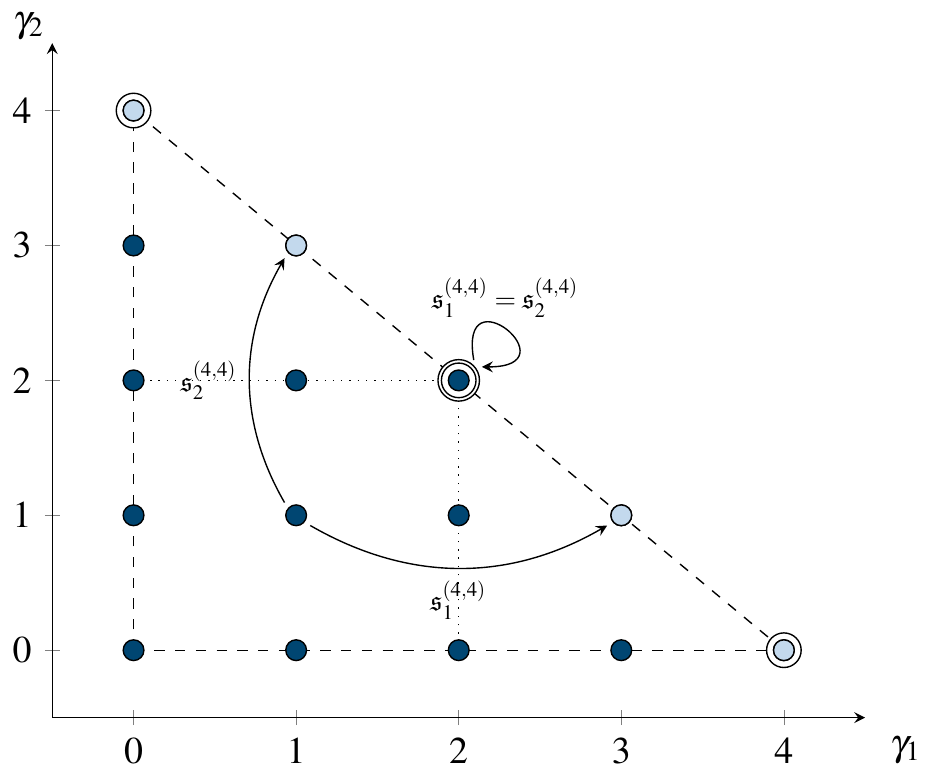}}
	\hfill	
	\subfigure[$\overline{\vect{\Gamma}}^{(5,5)}_{(0,0)}$ and $\vect{\Gamma}^{(5,5)}_{(0,0)} \! = \overline{\vect{\Gamma}}^{(5,5)}_{(0,0)} \! \setminus\! \{(3,2), (4,1), (5,0)\}$
	]{\includegraphics[scale=0.79]{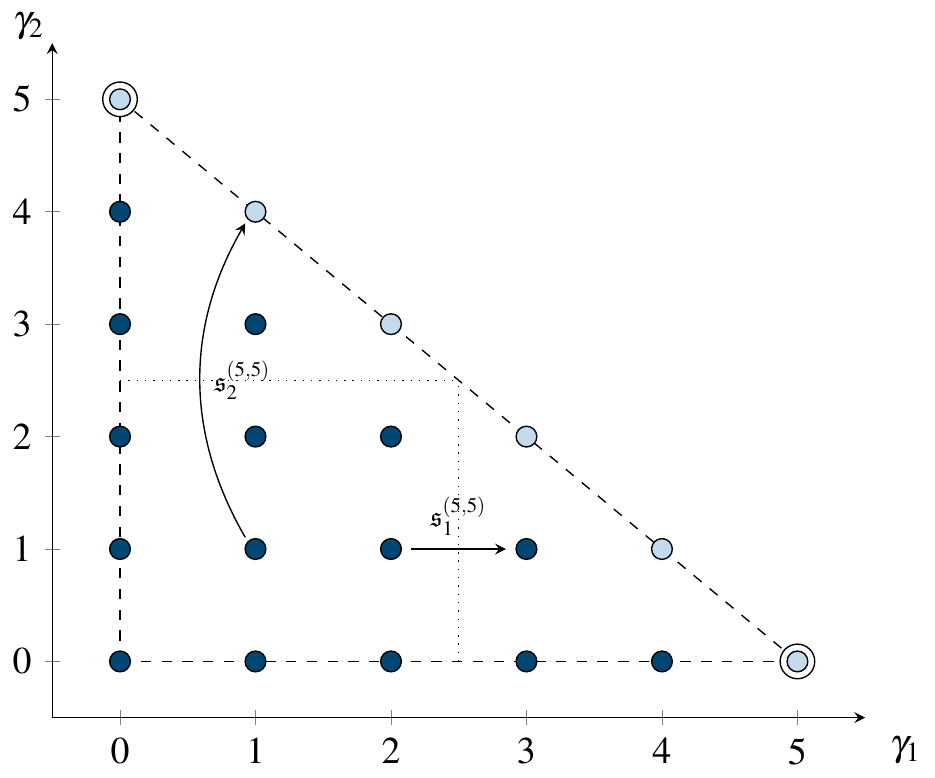}} 
  	\caption{Illustration of the spectral index sets $\overline{\vect{\Gamma}}^{(m,m)}_{(0,0)}$ and the mappings $\mathfrak{s}^{(m,m)}_{1}$ and $\mathfrak{s}^{(m,m)}_{2}$ 
  	for the Morrow-Patterson-Xu points given in Figure \ref{fig:MPX-1}.
  	The elements $\vectgamma$ with $\#[\vectgamma] = 1$ are colored in dark blue, the elements with $\#[\vectgamma] = 2$ in light blue. The
  	dots with one circle correspond to the elements in the class $\mathfrak{S}^{(m,m)}((0,0))$. The double circled dot is the only element 
  	with $\mathfrak{f}^{(\vect{m})}(\vectgamma) = 1$ in formula \eqref{1608311805}. 
  	} \label{fig:MPX-2}
\end{figure}
\begin{figure}[htb]
	\centering
	\subfigure[\hspace*{1em} $\LC^{(4,4,4)}_{(0,0,0)}$ and $\mathcal{C}^{(4,4,4)}_{(0,0,0)}$
	]{\includegraphics[scale=0.85]{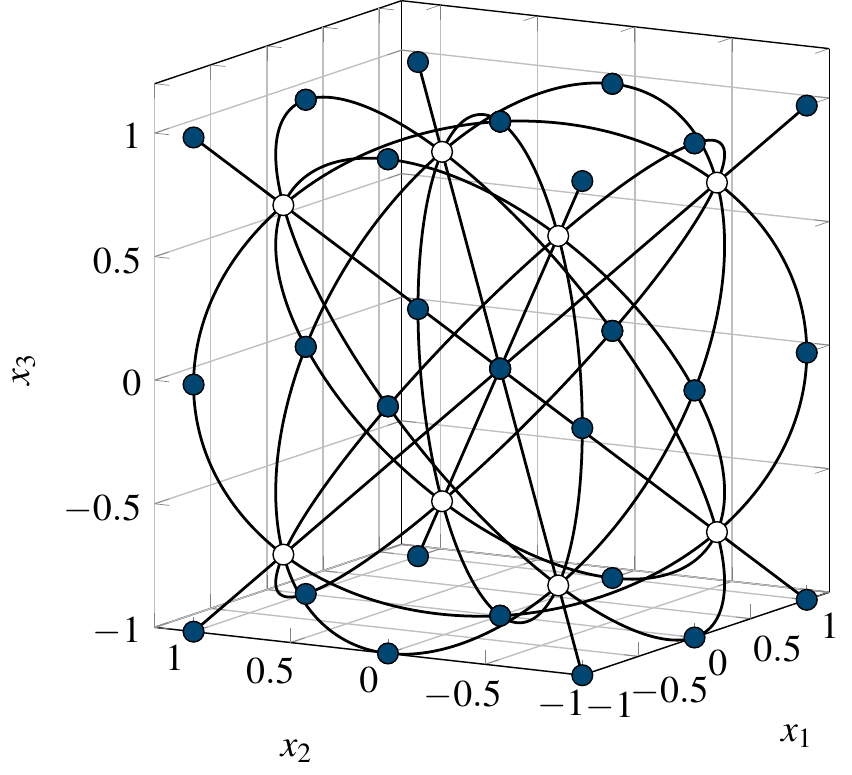}}
	\hfill	
	\subfigure[\hspace*{1em} $\overline{\vect{\Gamma}}^{(4,4,4)}_{(0,0,0)}$
	]{\includegraphics[scale=0.9]{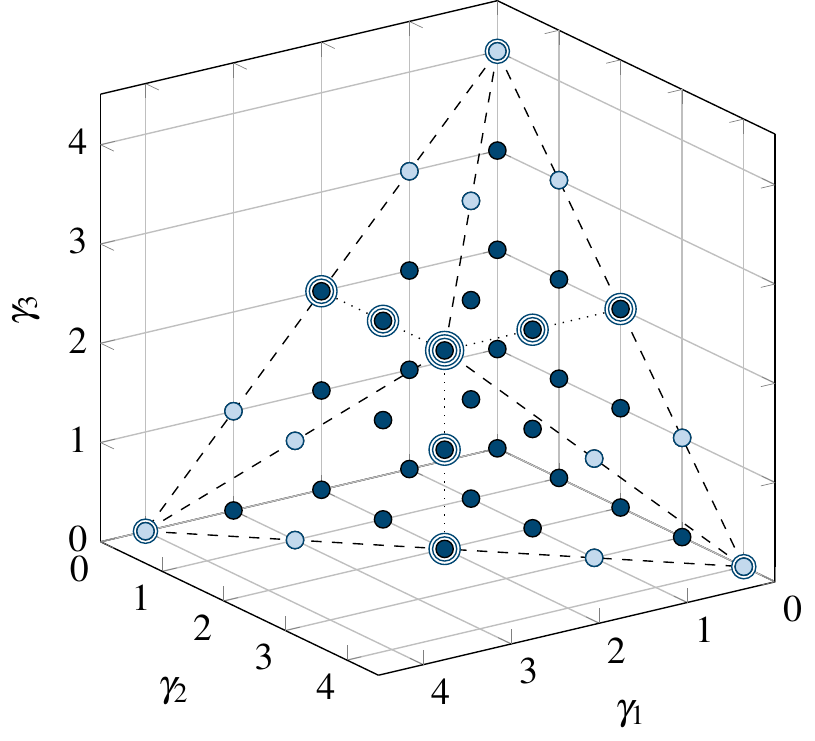}} 
  	\caption{Illustration of the Morrow-Patterson-Xu points $\LC^{(4,4,4)}_{(0,0,0)}$, 
  	the variety $\mathcal{C}^{(4,4,4)}_{(0,0,0)}$ and $\overline{\vect{\Gamma}}^{(4,4,4)}_{(0,0,0)}$.  $\mathcal{C}^{(4,4,4)}_{(0,0,0)}$ consists of the union of $10$ ellipses. 
  	The subsets $\LC^{(4,4,4)}_{\vect{0},0}$ and $\LC^{(4,4,4)}_{\vect{0},1}$ are colored in blue and 
  	white, respectively. On the right hand side, the corners $\vectgamma \in \mathfrak{S}^{(\vect{m})}(\vect{0})$ are marked with an extra ring. For all dots with two and three
  	rings we have $\mathfrak{f}^{(\vect{m})}(\vectgamma) = 1 $ and $\mathfrak{f}^{(\vect{m})}(\vectgamma) = 2$ in formula \eqref{1608311805},
  	respectively. The elements $\vectgamma$ with $\# [\vectgamma] = 1$ are colored in dark blue.
  	} \label{fig:MPX-3}
\end{figure}

\begin{figure}[htb]
	\centering
	\subfigure[\hspace*{1em} $\LC^{(5,4,2)}_{(0,0,1)}$ and $\mathcal{C}^{(5,4,2)}_{(0,0,1)} = \vect{\ell}^{(5,4,2)}_{(0,0,1)}([0,2\pi))$
	]{\includegraphics[scale=0.85]{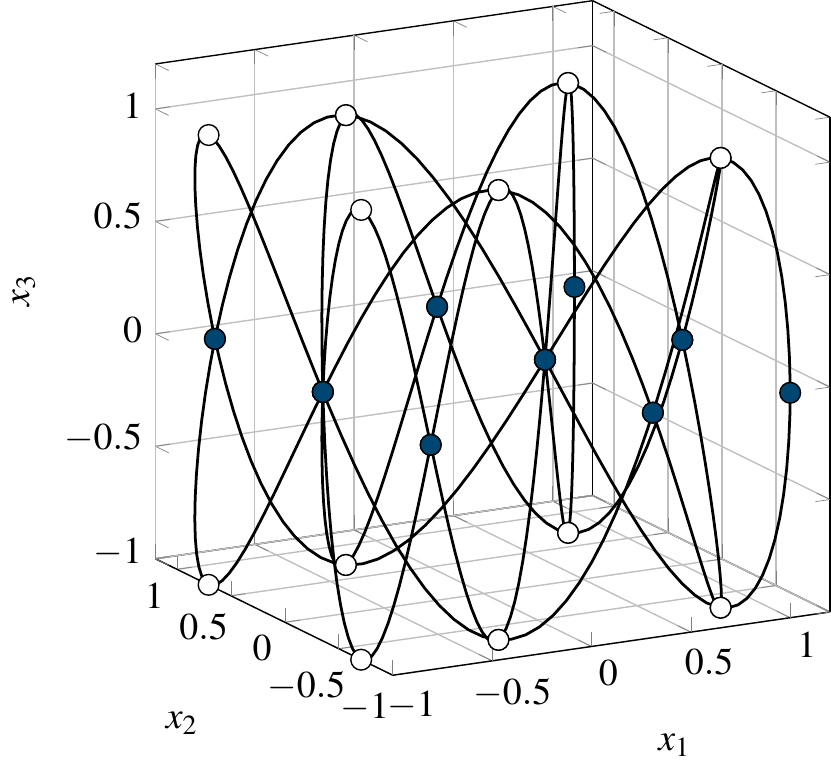}}
	\hfill	
	\subfigure[\hspace*{1em} $\overline{\vect{\Gamma}}^{(5,4,2)}_{(0,0,1)}$
	]{\includegraphics[scale=0.85]{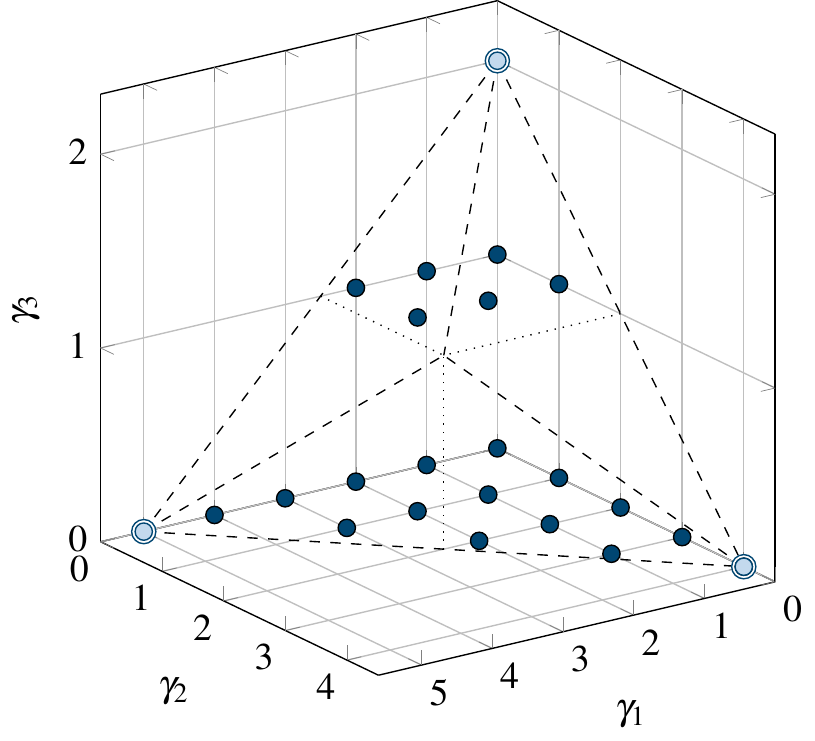}} 
  	\caption{Illustration of the node points $\LC^{(5,4,2)}_{(0,0,1)}$, the non-degenerate Lissajous curve $\vect{\ell}^{(5,4,2)}_{(0,0,1)}(t)$ and the set $\overline{\vect{\Gamma}}^{(4,4,4)}_{(0,0,0)}$. The subsets $\LC^{(5,4,2)}_{(0,0,1),0}$ and $\LC^{(5,4,2)}_{(0,0,1),1}$  are colored in blue and 
  	white, respectively. On the right hand side, the corners $\vectgamma \in \mathfrak{S}^{(\vect{m})}(\vect{0})$ are marked with an extra ring. 
  	The elements $\vectgamma$ with $\# [\vectgamma] = 1$ are colored in dark blue.
  	} \label{fig:LCpoints542}
\end{figure}
 
A particular family of non-tensor product sets introduced by Morrow and Patterson in \cite{MorrowPatterson1978} turned out to be particularly suitable for bivariate quadrature
and was studied by a series of authors. This family was extended in \cite{Xu1996} by Xu and also studied in terms of bivariate Lagrange interpolation \cite{Harris2010,Xu1996}. 
Generalizations of these point sets in three or more dimensions are given in \cite{BojanovPetrova1997,DeMarchiVianelloXu2007,Noskov91}. 

The node sets considered in the works referenced above are special cases of the 
Lissajous-Chebyshev nodes $\LC^{(\vect{m})}_{\vect{\kappa}}$ with the parameter $\vect{m}$ 
given as $\vect{m} = (m, \ldots, m)$ and $m \in \mathbb{N}$. In the following, we call these node sets Morrow-Patterson-Xu points and summarize the results of the previous sections for this particular case. 
A possible component-wise multiplicative decomposition of $\vect{m}$ according to  
       Proposition \ref{1509061252} is  
\[\vect{m}^{\sharp} = (m, 1, \ldots, 1) \quad \text{and} \quad \vect{m}^{\flat} = (1, m, \ldots, m).\]
Then, the respective sets $H^{(\vect{m}^{\sharp})}$ and $R^{(\vect{m}^{\flat})}$ are 
\[H^{(\vect{m}^{\sharp})}=\{0,\ldots,2m-1\} \quad \text{and} \quad
\vect{R}^{(\vect{m}^{\flat})} = \{ 0 \} \times \{0,\ldots,m-1\}^{\mathsf{d}-1}.\]
Further, the number of elements in the sets $\LC^{(\vect{m})}_{\vect{\kappa},\mathfrak{r}}$ is given as
\[ \# \LC^{(\vect{m})}_{\vect{\kappa},\mathfrak{r}} = \dfrac1{2^{\mathsf{d}}} \left\{ 
\begin{array}{cl}
 (m+1)^{\mathrm{d}}  & \text{if $m$ is odd}, \\
 (m+2)^{\# \{ \,\mathsf{i}\,|\,\kappa_{\mathsf{i}}\equiv\mathfrak{r}\tmod 2 \}}\,m^{\# \{ \,\mathsf{i}\,|\,\kappa_{\mathsf{i}}\not\equiv\mathfrak{r}\tmod 2 \}}& \text{if $m$ is even}.
\end{array} \right.
\]
By Theorem \ref{201708201819}, we have 
\[ N_{\mathrm{deg}} = \left\{ 
\begin{array}{cl}
1 & \text{if $m$ is odd}, \\
2^{\mathsf{d}- 1} & \text{if $m$ is even and $\kappa_{\mathsf{i}} \equiv \kappa_{\mathsf{j}} \tmod 2$ for all $\mathsf{i},\mathsf{j}$}, \\
0 & \text{if $m$ is even and $\kappa_{\mathsf{i}} \not\equiv \kappa_{\mathsf{j}} \tmod 2$ for some $\mathsf{i},\mathsf{j}$}.
\end{array} \right.
\]

\vspace*{-2em}

\noindent Furthermore, we have
\[ \#[\vect{\mathfrak{L}}^{(\vect{m}^{\sharp}, \, \vect{m}^{\flat})}_{\vect{\kappa}}] = \left\{ \begin{array}{cl}
(m^{\mathsf{d}-1}+1)/2& \text{if $m$ is odd}, \\
(m^{\mathsf{d}-1}+2^{\mathsf{d}- 1})/2  & \text{if $m$ is even and $\kappa_{\mathsf{i}} \equiv \kappa_{\mathsf{j}} \tmod 2$ for all $\mathsf{i},\mathsf{j}$}, \\
m^{\mathsf{d}-1}/{2}  & \text{if $m$ is even and $\kappa_{\mathsf{i}} \not\equiv \kappa_{\mathsf{j}} \tmod 2$ for some $\mathsf{i},\mathsf{j}$}.
\end{array} \right.
\]

\noindent The curves in $\vect{\mathfrak{L}}^{(\vect{m}^{\sharp}, \, 
\vect{m}^{\flat})}_{\vect{\kappa}}$ are ellipses in $[-1,1]^{\mathsf{d}}$. This is easily seen from the identity
\[\left( \cos \left(t + \dfrac{\kappa_1}{m} \right), \cos \left(t + \dfrac{\kappa_2+2\rho_2}{m}\right), \cdots, \cos \left(t + \dfrac{\kappa_{\mathsf{d}}+2\rho_{\mathsf{d}}}{m}\right) \right) 
= \vect{a}_1 \cos t + \vect{a}_2 \sin t,
\]

\vspace{-1.5em}

\begin{align*}\vect{a}_1 &= \left( \cos \left(\dfrac{\kappa_1}{m} \right), \cos \left(\dfrac{\kappa_2+2\rho_2}{m}\right), \cdots, \cos \left( \dfrac{\kappa_{\mathsf{d}}+2\rho_{\mathsf{d}}}{m}\right) \right), \\
 \vect{a}_2 &= \left( \sin \left(\dfrac{\kappa_1}{m} \right), \sin \left(\dfrac{\kappa_2+2\rho_2}{m}\right), \cdots, \sin \left( \dfrac{\kappa_{\mathsf{d}}+2\rho_{\mathsf{d}}}{m}\right) \right).
\end{align*}
Proposition \ref{1509221521} and Theorem \ref{201708201819} imply that every point $\vect{z}_{\vect{i}}^{(\vect{m})}$ with 
$\vect{i} \in \I^{(\vect{m})}_{\vect{\kappa},\mathsf{M}}$ and $\# \mathsf{M} \geq 2$, 
is the intersection point of precisely $2^{ \# \mathsf{M}-1}$ ellipses 
in $[\vect{\mathfrak{L}}^{(\vect{m}^{\sharp}, \, \vect{m}^{\flat})}_{\vect{\kappa}}]$. If $\# \mathsf{M} = 1$, then $\vect{z}_{\vect{i}}^{(\vect{m})}$ corresponds 
to a boundary point of exactly one of the ellipses in $[\vect{\mathfrak{L}}^{(\vect{m}^{\sharp}, \, \vect{m}^{\flat})}_{\vect{\kappa}}]$ at an edge 
of the hypercube $[-1,1]^\mathsf{d}$. If $\# \mathsf{M} = 0$, then $\vect{z}_{\vect{i}}^{(\vect{m})}$ is a vertex of $[-1,1]^\mathsf{d}$ and an element 
of exactly one of the $N_{\mathrm{deg}}$ degenerate ellipses in $[\vect{\mathfrak{L}}^{(\vect{m}^{\sharp}, \, \vect{m}^{\flat})}_{\vect{\kappa}}]$.

\begin{example} 
 For $\mathsf{d} = 1$ and $m \in \mathbb{N}$, the set $\vect{\mathfrak{L}}^{(m)}_{\kappa}$ consists only of the degenerate Lissajous curve 
 $\vect{\ell}_{\kappa}^{(m)}(t) = \cos t$ satisfying $\vect{\ell}_{\kappa}^{(m)}([0,2\pi)) = [-1,1]$. The points in $\LC^{(m)}_{\kappa} = 
 \left\{\, z^{(m)}_{i}\,\left|\,i \in \{0, \ldots , m\} \right. \right\}$ are, independent of the parameter $\kappa \in \mathbb{N}_0$, the univariate Chebyshev-Gauß-Lobatto points.
 The formulas in \eqref{1509082003}, \eqref{201513121708} correspond to the formulas of univariate Chebyshev-Gauß-Lobatto interpolation in the space of polynomials of 
 degree $m$ on the interval $[-1,1]$, cf. \cite[Section 3.4]{Shen2011}. Theorem \ref{1509082004} describes the respective quadrature rule for polynomials of degree at most $2m-1$. 
 \end{example}
\begin{example} \label{1609031131} 

Let  $\mathsf{d} = 2$ and $\vect{m} = (m,m)$, $m \in \mathbb{N}$.
In this bivariate setting, we have
\begin{align}
\vectGammacircvect{m} &= \left\{\,\left.\vectgamma \in \mathbb{N}_0^2 \ \right| \ \gamma_1 + \gamma_2 < m\,\right\},\nonumber \\
\overline{\vect{\Gamma}}^{(\vect{m})}_{\vect{\kappa}} &= \left\{\,\left.\vectgamma \in \mathbb{N}_0^2 \ \right| \ \gamma_1 + \gamma_2 \leq m, \vectgamma \neq (m/2,m/2) \quad \text{if}\quad \kappa_1 \not\equiv \kappa_2 \tmod 2 \,\right\}, \nonumber\\
\oversim{\Pi}^{(\vect{m})}_{\vect{\kappa}} &= \left\{\,\left.\tsum_{\gamma_1 + \gamma_2 \leq m} c_{\indexvectgamma}\,T_{\indexvectgamma}\ \right| \ c_{(\gamma_1,\gamma_2)} = (-1)^{\kappa_2 - \kappa_1} c_{(\gamma_2,\gamma_1)} \quad \text{if}\quad \gamma_1 + \gamma_2 = m\,\right\}. \label{1708251106}
\end{align}
For $\vect{i}\in\I^{(\vect{m})}_{\vect{\kappa},\mathsf{M}}$, $\mathsf{M}\subseteq\{1,2\}$, the 
interpolation polynomial  $\oversim{L}^{(\vect{m})}_{\vect{\kappa},\vect{i}}$ in \eqref{1608270850} is given as 
\begin{align*} \dfrac{2^{\#\mathsf{M}-1}}{m^2} & 
\left( \tsum_{\indexvectgamma\in \vectGammacircvect{m}} \!\!\! \dfrac{ {T}_{\indexvectgamma}( \vect{z}^{(\vect{m})}_{\vect{i}})}{\|{T}_{\indexvectgamma}\|^2}\,T_{\indexvectgamma} + \dfrac12 
       \tsum_{\substack{\indexvectgamma \in \overline{\vect{\Gamma}}^{(\vect{m})}_{\vect{\kappa}}\\ \gamma_1 + \gamma_2 = m}} \!\!\!\! \dfrac{ {T}_{\indexvectgamma}( \vect{z}^{(\vect{m})}_{\vect{i}})}{\|{T}_{\indexvectgamma}\|^2}T_{\indexvectgamma}  -  \dfrac14 
       \tsum_{\indexvectgamma \in \{(0,m),(m,0)\}} \!\! \dfrac{{T}_{\indexvectgamma}( \vect{z}^{(\vect{m})}_{\vect{i}})}{\|{T}_{\indexvectgamma}\|^2}T_{\indexvectgamma}  \right).
\end{align*}
The number of equivalence classes $[\vect{\mathfrak{L}}^{(\vect{m}^{\sharp}, \, \vect{m}^{\flat})}_{\vect{\kappa}}]$ of the ellipses is given by
 \[ \# [\vect{\mathfrak{L}}^{(\vect{m}^{\sharp}, \, \vect{m}^{\flat})}_{\vect{\kappa}}]  = \left\{ \begin{array}{cl}
(m+1)/2& \text{if $m$ is odd}, \\
(m+2)/2  & \text{if $m$ is even and $\kappa_{\mathsf{1}} \equiv \kappa_{\mathsf{2}} \tmod 2$}, \\
m/{2}  & \text{if $m$ is even and $\kappa_{1} \not\equiv \kappa_{2} \tmod 2$}.
\end{array} \right.
\] 
The point sets $\LC^{(\vect{m})}_{\vect{\kappa}}$ with  $\vect{m} = (2k,2k)$ and $\vect{\kappa} \in \{(0,0),(0,1)\}$ are the ones introduced by 
Morrow and Patterson \cite{MorrowPatterson1978}. The sets $\LC^{(2k,2k)}_{(0,1)}$, $\LC^{(2k+1,2k+1)}_{(0,0)}$ were later on studied by Xu \cite{Xu1996}
in terms of bivariate polynomial interpolation. In \cite{Xu1996}, a variant of Theorem \ref{201512131945} is proven for the space~\eqref{1708251106} with  $\vect{m} = (2k,2k)$ and $\vect{\kappa} \in \{(0,0),(0,1)\}$.

\medskip

\noindent The general statement of Theorem \ref{201512131945} for 
two-dimensional point sets $\LC^{(\vect{m})}_{\vect{\kappa}}$ with $\vect{m} = (m,m)$ is formulated in \cite{Harris2010}. 
In \cite{Harris2013}, this statement is extended to polynomial spaces of Chebyshev polynomials of the second, third and fourth kind. Using the zeros of Jacobi polynomials instead
of the zeros of Chebyshev polynomials a corresponding extension can be found in \cite{Xu2012}.
Two concrete examples of bivariate Morrow-Patterson-Xu points and the respective sets $\vect{\Gamma}^{(m,m)}_{(0,0)}$ 
are given in Figure \ref{fig:MPX-1} and Figure \ref{fig:MPX-2}.
\end{example}
\begin{example}

For $\mathsf{d} \geq 3$, different variants of the point sets $\LC^{(\vect{m})}_{\vect{\kappa}}$, $\vect{m} = (m, \ldots, m)$ are studied in
\cite{BojanovPetrova1997,DeMarchiVianelloXu2007,Noskov91} for multivariate cubature. 
 For $\mathsf{d} \geq 3$, the exact interpolation statements of Theorem \ref{1509082002} and 
Theorem \ref{201512131945}  are novel in this work. An explicit example of trivariate Morrow-Patterson-Xu nodes is illustrated
in Figure \ref{fig:MPX-3}.
\end{example}

\subsection{Node points of single non-degenerate Lissajous curves}

For the class $\mathfrak{S}^{(\vect{m})}(\vect{0})$, we generally have $\#\mathfrak{S}^{(\vect{m})}(\vect{0})=\mathsf{d}$.  In Proposition \ref{201512151534}, for $\mathsf{d}\geq 2$ we 
characterized all setups in which all other classes consist of precisely one element. The given condition~i) in Proposition \ref{201512151534} can be stated as follows: there exist $\vect{\epsilon}\in \{1,2\}^{\mathsf{d}}$ 
and $\vect{n}\in\mathbb{N}^{\mathsf{d}}$ satisfying \eqref{16009031509} such that $m_{\mathsf{i}} = \epsilon_{\mathsf{i}} n_{\mathsf{i}}$ 
for all $\mathsf{i}\in\{1,\ldots,\mathsf{d}\}$.  In \cite{DenckerErb2015a}, the cases $\vect{\epsilon}\in \{\vect{1},\vect{2}\}$ are discussed in more detail. 

For $\mathsf{d} \geq 2$, we consider now conditions that ensure 
the node points to be generated by a single non-degenerate Lissajous curve. 
For $\mathsf{d} = 2$, such non-degenerate Lissajous curves and its node points were first studied in \cite{ErbKaethnerAhlborgBuzug2015}.

We consider the statements of Theorem \ref{201708201819}. If $\#[\vect{\mathfrak{L}}^{(\vect{m}^{\sharp}, \, \vect{m}^{\flat})}_{\vect{\kappa}}]$ is supposed to be one and $N_{\mathrm{deg}} = 0$, then $\p[\vect{m}^{\flat}]=2$. Thus, we are in the above mentioned setup of Proposition \ref{201512151534}  and 
there are $\mathsf{g},\mathsf{g}' \in \{1, \ldots, \mathsf{d} \}$ with $\mathsf{g} \neq \mathsf{g}'$ such that  $n_{\mathsf{g}}$ is even and $n_{\mathsf{i}}$ is odd for all $\mathsf{i} \neq \mathsf{g}$ and $\epsilon_{\mathsf{g}'} = 2$ and $\epsilon_{\mathsf{i}} = 1$ for all $\mathsf{i} \neq \mathsf{g}'$.
Moreover, $N_{\mathrm{deg}} = 0$ is only possible if $\kappa_{\mathsf{g}'} \not\equiv \kappa_{\mathsf{g}} \tmod 2$ is satisfied. 

On the other hand, if  $n_{\mathsf{g}}$ is even for some $\mathsf{g} $, $n_{\mathsf{i}}$ is odd for all $\mathsf{i} \neq \mathsf{g}$, 
$\epsilon_{\mathsf{g}'} = 2$ for some $\mathsf{g}' \neq\mathsf{g}$, $\epsilon_{\mathsf{i}} = 1$ for all $\mathsf{i} \neq \mathsf{g}'$, and  $\kappa_{\mathsf{g}'} \not\equiv \kappa_{\mathsf{g}} \tmod 2$, then, using $m_{\mathsf{i}} = \epsilon_{\mathsf{i}} n_{\mathsf{i}}$, we have $\# [\vect{\mathfrak{L}}^{(\vect{m}^{\sharp}, \, \vect{m}^{\flat})}_{\vect{\kappa}}]=1$ and $N_{\mathrm{deg}} = 0$.
Furthermore, in this case  we easily compute
\[ \# \LC^{(\vect{m})}_{\vect{\kappa}} = \# \I^{(\vect{m})}_{\vect{\kappa}} = \# \vect{\Gamma}^{(\vect{m})}_{\vect{\kappa}} = 
\dfrac{1}{2^{\mathsf{d}-1}} \left( \p[\vect{m}+\vect{1}] - \tprod_{\mathsf{i} \notin \{\mathsf{g},\mathsf{g}'\}}(n_{\mathsf{i}}+1) \right),
\]
where the product over the empty set is considered to be $1$ if $\mathsf{d}=2$. By construction, 
the set $\vect{\mathfrak{L}}^{(\vect{m}^{\sharp}, \, \vect{m}^{\flat})}_{\vect{\kappa}}$ contains 
the non-degenerate Lissajous curves $\vect{\ell}^{(\vect{m})}_{\vect{\kappa}}$ and $\vect{\ell}^{(\vect{m})}_{\vect{\xi}}$ with 
$\vect{\xi} = (\kappa_1, \ldots ,\kappa_{\mathsf{g}'-1}, \kappa_{\mathsf{g}'} + 2 n_{\mathsf{g}'}, \kappa_{\mathsf{g}'+1}, \ldots, \kappa_{\mathsf{d}})$.
Clearly, $\vect{\ell}^{(\vect{m})}_{\vect{\xi}}([0,2\pi)) = \vect{\ell}^{(\vect{m})}_{\vect{\kappa}}([0,2\pi))$.
For the sets $\overline{\vect{\Gamma}}^{(\vect{m})}_{\vect{\kappa}}$ and $\vect{\Gamma}^{(\vect{m})}_{\vect{\kappa}}$, a stronger
statement than in Proposition \ref{201512151534} can be made.
We obtain the simple decomposition
\[
\vect{\Gamma}^{(\vect{m})}_{\vect{\kappa}} = \vectGammacircvect{m} \cup \{(0, \ldots, 0, m_{\mathsf{d}})\}, \qquad 
\overline{\vect{\Gamma}}^{(\vect{m})}_{\vect{\kappa}} = \vectGammacircvect{m} \cup \mathfrak{S}^{(\vect{m})}(\vect{0}). 
\]
An explicit trivariate example is illustrated in Figure \ref{fig:LCpoints542}.



\end{document}